\documentclass[11pt,reqno]{amsart}
\usepackage{amsthm,amsfonts,amsmath, amscd, amssymb, hyperref,amssymb}
\usepackage{fullpage}
\usepackage[pdftex]{graphicx}
\usepackage{enumerate}
\usepackage{cite}
\usepackage[usenames,dvipsnames]{color}
\usepackage{subcaption}
\usepackage{graphicx}
\usepackage{hyphenat}
\usepackage[normalem]{ulem}

\allowdisplaybreaks
\newtheorem{thm}{THEOREM}[section]
\newtheorem{lem}[thm]{LEMMA}
\newtheorem{cor}[thm]{COROLLARY}
\newtheorem{prop}[thm]{PROPOSITION}
\newtheorem{as}[thm]{ASSUMPTION}

\theoremstyle{definition}
\newtheorem{defi}[thm]{DEFINITION}
\newtheorem{remark}[thm]{REMARK}

\newcommand{\be}{\begin{equation}}
\newcommand{\ee}{\end{equation}}
\newcommand{\bes}{\begin{equation*}}
\newcommand{\ees}{\end{equation*}}

\newcommand{\mt}[1]{\mathrm{#1}}
\newcommand{\norm}[1]{\left\Vert#1\right\Vert}
\newcommand{\abs}[1]{\left|#1\right|}

\newcommand{\R}{{\mathord{\mathbb R}}}
\newcommand{\Rd}{{\mathord{\mathbb R}^d}}

\newcommand{\N}{{\mathord{\mathbb N}}}
\newcommand{\loc}{{\rm loc}}
\newcommand{\supp}{{\mathop{\rm supp\ }}}
\newcommand{\id}{\mathop{\rm id}}
\newcommand{\grad}{\nabla}
\newcommand{\wsto}{\stackrel{*}{\rightharpoonup}}

\newcommand{\la}{\left\langle}
\newcommand{\ra}{\right\rangle}
\newcommand{\F}{\mathcal{F}}
\newcommand{\G}{\mathcal{G}}

\newcommand{\ird}{\int_{\mathord{\mathbb R}^d}}
\newcommand{\irdrd}{\int_{\mathord{\mathbb R}^d \times \mathord{\mathbb R}^d}}
\newcommand{\E}{\mathcal{E}}

\renewcommand{\:}{\colon}

\def\P{{\mathcal P}}

\def\epsilon{\varepsilon}
\def\e{\varepsilon}
\def\F{\mathcal{F}}

\DeclareMathOperator{\Tan}{Tan}

\title{A Blob Method For Diffusion}
\author{Jos\'e Antonio Carrillo}
\address{Department of Mathematics, Imperial College London, South Kensington Campus, London SW7 2AZ, UK}
\email{carrillo@imperial.ac.uk}
\author{Katy Craig}
\address{Department of Mathematics, University of California, Santa Barbara, CA 93117, USA}
\email{kcraig@math.ucsb.edu}
\author{Francesco S. Patacchini}
\address{Department of Mathematical Sciences, Carnegie Mellon University, Pittsburgh, PA 15203, USA}
\email{fpatacch@math.cmu.edu}

\thanks{JAC was partially supported by the Royal Society via a Wolfson Research Merit Award and by EPSRC grant number EP/P031587/1. KC was supported by a UC President's Postdoctoral Fellowship and NSF DMS-1401867. FSP was partially supported by a 2015 Doris Chen mobility award through Imperial College London, and also acknowledges a 2015 SIAM student travel award. The authors were supported by NSF RNMS (KI-Net) grant \#11-07444, and acknowledge the CNA at CMU for their kind support of a visit to Pittsburgh in the final stages of this work. This work used XSEDE Comet at the San Diego Supercomputer Center through allocation ddp287, which is supported by NSF ACI-1548562.}

\subjclass[2010]{35Q35 35Q82 65M12 82C22; \newline \indent
\emph{Key words and phrases.} Particle method, porous medium equation, Wasserstein gradient flow, vortex blob method}

\begin{document}
 
\begin{abstract}
As a counterpoint to classical stochastic particle methods for diffusion, we develop a deterministic particle method for linear and nonlinear diffusion. At first glance, deterministic particle methods are incompatible with diffusive partial differential equations since initial data given by sums of Dirac masses would be smoothed instantaneously: particles do not remain particles. Inspired by classical vortex blob methods, we introduce a nonlocal regularization of our velocity field that ensures particles do remain particles and apply this to develop a numerical blob method for a range of diffusive partial differential equations of Wasserstein gradient flow type, including the heat equation, the porous medium equation, the Fokker--Planck equation, and the Keller--Segel equation and its variants. Our choice of regularization is guided by the Wasserstein gradient flow structure, and the corresponding energy has a novel form, combining aspects of the well-known interaction and potential energies. 
In the presence of a confining drift or interaction potential, we prove that minimizers of the regularized energy exist and, as the regularization is removed, converge to the minimizers of the unregularized energy. We then restrict our attention to nonlinear diffusion of porous medium type with at least quadratic exponent. Under sufficient regularity assumptions, we prove that gradient flows of the regularized porous medium energies converge to solutions of the porous medium equation. As a corollary, we obtain convergence of our numerical blob method. We conclude by considering a range of numerical examples to demonstrate our method's rate of convergence to exact solutions and to illustrate key qualitative properties preserved by the method, including asymptotic behavior of the Fokker--Planck equation and critical mass of the two-dimensional Keller--Segel equation.
\end{abstract}

\maketitle

\vspace{-1cm}
\section{Introduction}
For a range of partial differential equations, from the heat and porous medium equations to the Fokker--Planck and Keller--Segel equations, solutions can be characterized as \emph{gradient flows} with respect to the \emph{quadratic Wasserstein distance}.
In particular,  solutions of the equation 
\begin{align} \label{W2PDE1}
\partial_t \rho = \underbrace{\nabla \cdot (\nabla V \rho)}_\text{drift} + \underbrace{\nabla \cdot ((\nabla W*\rho) \rho)}_{\text{interaction}} + \underbrace{\Delta \rho^m \vphantom{\nabla \cdot(\grad v \rho)}}_{\text{diffusion}} \qquad V\: \Rd \to \R, \quad W\: \Rd \to \R , \quad m \geq 1,
\end{align}
where $\rho$ is a curve in the space of probability measures, are formally \emph{Wasserstein gradient flows} of the energy
\begin{align} \label{W2energy1}
	\E(\rho) = \int V\,d\rho  + \frac12 \int (W*\rho) \,d\rho +  \F^m(\rho), \quad \F^m(\rho) = 
 \begin{cases}
 \displaystyle\int \rho \log(\rho)\,d\mathcal{L}^d &\text{ for } m=1, \rho \ll \mathcal{L}^d , \\[2mm]
  \displaystyle \int  \frac{\rho^m}{m-1}\,d\mathcal{L}^d &\text{ for } m >1, \rho \ll \mathcal{L}^d \\[2mm]
 +\infty &\text{ otherwise,}
  \end{cases}
\end{align}
where $\mathcal{L}^d$ is $d$-dimensional Lebesgue measure.
This implies that solutions $\rho(t,x)$ of (\ref{W2PDE1}) satisfy
\[ \partial_t \rho = - \grad_{W_2} \E(\rho) , \]
for a generalized notion of gradient $\grad_{W_2}$, which is formally given by 
\bes
	\grad_{W_2} E(\rho) = - \grad \cdot \left(\rho \grad \frac{\delta \E}{\delta \rho} \right),
\ees 
where $\delta \E/\delta \rho$ is the first variation density of $\E$ at $\rho$ (c.f. \cite{Villani,cmcv-03,AGS,cmcv-06}).

Over the past twenty years, the Wasserstein gradient flow perspective has led to several new theoretical results, including asymptotic behavior of solutions of nonlinear diffusion and aggregation-diffusion equations \cite{Otto,cmcv-03,cmcv-06}, stability of steady states of the Keller--Segel equation \cite{BCL,BlanchetCarlenCarrillo}, and uniqueness of bounded solutions \cite{CarrilloLisiniMainini}. The underlying gradient flow theory has been well developed in the case of convex (or, more generally, semiconvex) energies \cite{JKO,Villani,Villani2,AGS,AmGi,AmSa,5person,Santambrogio}, and more recently, is being extended to consider energies with more general moduli of convexity \cite{CraigNonconvex,cmcv-06,CarrilloLisiniMainini, AmbrosioSerfaty}.

Wasserstein gradient flow theory has also inspired new numerical methods, with a common goal of maintaining the gradient flow structure at the discrete level, albeit in different ways. Recent work has considered finite volume, finite element, and discontinuous Galerkin methods \cite{Filbet,BCW,CCH,ZCS,LiuWangZhou}. Such methods are energy decreasing, positivity preserving, and mass conserving at the semidiscrete level, leading to high-order approximations.  They naturally preserve stationary states, since dissipation of the free energy provides inherent stability, and often also capture the rate of asymptotic decay. Another common strategy for preserving the gradient flow structure at the discrete level is to leverage the discrete-time variational scheme introduced by Jordan, Kinderlehrer, and Otto \cite{JKO}. 
A wide variety of strategies have been developed for this approach: 
working with different discretizations of the space of Lagrangian maps \cite{during2010gradient,MO1,MO2,MO3,JMO}, using alternative formulations of the variational structure \cite{ESG2005},  
making use of convex analysis and computational geometry to solve the optimality conditions \cite{BCMO}, and many others \cite{GT,gosse2006identification,BlanchetCalvezCarrillo,CM,CG,CRW, WW,CCWW}.

In this work, we develop a deterministic particle method for Wasserstein gradient flows. The simplest implementation of a particle method for equation (\ref{W2PDE1}), in the absence of diffusion, begins by first discretizing the initial datum $\rho_0$ as a finite sum of $N$ Dirac masses, that is,
\begin{align} \label{partinitial1}
 \rho_0 \approx \rho_0^N = \sum_{i = 1}^N \delta_{x_i} m_i, \qquad x_i \in \Rd, \quad m_i \geq 0 , 
 \end{align}
where $\delta_{x_i}$ is a Dirac mass centered at $x_i \in \Rd$.
Without diffusion and provided sufficient regularity of $V$ and $W$, the solution $\rho^N$ of (\ref{W2PDE1}) with initial datum $\rho^N_0$ remains a sum of Dirac masses at all times $t$, so that
 \begin{align} \label{partsol1} 
 	\rho^N(t) =  \sum_{i = 1}^N \delta_{x_i(t)} m_i, 
\end{align}
 and solving the partial differential equation (\ref{W2PDE1}) reduces to solving a system of ordinary differential equations for the locations of the Dirac masses,
\begin{align} \label{particle odes1}
\dot{x}_i = - \grad V(x_i) - \sum_{j =1}^N \grad W(x_i - x_j) m_j, \quad i\in\{1,\dots,N\}.
\end{align}
The particle solution $\rho^N(t)$ is the Wasserstein gradient flow of the energy (\ref{W2energy1}) with initial data $\rho_0^N$, so in particular the energy decreases in time along this spatially discrete solution. The ODE system \eqref{particle odes1} can be solved 
using range of fast numerical methods, and the resulting discretized solution $\rho^N(t)$ can be interpolated in a variety of ways for graphical visualization.

This simple particle method converges to exact solutions of equation (\ref{W2PDE1}) under suitable assumptions on $V$ and $W$, as has been shown in the rigorous derivation of this equation as  the mean-field limit of particle systems \cite{CarrilloChoiHauray,5person,Jabin}. Recent work, aimed at capturing competing effects in repulsive-attractive systems and developing methods with higher-order accuracy, has considered enhancements of  standard particle methods inspired by techniques from classical fluid dynamics, including \emph{vortex blob methods} and \emph{linearly transformed particle methods}  \cite{GHL,Hauray,CB,CCCC}. Bertozzi and the second author's blob method for the aggregation equation obtained improved rates of convergence to exact solutions for singular interaction potentials $W$ by convolving $W$ with a mollifier $\varphi_\e$. In terms of the Wasserstein gradient flow perspective this translates into regularizing the interaction energy $(1/2) \int (W*\rho) \,d\rho$ as $(1/2) \int (W*\varphi_\e*\rho)\,d\rho$.

   When diffusion is present in equation \eqref{W2PDE1}, the fundamental assumption underlying basic particle methods breaks down: particles do not remain particles, or in other words,  the solution of (\ref{W2PDE1}) with initial datum (\ref{partinitial1}) is not of the form (\ref{partsol1}). A natural way to circumvent this difficulty, at least in the case of linear diffusion ($m=1$), is to consider a stochastic particle method, in which the particles evolve via Brownian motion. Such approaches were originally developed in the classical fluids case \cite{CottetKoumoutsakos}, and several recent works have considered analogous methods for equations of Wasserstein gradient flow type, including the Keller--Segel equation \cite{Jabin,JW,Liu1,Liu2}. The main practical disadvantage of these stochastic methods is that their results must be averaged over a large number of runs to compensate for the inherent randomness of the approximation. Furthermore, to the authors' knowledge, such methods have not been extended to the case of degenerate diffusion $m>1$.

Alternatives to stochastic methods have been explored for similar equations, motivated by particle-in-cell methods in classical fluid, kinetic, and plasma physics equations. These alternatives proceed by introducing a suitable regularization of the flux of the continuity equation \cite{Russo2,CR}. Degond and Mustieles considered the case of linear diffusion ($m=1$) by interpreting the Laplacian as induced by a velocity field $v$, $\Delta \rho = \grad \cdot (v \rho)$, $v = \grad \rho/\rho$, and regularizing the numerator and denominator separately by convolution with a mollifier \cite{DegondMustieles,Russo}. For this regularized equation, particles do remain particles, and a standard particle method can be applied. Well-posedness of the resulting system of ordinary differential equations and a priori estimates relevant to the method were studied by Lacombe and Mas-Gallic \cite{LacombeMasGallic} and extended to the case of the porous medium equation by Oelschl{\"a}ger and  Lions and Mas-Gallic \cite{oelschlager1990large,LionsMasGallic, MGallic}. In the case $m=2$  on bounded domains, Lions and Mas-Gallic succeeded in showing that solutions to the regularized equation converge to solutions of the unregularized equation, as long as the initial data has uniformly bounded entropy. Unfortunately, this assumption fails to hold when the initial datum is given by a particle approximation  (\ref{partinitial1}), and consequently Lions and Mas-Gallic's result doesn't guarantee convergence of the particle method. Oelschl{\"a}ger \cite{oelschlager1990large}, on the other hand, succeeded in proving convergence of the deterministic particle method, as long as the corresponding solution of the porous medium equation is smooth and positive.
 An alternative approach, now known as the particle strength exchange method, incorporates instead the effects of diffusion by allowing the weights of the particles $m_i$ to vary in time. Degond and Mas-Gallic developed such a method for linear diffusion ($m=1$) and proved second order convergence with respect to the initial particle spacing \cite{DegondMasGallic1, DegondMasGallic2}. The main disadvantage of these existing deterministic particle methods is that, with the exception of Lions and MasGallic's work when $m=2$, they do not preserve the gradient flow structure \cite{LionsMasGallic}. Other approaches that respect the method's variational structure have been recently proposed in one dimension by approximating particles by non-overlapping blobs \cite{CHPW,CPSW}. For further background on deterministic particle methods, we refer the reader to Chertock's comprehensive review \cite{Chertock}.

The goal of the present paper is to introduce a new deterministic particle method for equations of the form \eqref{W2PDE1}, with linear and nonlinear diffusion ($m \geq 1$), that respects the problem's underlying gradient flow structure and naturally extends to all dimensions. In contrast to the above described work, which began by regularizing the flux of the continuity equation, we follow an approach analogous to Bertozzi and the second author's blob method for the aggregation equation and  regularize the associated internal energy $\F$. For a mollifier $\varphi_\e(x) = \varphi(x/\epsilon)/\epsilon^d$, $x \in \Rd$, $\e>0$, we define
\begin{align} \label{regularizedentropy} \F^m_\e(\rho) =  \begin{cases}
	\displaystyle \int \log(\varphi_\e*\rho)\,d \rho &\text{ for } m=1  , \\[2mm]
	\displaystyle \int \frac{(\varphi_\e*\rho)^{m-1}}{m-1}\,d \rho &\text{ for } m >1. \end{cases}
\end{align}
For more general nonlinear diffusion, we define 
\begin{align} \label{regularizedentropygen} 
	\F_\e(\rho) =  \int F(\varphi_\e*\rho)\, d\rho, \qquad F\: (0,\infty) \to \R. 
\end{align}

As $\e \to 0$, we prove that the regularized internal energies $\F^m_\e$ $\Gamma$-converge to the unregularized energies $\F^m$ for all $m \geq 1$; see Theorem \ref{Gamma convergence theorem2}. In the presence of a confining drift or interaction potential, so that minimizers exist, we also show that minimizers converge to minimizers; see Theorem \ref{minimizers converge theorem}. For $m \geq 2$ and semiconvex potentials $V,W \in C^2(\Rd)$, we show that the gradient flows of the regularized energies $\E_\e^m$ are well-posed and are characterized by solutions to the partial differential equation  
\begin{align} \label{W2PDE1e}
\partial_t \rho = \nabla \cdot ((\nabla V + \nabla W*\rho) \rho) + \grad \cdot \left[ \rho \left( \grad \varphi_\e* \left((\varphi_\e*\rho)^{m-2} \rho \right) + (\varphi_\e* \rho)^{m-2} (\grad \varphi_\e * \rho)\right) \right] .
\end{align}
Under sufficient regularity conditions, we prove that solutions of the regularized gradient flows converge to solutions of equation (\ref{W2PDE1}); see Theorem \ref{Gamma GF theorem}. When $m=2$ and the initial datum has bounded entropy, we show that these regularity conditions automatically hold, thus generalizing Lions and Mas-Gallic's result for the porous medium equation on bounded domains to the full equation \eqref{W2PDE1} on all of $\Rd$; see Corollary \ref{m2 theorem} and \cite[Theorem 2]{LionsMasGallic}.

For this regularized equation \eqref{W2PDE1e}, particles do remain particles; see Corollary \ref{particles well posed}. Consequently, our numerical blob method for diffusion  consists of taking a particle approximation for \eqref{W2PDE1e}. We conclude by showing that, under sufficient regularity conditions, our blob method's particle solutions converge to exact solutions of \eqref{W2PDE1}; see Theorem \ref{numerics convergence}. We then give several numerical examples illustrating the rate of convergence of our method and its qualitative properties.

A key advantage of our approach is that, by regularizing the energy functional and not the flux, we preserve the problem's gradient flow structure. Still, at first glance, our regularization of the energy (\ref{regularizedentropy}) may seem less natural than other potential choices. For example, one could instead consider the following more symmetric regularization
\bes
	\mathcal{U}^m_\e(\rho) := \mathcal{F}^m(\varphi_\e* \rho) =  \begin{cases}
\displaystyle \int (\varphi_\e* \rho) \, \log(\varphi_\e*\rho)\,d\mathcal{L}^d &\text{ for } m=1  , \\[2mm]
\displaystyle \int  \frac{(\varphi_\e*\rho)^{m}}{m-1}\,d\mathcal{L}^d &\text{ for } m >1,
  \end{cases}\ees
 for more general nonlinear diffusion, 
\[ \mathcal{U}_\e(\rho) = \int U(\varphi_\e*\rho)\,d\mathcal{L}^d, \qquad U\: [0,\infty) \to \R . \]
Although studying the above regularization is not without interest, we focus our attention on the regularization in \eqref{regularizedentropy} and \eqref{regularizedentropygen} for numerical reasons. Indeed, computing the first variation density of $\mathcal{U}_\e$ gives
\bes
	\frac{\delta \mathcal{U}_\e}{\delta \rho} = \varphi_\e * (U' \circ (\varphi_\e*\rho)),
\ees
as compared to 
\bes
	\frac{\delta \mathcal{F}_\e}{\delta \rho} = \varphi_\e * (F' \circ (\varphi_\e*\rho)\rho) + F\circ(\varphi_\e *\rho)
\ees
for $\F_\e$. In the first case, one can see that replacing $\rho$ by a sum of Dirac masses still requires the computation of an integral convolution with $\varphi_\e$. Indeed, if $\rho=  \sum_{i=1}^N \delta_{x_i} m_i$, where $(x_i)_{i=1}^N$ are $N$ particles in $\R^d$ with masses $m_i > 0$, then, for all $x\in\R^d$,
\bes
	\frac{\delta \mathcal{U}_\e}{\delta \rho}(x) = \varphi_\e * \left[ U' \left(   \sum_{i=1}^N \varphi_\e(x-x_i) m_i \right) \right] = \ird \varphi_\e(x-y) \left[ U' \left(   \sum_{i=1}^N \varphi_\e(y-x_i) m_i \right) \right] \,d y,
\ees
which does not allow for a complete discretization of the integrals. On the contrary, in the second case, all convolutions involve $\rho$, so a similar computation (as it can be found in the proof of Corollary \ref{particles well posed}) shows that they  reduce to finite sums, which are numerically less costly.

Another advantage of our approach, in the $m=2$ case, is that our regularization of the energy can naturally be interpreted as an approximation of the porous medium equation by a very localized nonlocal interaction potential. In this way, our proof of the convergence of the associated particle method provides a theoretical underpinning to approximations of this kind in the computational math and swarming literature \cite{LeverentzTopazBernoff, Klar}. Further advantages our blob method include the ease with which it may be combined with particle methods for interaction and drift potentials, its simplicity in any dimension, and the good numerical performance we observe for a wide choice of interaction and drift potentials.

Our paper is organized as follows. In Section \ref{preliminaries section}, we collect preliminary results concerning the regularization of measures via convolution with a mollifier, including a mollifier exchange lemma (Lemma \ref{move mollifier prop}), and relevant background on Wasserstein gradient flow and weak convergence of measures. In Section \ref{energies section}, we prove several results on the general regularized energies (\ref{regularizedentropygen}), which are of a novel form from the perspective of Wasserstein gradient flow theory, combining aspects of the well-known interaction and internal energies. We show that these regularized energies are semiconvex and differentiable in the Wasserstein metric and characterize their subdifferential with respect to this structure; see Propositions \ref{diff prop}--\ref{subdiffchar}. 
In Section \ref{Gamma convergence section}, we prove that $\F_\e$ $\Gamma$-converges to $\F$ as $\epsilon \to 0$ and that minimizers converge to minimizers, when in the presence of a confining drift or interaction term; see Theorems \ref{Gamma convergence theorem2} and \ref{minimizers converge theorem}.
With this $\Gamma$-convergence in hand, in Section \ref{Gamma convergence GF section} we then turn to the question of convergence of gradient flows, restricting to the case $m \geq 2$. Using the framework introduced by Sandier and Serfaty \cite{Serfaty,SaSe}, we prove that, under sufficient regularity assumptions, gradient flows of the regularized energies converge as $\e \to 0$ to gradient flows of the unregularized energy, recovering a generalization of Lions and Mas-Gallic's results when $m=2$; see Theorem \ref{Gamma GF theorem} and Corollary \ref{m2 theorem}. Finally, in Section \ref{numerics section}, we prove the convergence of our numerical blob method, under sufficient regularity assumptions, when the initial particle spacing $h$ scales with the regularization like $h = o(\e)$; see Theorem \ref{numerics convergence}.

We close with several numerical examples, in one and two dimensions, analyzing the rate of convergence to exact solutions with respect to the $2$-Wasserstein metric, $L^1$-norm, and $L^\infty$-norm and illustrating qualitative properties of the method, including asymptotic behavior of the Fokker--Planck equation and critical mass of the two-dimensional Keller--Segel equation; see Section \ref{simulations}. In particular, for the heat equation and porous medium equations ($V=W=0$, $m=1,2,3$), we observe that the $2$-Wasserstein error depends linearly on the grid spacing $h \sim N^{-1/d}$ for $m=1,2,3$, while the $L^1$-norm depends quadratically on the grid spacing for $m=1,2$ and superlinearly for $m=3$. We apply our method to study long time behavior of the nonlinear Fokker--Planck equation ($V=\abs{\cdot}^2/2$, $W = 0$, $m=2$), showing that the blob method accurately captures convergence to the unique steady state. Finally, we conduct a detailed numerical study of equations of Keller--Segel type, including a one-dimensional variant ($V=0, W = 2\chi \log\abs{\cdot}, \chi>0, m=1,2$) and the original two-dimensional equation ($V=0$, $W = \Delta^{-1}$, $m=1$). The one-dimensional equation has a critical mass $1$, and the two-dimensional equation has critical mass $8 \pi$, at which point the concentration effects from the nonlocal interaction term balance with linear diffusion ($m=1$) \cite{DP,BlanchetDolbeaultPerthame}. We show that the same notion of criticality is present in our numerical solutions and demonstrate convergence of the critical mass as the grid spacing $h$ and regularization $\e$ are refined. 

There are several directions for future work. Our convergence theorem for $m \geq 2$ requires additional regularity assumptions, which we are only able to remove in the case $m=2$ when the initial data has bounded entropy.  In the case of $m>2$ or more general initial data, it remains an open question how to control certain nonlocal norms of the regularized energies, which play an important role in our convergence result; see Theorem \ref{Gamma GF theorem}. Formally, we expect these to behave as approximations of the $BV$-norm of $\rho^m$, which should remain bounded by the gradient flow structure; see equations \eqref{BV heuristic 1} and \eqref{BV heuristic 2}.  When $1\leq m<2$, it is not clear how to use  these nonlocal norms to get the desired convergence result or whether an entirely different approach is needed. Perhaps related to these questions is the fact that our estimate on the semiconvexity of the regularized energies \eqref{regularizedentropy} deteriorates as $\e \to 0$, while we expect that the semiconvexity should not deteriorate along smooth geodesics; see Proposition \ref{prop:conv}. Finally, while our results show convergence of the blob method for diffusive Wasserstein gradient flows, they do not quantify the rate of convergence in terms of $h$ and $\e$. In particular, a theoretical result on the optimal scaling relation between $h$ and $\e$ remains open, though we observe good numerical performance for $\e = h^{1-p}$, $0< p \ll 1$. In a less technical direction, we foresee a use of the presented ideas in conjunction with splitting schemes for certain nonlinear kinetic equations \cite{CarGang,Agueh}, as well as in the fluids \cite{Hauray}, since our numerical results demonstrate comparable rates of convergence to the particle strength exchange method, which has already gained attention in these contexts \cite{DegondMustieles}.




\section{Preliminaries} \label{preliminaries section}

\subsection{Basic notation}

For any $r>0$ and $x \in \Rd$ we denote the open ball of center $x$ and radius $r$ by $B_r(x)$. Given a set $S \subset \R^d$, we write $1_{S}\colon \R^d \to \{0,1\}$ for the indicator function of $S$, i.e., $1_S(x) = 1$ for $x \in S$ and $1_S(x) = 0$ otherwise. We say a function $A:\Rd \to \R$ has \emph{at most quadratic growth} if there exist $c_0, c_1 >0$ so that $|A(x)| \leq c_0 + c_1|x|^2$ for all $x \in \Rd$.

Let $\P(\Rd)$ denote the set of Borel probability measures on $\Rd$, and for, any $p\in\N$, $\P_p(\Rd)$ denotes  elements of $\P(\Rd)$ with finite $p$th moment, $M_p(\R^d) := \textstyle \ird |x|^p\,d\mu(x) < +\infty$. We write $\mathcal{L}^d$ for the $d$-dimensional Lebesgue measure, and for given $\mu\in\P(\Rd)$, we write $\mu \ll \mathcal{L}^d$ if $\mu$ is absolutely continuous with respect to the Lebesgue measure. Often we use the same symbol for both a probability measure and its Lebesgue density, whenever the latter exists.  We let $L^p(\mu;\Rd)$ denote the Lebesgue space of functions with $p$th power integrable against $\mu$. 

Given $\sigma$ a finite, signed Borel measure on $\R^d$, we denote its variation by $|\sigma|$. For a Borel set $E \subset \Rd$ we write $\sigma(E)$ for the $\sigma$-measure of set $E$.
For a Borel map $T \: \R^d \to \R^d$ and $\mu \in \P(\R^d)$, we write $T_\# \mu$ for the push-forward of $\mu$ through $T$. We let $\id\: \Rd \to \Rd$ denote the identity map on $\R^d$ and define $(\id,T) \: \R^d \to \R^d \times \R^d$ by $(\id,T)(x) = (x,T(x))$ for all $x \in \Rd$.
For a sequence $(\mu_n)_n \subset \P(\Rd)$ and some $\mu\in\P(\Rd)$, we write $\mu_n \wsto \mu$ if $(\mu_n)_n$ converges to $\mu$ in the weak-$^*$ topology of probability measures, i.e., in the duality with bounded continuous functions.

\subsection{Convolution of measures}
A key aspect of our approach is the regularization of the energy (\ref{W2energy1}) via convolution with a mollifier. In this section, we collect some elementary results on the convolution of probability measures, including a mollifier exchange lemma, Lemma \ref{move mollifier prop}.

For any  $\mu \in\P(\Rd)$ and measurable function $\phi$, the convolution of $\phi$ with $\mu$ is given by
\bes
	\phi*\mu(x) = \int_{\Rd} \phi(x-y) \,d\mu(y) \quad \mbox{for all $x\in\Rd$} ,
\ees
whenever the integral converges.
We consider mollifiers $\varphi$ satisfying the following assumption.

\begin{as}[mollifier] \label{mollifierAssumption}
Let $\varphi = \zeta * \zeta$, where $\zeta \in C^2(\Rd;[0,\infty))$  is even, $\|\zeta\|_{L^1(\Rd)} =1$, and
\bes
	\zeta(x) \leq C_\zeta |x|^{-q} , |\grad \zeta (x) | \leq C'_\zeta |x|^{-q'} \quad \mbox{for some $C_\zeta, C_\zeta' >0$ and $q >d+1, \ q' >d$}.
\ees
\end{as}

This assumption is satisfied by both Gaussians and smooth functions with compact support. Assumption \ref{mollifierAssumption} also ensures that $\varphi$ has finite first moment. For any $\epsilon >0$, we write 
\bes
	\varphi_\e =  \e^{-d} \varphi (\cdot/\e) \quad \mbox{and} \quad \zeta_\e = \e^{-d} \zeta(\cdot/\e) .
\ees
Throughout, we use the fact that the definition of convolution allows us to move mollifiers from the measure to the integrand. In particular, for any $\phi$ bounded below and $\psi \in L^1(\Rd)$  even,
\bes
	\ird \phi \,d(\psi*\mu) = \ird \phi*\psi \,d\mu.
\ees

Likewise, the technical assumption that $\varphi = \zeta * \zeta$, and therefore that $\varphi_\e = \zeta_\e * \zeta_\e$, allows us to regularize integrands involving the mollifier $\varphi_\e$; indeed, the following lemma provides sufficient conditions for moving functions in and out convolutions with mollifiers within integrals. (See also \cite{LionsMasGallic} for a similar result.) This is an essential component in the proofs of both  main results, Theorems \ref{Gamma convergence theorem2} and \ref{Gamma GF theorem}, on the the $\Gamma$-convergence of the regularized energies and the convergence of the corresponding gradient flows. See  Appendix \ref{appendix preliminaries} for the proof of this lemma.
\begin{lem}[mollifier exchange lemma] \label{move mollifier prop}
Let $f\colon \R^d \to \R$ be Lipschitz continuous with  Lipschitz constant $L_f>0$, and let $\sigma$ and $\nu$ be finite, signed Borel measures on $\R^d$. There is $p = p(q,d)>0$ so that 
\[ 
	\left| \int \zeta_\e *(f\nu) \,d\sigma - \int (\zeta_\e *\nu)f \,d\sigma \right|  \leq \e^{p} L_f \left( \int (\zeta_\e*|\nu|)\,d|\sigma| + C_\zeta |\sigma|(\Rd) |\nu|(\R^d) \right)  \  \text{ for all } \e >0.
\]
\end{lem}

We conclude this section with a lemma stating that if a sequence of measures converges in the weak-$^*$ topology of $\P(\Rd)$, then the mollified sequence converges to the same limit. We refer the reader to Appendix \ref{appendix preliminaries} for the proof.

\begin{lem} \label{weakst convergence mollified sequence}
	Let $\mu_\e$ be a sequence in $\P(\R^d)$ such that $\mu_\e \wsto \mu$ as $\e\to0$ for some $\mu\in\P(\R^d)$. Then $\varphi_\e *\mu_\e \wsto \mu$.
\end{lem}

\subsection{Optimal transport, Wasserstein metric, and gradient flows}

We now describe basic facts about optimal transport, including the Wasserstein metric and associated gradient flows. (See also \cite{AGS,Villani,Santambrogio,Villani2,AmGi,AmSa} for further background and more details on  the definitions and remarks found in this section.)

For $\mu,\nu\in\P(\R^d)$, we denote the set of transport plans from $\mu$ to $\nu$ by
\bes
	\Gamma(\mu,\nu) := \{\gamma \in\P(\R^d\times\R^d) \mid {\pi^1}_\# \gamma = \mu,\, {\pi^2}_\# \gamma = \nu\},
\ees
where $\pi^1,\pi^2\colon \R^d \times \R^d \to \R^d$ are the projections of $\R^d\times \R^d$ onto the first and second copy of $\R^d$, respectively.
The \emph{Wasserstein distance} $W_2(\mu,\nu)$ between two probability measures $\mu,\nu\in\P_2(\R^d)$ is given by
\be\label{eq:wass-p}
	W_2(\mu,\nu) = \min_{\gamma \in \Gamma(\mu,\nu)}  \left( \int_{\R^d\times \R^d} |x-y|^2 d \gamma(x,y) \right)^{1/2} ,
\ee
and a transport plan  $\gamma_\mathrm{o}$ is \emph{optimal} if it attains the minimum in \eqref{eq:wass-p}. We denote the set of optimal transport plans by $\Gamma_\mathrm{o}(\mu,\nu)$. If $\mu$ is absolutely continuous with respect to the Lebesgue measure, then there is a unique optimal transport plan $\gamma_\mathrm{o}$, and
\bes
	\gamma_\mathrm{o} = (\id,T_\mathrm{o})_\#\mu,
\ees
for a Borel measurable function $T_\mathrm{o}\:\Rd\to\Rd$. $T_\mathrm{o}$ is unique up to sets of $\mu$-measure zero and is known as the \emph{optimal transport map} from $\mu$ to $\nu$. 
Convergence with respect to the Wasserstein metric is stronger than weak-$^*$ convergence. In particular,  if $(\mu_n)_n \subset \P_2(\R^d)$ and $\mu \in \P_2(\R^d)$, then 
\bes
	\mbox{$W_2(\mu_n,\mu) \to 0$ as $n\to\infty$} \iff \left(\mbox{$\mu_n \wsto \mu$ and $M_2(\mu_n) \to M_2(\mu)$ as $n\to\infty$}\right).
\ees

In order to define Wasserstein gradient flows, we will require the following notion of regularity in time with respect to the Wasserstein metric.

\begin{defi}[absolutely continuous]\label{defi:ac-curve}
$ \mu \in AC^2_\loc((0,\infty);P_2(\Rd))$ if there is $f\in L^2_\loc((0,\infty))$ so that 
\bes
	W_2(\mu(t),\mu(s)) \leq \int_s^t f(r)\,d r \quad \mbox{for all $t,s\in (0,\infty)$ with $s\leq t$.}
\ees
\end{defi}
Along such curves, we have a notion of metric derivative.
\begin{defi}[metric derivative]\label{defi:metric-derivative}
Given $ \mu \in AC^2_\loc((0,\infty);P_2(\Rd))$, its metric derivative is
\bes
	|\mu'|(t) := \lim_{s\to t} \frac{W_2(\mu(t),\mu(s))}{|t-s|}
\ees
\end{defi}

An important class of curves in the Wasserstein metric are the (constant speed) \emph{geodesics}. Given $\mu_0,\mu_1 \in \P_2(\R^d)$, geodesics connecting $\mu_0$ to $\mu_1$ are of the form
\bes
	\mu_\alpha = ((1-\alpha) \pi^1 + \alpha \pi^2)_\# \gamma_\mathrm{o} \quad \mbox{for $\alpha\in[0,1]$, $\gamma_\mathrm{o}\in \Gamma_\mathrm{o}(\mu,\nu)$}.
\ees
If $\gamma_\mathrm{o}$ is induced by a map $T_\mathrm{o}$, then
\bes
	\mu_\alpha = ((1-\alpha)\text{id} + \alpha T_\mathrm{o})_\# \mu_0.
\ees
More generally, given $\mu_1,\mu_2,\mu_3\in \P_2(\R^d)$, a \emph{generalized} geodesic connecting $\mu_2$ to $\mu_3$ with base $\mu_1$ is given by
\begin{align} \label{eq:gen-geodesic}
	\mu_\alpha^{2\to3} = \left((1-\alpha)\pi^2+\alpha\pi^3\right)_\# \gamma \quad &\mbox{for }\alpha \in[0,1] \text{ and }\gamma \in \P(\R^d \times \R^d  \times  \R^d) \\ & \text{ such that } {\pi^{1,2}}_\#\gamma\in\Gamma_\mathrm{o}(\mu_1,\mu_2) \text{ and }{\pi^{1,3}}_\#\gamma\in\Gamma_\mathrm{o}(\mu_1,\mu_3). \nonumber
\end{align}
	with $\pi^{1,i}\colon \R^d \times \R^d \times \R^d \to \R^d\times \R^d$ the projection of onto the first and $i$th copies of $\R^d$. When the base $\mu_1$ coincides with one of the endpoints $\mu_2$ or $\mu_3$, generalized geodesics are geodesics.

A key property for the uniqueness and stability of Wasserstein gradient flows is that the energies are convex, or more generally semiconvex, along generalized geodesics.
\begin{defi}[semiconvexity along generalized geodesics]\label{defi:semiconvexity-geod}
	We say a functional $\G\colon\P_2(\R^d)\to (-\infty,\infty]$ is \emph{semiconvex along generalized geodesics} if there is $\lambda \in \R$ such that 
for all $\mu_1,\mu_2,\mu_3 \in \P_2(\R^d)$ there exists a generalized geodesic connecting $\mu_2$ to $\mu_3$ with base $\mu_1$ such that
\bes
	\G(\mu_\alpha^{2\to3}) \leq (1-\alpha)\G(\mu_2) + \alpha\G(\mu_3) - \frac{\lambda (1-\alpha)\alpha}{2} W_{2,\gamma}^2(\mu_2,\mu_3) \quad \mbox{for all $\alpha\in[0,1]$},
\ees
where 
\bes
	W_{2,\gamma}^2(\mu_2,\mu_3) = \int_{\Rd\times\Rd\times\Rd} |y-z|^2 \,d\gamma(x,y,z).
\ees
\end{defi}

For any subset $X\subset \P(\Rd)$ and functional $\G \colon X \to (-\infty,\infty]$, we denote the \emph{domain} of $\G$ by $D(\G) = \{ \mu \in X \mid \G(\mu) < +\infty\}$, and we say that $\G$ is \emph{proper} if $D(\G) \neq \emptyset$. 
As soon as a functional is proper and  lower semicontinuous with respect to the weak-* topology, we may define its \emph{subdifferential}; see \cite[Definition 10.3.1 and Equation 10.3.12]{AGS}. Following the approach in \cite{5person}, the notion of subdifferential we use in this paper is, in fact, the following reduced one.

\begin{defi}[subdifferential] \label{subdiffdef}
	Given $\G:\P_2(\Rd) \to (-\infty,\infty]$ proper  and lower semicontinuous, $\mu \in D(\G)$, and $\xi:\Rd \to \Rd$ with $\xi \in L^2(\mu;\Rd)$, then $\xi$ belongs to the \emph{subdifferential} of $\G$ at $\mu$, written $\xi \in \partial \G(\mu)$, if as $\nu \xrightarrow{W_2} \mu$,
\[ 
\G(\nu) - \G(\mu) \geq \inf_{\gamma \in \Gamma_0(\mu,\nu)} \int_{\Rd \times \Rd} \la \xi(x),y-x \ra d\gamma(x,y) + o(W_2(\mu,\nu)) . 
\]
\end{defi}

 The Wasserstein metric is formally Riemannian, and we may define the tangent space as follows.
\begin{defi}
	Let $\mu \in \P_2(\R^d)$. The \emph{tangent space} at $\mu$ is
\bes
	\Tan_\mu \P_2(\R^d) = \overline{\left\{ \grad\phi \mid \phi \in C_\mathrm{c}^\infty(\R^d) \right\}},
\ees
where the closure is taken in $L^2(\mu;\R^d)$.
\end{defi}

We now turn to the definition of a gradient flow in the Wasserstein metric (c.f. \cite[Proposition 8.3.1, Definition 11.1.1]{AGS}).
\begin{defi}[gradient flow] \label{gradientflowdef}
	Suppose $\G\colon \P_2(\Rd) \to \R \cup \{+\infty\}$ is proper and lower semicontinuous. A curve $\mu \in AC^2_\loc((0,+\infty); \P_2(\Rd))$ is a \emph{gradient flow of $\G$} if there exists a velocity vector field $v\colon (0,\infty)\times \R^d \to \R^d$ with $-v(t) \in \partial \G(\mu(t)) \cap \Tan_{\mu(t)}\P_2(\R^d) $ for almost every $t >0$ such that $\mu$ is a weak solution of the continuity equation
\[ 
	\partial_t \mu(t,x) + \grad \cdot (v(t,x) \mu(t,x)) = 0 ;
\]
i.e., $\mu$ is a solution to the continuity equation in duality with $C_\mathrm{c}^\infty(\Rd)$.
\end{defi}

We close this section with the following definition of the Wasserstein local slope.
\begin{defi}[local slope]
	Given $\G \colon \P_2(\Rd)  \to (-\infty,\infty]$, its \emph{local slope} is  
\bes
	|\partial \G|(\mu) = \limsup_{\mu \to \nu} \frac{(\G(\mu) - \G(\nu) )_+}{W_2(\mu,\nu)} \quad \mbox{for all $\mu\in D(\G)$},
\ees
where the subscript $+$ denotes the positive part.
\end{defi}

\begin{remark} \label{metric slope remark}
	When the functional $\G$ in Definition \ref{gradientflowdef} is in addition semiconvex along geodesics the local slope $|\partial \G|$ is a strong upper gradient for $\G$. In this case a gradient flow of $\G$ is characterized as being a $2$-curve of maximal slope with respect to $|\partial \G|$; see \cite[Theorem 11.1.3]{AGS}.
\end{remark}


\section{Regularized internal energies} \label{energies section}
The foundation of our blob method is the regularization of the internal energy $\F$ via convolution with a mollifier. This allows us to preserve the gradient flow structure and approximate our original partial differential equation (\ref{W2PDE1}) by a sequence of equations for which particles do remain particles.
In this section, we consider several fundamental properties of the regularized internal energies $\F_\e$, including convexity, lower semicontinuity, and differentiability. In what follows, we will suppose that our internal energies satisfy the following assumption. 
\begin{as}[internal energies]\label{as:F}
Suppose $F \in C^2(0,+\infty)$ satisfies $\lim_{s \to +\infty} F(s) = +\infty$ and  either $F$ is bounded below or $\liminf_{s\to0} F(s) / s^\beta >-\infty$ for some $\beta >-2/(d+2)$. Suppose further that $U(s) = sF(s)$ is convex, bounded below, and $\lim_{s \to 0} U(s) = 0$. 
\end{as}

Thanks to this assumption we can define the \emph{internal energy} corresponding to $F$ by
\[ \F(\rho) = \begin{cases} \int F(\rho)\, d\rho & \mbox{if $\rho \ll \mathcal{L}^d$}, \\ +\infty & \mbox{otherwise}. \end{cases}\]
If $F$ is bounded below, this is well-defined on all of $\P(\Rd)$. If $\liminf_{s\to0} F(s) / s^\beta >-\infty$ for some $\beta >-2/(d+2)$, this is well-defined on $\P_2(\Rd)$; see \cite[Example 9.3.6]{AGS}.

\begin{remark}[nondecreasing]\label{rem:Fprime-positive}
	Assumption \ref{as:F} implies that $F$ is nondecreasing. Indeed, by the convexity of $U(s)$ and the fact that $\lim_{s \to 0} sF(s) =0$,
	\bes
	sF(s) = \int_0^s U'(r) \,dr \leq sU'(s) = s^2F'(s) + sF(s) \quad \mbox{for all $s\in(0,\infty)$},
\ees
which leads to $F'(s)\geq0$ for all $s \in (0,\infty)$.
\end{remark}

Our assumption does not ensure that $\F$ is convex along Wasserstein geodesics, unless $F$ is convex.

\begin{remark}[McCann's convexity condition]
McCann's condition \cite{McCann} on the internal density $U$ for the convexity of the internal energy $\F$ can be stated on the function $F$ instead: the function $s \mapsto F(s^{-d})$ is nonincreasing and convex on $(0,\infty)$, i.e., 
\bes
	F'(s) \geq0 \quad \mbox{and} \quad (d+1) F'(s) + d s F''(s) \geq 0 \quad \mbox{for all $s \in (0,\infty)$},
\ees
which, by Remark \ref{rem:Fprime-positive}, holds when for example $F$ is convex and satisfies Assumption \ref{as:F}.
\end{remark}

We regularize the internal energies by convolution with a mollifier.

\begin{defi}[regularized internal energies] \label{energy def} Given $F\colon (0,\infty)\to\R$ satisfying Assumption \ref{as:F}, we define, for all $\mu\in \P(\R^d)$, the regularized internal energies by
\begin{align*}
	 \F_\e(\mu) =  \int F(\varphi_\e * \mu) \, d\mu \quad \mbox{for all $\e>0$}.
\end{align*}
\end{defi}
\noindent Note that, for all $\mu \in \P(\Rd)$ and $\epsilon >0$, $\F_\e(\mu) < F(\norm{\varphi_\e}_{L^\infty(\R^d)}) < \infty$.

An important class of internal energies satisfying Assumption \ref{as:F} are given by the (negative) entropy and R\'enyi entropies.
\begin{defi}
The entropy and R\'enyi entropies, and their regularizations, are given by
\[ \F^m(\rho) = \int F_m(\rho)\, d \rho , \quad \F^m_\e(\mu) = \int F_m(\varphi_\e*\mu) \, d \mu , \quad \text{ for } F_m(s) = \begin{cases} \log s &\text{ for } m=1,\\ s^{m-1}/(m-1) &\text{ for } m>1 .\end{cases} \] 
\end{defi}
Note that, as per our observation just below the definition of $\F$, the entropy  $\F^1$ is well-defined on $\P_2(\Rd)$ and the R\'enyi entropies ($\F^m, m>1$) are well-defined on all of $\P(\Rd)$. Also note that the regularized entropies ($\F^m_\e, m \geq 1, \e>0$) are well-defined on all of $\P(\Rd)$.

In order to approximate solutions of equation (\ref{W2PDE1}), we will consider combinations of the above regularized internal energies with potential and interaction energies.

\begin{defi}[regularized energies] \label{full energies}
	Let $V,W\: \Rd \to (-\infty,\infty]$  be proper and lower semicontinuous. Suppose further that $W$ is locally integrable. For all $\mu \in \P(\Rd)$ define
\begin{align*}
	\E_\e(\mu) = \int V \,d\mu  + \frac12 \int (W*\mu)\,d \mu +  \F_\e(\mu) \quad \mbox{for all $\e>0$}.
\end{align*}
When $F=F_m$ for some $m\geq1$, then we denote $\E$ by $\E^m$ and $\E_\e$ by $\E_\e^m$.
\end{defi}

The regularized internal energy in Definition \ref{energy def} incorporate a blend of interaction and internal phenomena, through the convolution with the mollifier, or potential, $\varphi_\e$ and the composition with the function $F$. To our knowledge, this is a novel form of functional on the space of probability measures. We now describe some of its basic properties: energy bounds and lower semicontinuity, when $F$ is the logarithm or a power, and differentiability, convexity and subdifferential characterization when $F$ is convex. For the existence and uniqueness of gradient flows associated to this regularized energy, see Section \ref{Gamma convergence GF section}.

\begin{remark}
	Although the regularized energy in Definition \ref{energy def} is of a novel form, it was noticed in \cite[Proposition 6.9]{PatacchiniThesis} that a previous particle method for diffusive gradient flows leads to a similar regularized internal energy after space discretization \cite{CPSW,CHPW}. The essential difference between these two methods stands in the choice of the mollifier, which, instead of satisfying \ref{mollifierAssumption}, is a very singular potential.
\end{remark}

We begin with inequalities relating the regularized internal energies to the unregularized energies. See Appendix \ref{appendix preliminaries} for the proof, which is a consequence of Jensen's inequality and a Carleman-type estimate on the lower bound of the entropy \cite[Lemma 4.1]{CPSW}.

\begin{prop} \label{relative sizes lemma}
Let $\e >0$. If $m = 1$, suppose $\mu \in \P_2(\Rd)$, and if $m>1$, suppose $\mu \in \P(\Rd)$. Then,
\begin{align} \label{relative sizes equation}
	\F^m(\mu) + C_\e \geq \F^m_\e(\mu) \geq \F^m(\zeta_\e*\mu)  &\quad \text{ for } \quad 1 \leq m \leq 2, \\
	\F^m_\e(\mu) \leq \F^m(\zeta_\e*\mu) &\quad \text{ for } \quad m \geq 2. \label{relative sizes equation 2}
\end{align}
where $C_\e= C_\e(m,\mu)  \to 0$ as $\e \to 0$.
Furthermore,  for all $\delta >0$, we have
\begin{align} \label{lower bounds}  
	\F^m_\e(\mu)  \geq  \begin{cases}- (2\pi/\delta)^{d/2} - 2\delta (M_2(\mu)+ \e^2 M_2(\zeta)) &\text{ if } m =1, \\ 0 &\text{ if } m>1 .\end{cases} 
\end{align}
\end{prop}

For all $\e >0$, the regularized entropies are lower semicontinuous with respect to weak-* convergence ($m>1$) and Wasserstein convergence ($m=1$). For $m>2$, we prove this using a theorem of Ambrosio, Gigli, and Savar\'e on the convergence of maps with respect to varying measures; see Proposition \ref{AGSthm}. For $1<m \leq 2$, this is a consequence of Jensen's inequality.  For $m=1$, we apply both Jensen's inequality and a version of Fatou's lemma for varying measures; see Lemma \ref{lem:fatou-varying}. In this case, we also require that the mollifier $\varphi$ is a Gaussian, so that we can get the bound from below required by Fatou's lemma. We refer the reader to Appendix \ref{appendix preliminaries} for the proof.

\begin{prop}[lower semicontinuity] \label{lower semicontinuity}
Let $\e >0$. Then
	\begin{enumerate}[(i)]
		\item \label{it:lsc-pme} $\F^m_\e$ is lower semicontinuous with respect to weak-$^*$ convergence in $\P(\Rd)$ for all $m>1$;
		\item \label{it:lsc-he} if $\varphi$ is a Gaussian, then $\F^1_\e$ is lower semicontinuous with respect to the quadratic Wasserstein convergence in $\P_2(\R^d)$.
	\end{enumerate}
\end{prop}

When $F$ is convex, the regularized internal energies are differentiable along generalized geodesics. The proof relies on the fact that $F$ is differentiable and $\varphi_\e \in C^2(\Rd)$, with bounded Hessian; see Appendix \ref{appendix preliminaries}.
\begin{prop}[differentiability] \label{diff prop}
Let $F$ satisfy Assumption \ref{as:F} and be convex. Given $\mu_1,\mu_2,\mu_3 \in \P_2(\Rd)$ and $\gamma\in \P_2(\Rd\times\Rd\times\Rd)$ with $\pi^i_\#\gamma = \mu_i$, let $\mu_\alpha^{2\to3} = \left((1-\alpha)\pi^2+\alpha\pi^3\right)_\# \gamma$ for $\alpha \in[0,1]$. Then
\begin{equation} \label{diff lem eqn} 
	\begin{split}
		\frac{d}{d \alpha } &\left.\F_\e(\mu_\alpha^{2\to3}) \right|_{\alpha = 0}\\
		&=\iiint \iiint F'\left(\varphi_\e*\mu_2(y) \right) \grad \varphi_\e(y-v) \cdot (z-w-(y-v)) \,d \gamma(u,v,w)\, d \gamma(x,y,z). 
	\end{split}
\end{equation}
\end{prop}

A key consequence of the preceding proposition is that the regularized energies are semiconvex along generalized geodesics, as we now show.

\begin{prop}[convexity]\label{prop:conv}
Suppose $F$ satisfies Assumption \ref{as:F} and is convex. Then $\F_\e$ is $\lambda_F$-convex along generalized geodesics, where
\be \label{eq:lambda-F}
	\lambda_F = - 2 \|D^2 \varphi_\e \|_{L^\infty(\Rd)} F'(\|\varphi_\e\|_{L^\infty(\Rd)}).
\ee
\end{prop}

\begin{proof}
	 Let $(\mu_\alpha^{2\to3})_{\alpha\in[0,1]}$ be a generalized geodesic connecting two probability measures $\mu_2,\mu_3\in\P_2(\R^d)$ with base $\mu_1\in\P_2(\R^d)$; see \eqref{eq:gen-geodesic}.
	 We have, using the above-the-tangent inequality for convex functions,
\begin{align*}
	\F_\e(\mu_3) - \F_\e(\mu_2) &= \iiint \left(F(\varphi_\e*\mu_3)(y) - F(\varphi_\e*\mu_2)(z) \right)  \,d\gamma(x,y,z)\\
	&\geq \iiint F'(\varphi_\e*\mu_2(y)) \left( \varphi_\e*\mu_3(z) - \varphi_\e*\mu_2(y) \right)\,d\gamma(x,y,z) \\
	&= \iiint \iiint F'(\varphi_\e*\mu_2(y)) \left( \varphi_\e(z-w) - \varphi_\e(y-v) \right)\,d\gamma(u,v,w)\,d\gamma(x,y,z).
\end{align*}
Therefore, by Proposition \ref{diff prop},
\begin{align*}
	\F_\e(\mu_3) &- \F_\e(\mu_2) - \frac{d}{d \alpha} \left.\F_\e(\mu_\alpha^{2\to3}) \right|_{\alpha = 0} \\
 	&\geq \iiint \iiint F'(\varphi_\e*\mu_2(y)) \\
 	&\phantom{{}={}}\times \left[ \varphi_\e(z-w) - \varphi_\e(y-v) - \grad \varphi_\e(y-v) \cdot (z-w-(y-v)) \right] d \gamma(u,v,w)\,d \gamma(x,y,z)\\
	&\geq -\frac{\|D^2 \varphi_\e\|_{L^\infty(\Rd)}}{2} \iiint \iiint F'(\varphi_\e*\mu_2(y)) |z-w - (y-v)|^2 \,d \gamma(u,v,w)\,d \gamma(x,y,z) \\
	&\geq -\frac{\|D^2 \varphi_\e\|_{L^\infty(\Rd)} F'(\|\varphi_\e\|_{L^\infty(\Rd)})}{2} \iiint \iiint |z-w - (y-v)|^2\,d \gamma(u,v,w)\,d \gamma(x,y,z) \\
	&\geq -2\|D^2 \varphi_\e\|_{L^\infty(\Rd)} F'(\|\varphi_\e\|_{L^\infty(\Rd)}) W_{2,\gamma}^2(\mu_2,\mu_3),
 \end{align*}
which gives the result.
\end{proof}

We now use the previous results to characterize the subdifferential of the regularized internal energy. The structure of argument is classical (c.f. \cite{JKO, 5person, AGS}), but due to the novel form of our regularized energies, we include the proof in Appendix \ref{appendix preliminaries}.

\begin{prop}[subdifferential characterization] \label{subdiffchar}
Suppose $F$ satisfies Assumption \ref{as:F} and is convex. Let $\e >0$ and $\mu \in D(\F_\e)$. Then
\bes
	v \in \partial \F_\e(\mu) \cap \Tan_\mu\P_2(\R^d) \iff v = \grad \frac{\delta \F_\e}{\delta \mu},
\ees
where 
\be \label{subdiffform} 
	\frac{\delta \F_\e}{\delta \mu} = \varphi_\e* \left(F'\circ(\varphi_\e*\mu) \mu \right) + F\circ(\varphi_\e*\mu), \quad \mbox{$\mu$-almost everywhere}.
\ee
In particular, we have $	|\partial \F_\e|(\mu) = \left\|\grad \frac{\delta \F_\e}{\delta \mu} \right\|_{L^2(\mu;\Rd)}$.
\end{prop}

As a consequence of this characterization of the subdifferential, we obtain the analogous result for the full energy $\E_\e$, as in Definition \ref{full energies}. See Appendix \ref{appendix preliminaries} for the proof.
\begin{cor} \label{full subdiff char}
Suppose $F$ satisfies Assumption \ref{as:F} and is convex. Let $\e >0$ and $\mu \in D(\E_\e)$. Suppose $V, W \in C^1(\R^d)$ are semiconvex, with at most quadratic growth, and $W$ is even.
Then
\bes
	v \in \partial \E_\e (\mu) \cap \Tan_\mu\P_2(\R^d) \iff v = \grad \frac{\delta \E_\e}{\delta \mu},
\ees
where
\bes
	\frac{\delta \E_\e}{\delta \mu} = \varphi_\e* \left(F'\circ(\varphi_\e*\mu) \mu \right) + F\circ(\varphi_\e*\mu) + V + W*\mu, \quad \mbox{$\mu$-almost everywhere}.
\ees
In particular, we have $|\partial \E_\e|(\mu) = \left\|\grad \frac{\delta \E_\e}{\delta \mu} \right\|_{L^2(\mu;\Rd)}$.
\end{cor}


\section{$\Gamma$-convergence of regularized internal energies} \label{Gamma convergence section}
We now turn to the convergence of the regularized energies and, when in the presence of confining drift or interaction terms, the corresponding convergence of their minimizers. In this section, and for the remainder of the work, we consider regularized entropies and R\'enyi entropies of the form $\F_\e^m$ for $m\geq 1$. We begin by showing that $\F_\e^m$ $\Gamma$-converges to $\F$ as $\e \to 0$ with respect to the weak-$^*$ topology.

\begin{thm}[$\F_\e$ $\Gamma$-converges to $\F$] \label{Gamma convergence theorem2}
If $m =1$, consider $(\mu_\e)_\e \subset \P_2(\Rd)$ and $\mu \in \P_2(\Rd)$, and if $m >1$, consider $(\mu_\e)_\e \subset \P(\Rd)$ and $\mu \in \P(\Rd)$.
\begin{enumerate}[(i)]
	\item If $\mu_\e \wsto \mu$, we have $\liminf_{\e \to 0 } \F^m_\e(\mu_\e) \geq \F^m(\mu)$. \label{liminf condition 2}
	\item We have $\limsup_{\e \to 0} \F^m_\e(\mu) \leq \F^m(\mu)$. \label{limsup condition 2}
\end{enumerate}
\end{thm}

\begin{proof}
We begin by showing the result for $1 \leq m \leq 2$, in which case the function $F$ is concave. We first show part (\ref{liminf condition 2}). By Proposition \ref{relative sizes lemma}, for all $\e>0$,
\[ 
	\F^m_\e(\mu_\e) \geq \F^m(\zeta_\e* \mu_\e) . 
\]
By Lemma \ref{weakst convergence mollified sequence},  $ \mu_\e \wsto \mu$ implies $\zeta_\e* \mu_\e \wsto \mu$. Therefore,  by the lower semicontinuity of $\F^m$ with respect to weak-$^*$ convergence \cite[Remark 9.3.8]{AGS},
\[ 
	\liminf_{\e \to 0} \F^m_\e(\mu_\e) \geq  \liminf_{\e \to 0}  \F^m(\zeta_\e* \mu_\e) \geq \F^m(\mu) , 
\]
which gives the result.
We now turn to part (\ref{limsup condition 2}). Again, by Proposition \ref{relative sizes lemma}, for all $\e>0$,
\[ 
	\F^m(\mu) + C_\e \geq \F^m_\e(\mu) , 
\]
where $C_\e \to 0$ as $\e \to 0$.
Therefore, $\limsup_{\e \to 0} \F^m_\e(\mu) \leq \F^m(\mu)$.

We now consider the case when $m>2$. Part (\ref{limsup condition 2}) follows quickly: by Proposition \ref{relative sizes lemma}, Young's convolution inequality, and the fact that $ \|\zeta_\e\|_{L^1(\Rd)} = 1$, for all $\e>0$ we have
\[ 
	\F^m_\e(\mu) \leq \F^m(\zeta_\e *\mu) = \tfrac{1}{m-1} \|\zeta_\e*\mu\|_{L^m(\Rd)}^m \leq \tfrac{1}{m-1} \|\zeta_\e\|_{L^1(\Rd)}^m \|\mu\|_{L^m(\Rd)}^m =\tfrac{1}{m-1}  \|\mu\|_{L^m(\Rd)}^m   = \F^m(\mu).
\]
Taking the supremum limit as $\e\to0$ then gives the result. Let us prove part (\ref{liminf condition 2}). Without loss of generality, we may suppose that $\liminf_{\e \to 0} \mathcal{F}^m_\e(\mu_\e)$ is finite. Furthermore, there exists a positive sequence $(\e_n)_n$ such that $\e_n \to 0$ and $\lim_{n \to + \infty} \mathcal{F}_{\e_n}^m(\mu_{\e_n}) = \liminf_{\e \to 0} \mathcal{F}^m_\e(\mu_\e)$. In particular, there exists $C>0$ for which $\F^m_{\e_n}(\mu_{\e_n})  < C$ for all $n \in \mathbb{N}$. By Jensen's inequality for the convex function $x \mapsto x^{m-1}$ and the fact that $\zeta_\e * \zeta_\e = \varphi_\e$ for all $\e>0$,
\[ 
	(m-1)\F^m_\e(\mu_\e)= \int (\varphi_\e * \mu_\e)^{m-1} \,d \mu_\e \geq \left( \int \varphi_\e * \mu_\e \,d \mu_\e \right)^{m-1} = \left( \ird |\zeta_\e*\mu_\e(x)|^2 \,dx \right)^{m-1}.
\]
Thus, since $\F^m_{\e_n}(\mu_{\e_n})  < C$ for all $n \in \N$, we have $\|\zeta_{\e_n}*\mu_{\e_n}\|_{L^2(\Rd)}< C':= (C(m-1))^{1/2(m-1)}$. We now use this bound on the $L^2$-norm of $\zeta_{\e_n}*\mu_{\e_n}$ to deduce a stronger notion of convergence of $\zeta_{\e_n}*\mu_{\e_n}$ to $\mu$. First, since $(\mu_{\e_n})_n$ converges weakly-$^*$ to $\mu$ as $n\to\infty$, Lemma \ref{weakst convergence mollified sequence} ensures that $(\zeta_{\e_n}*\mu_{\e_n} - \mu_{\e_n})_n$ converges weakly-$^*$ to $0$. Since the $L^2$-norm is lower semicontinuous with respect to weak-$^*$ convergence \cite[Lemma 3.4]{McCann}, we have
\[ 
	C' \geq \liminf_{n\to\infty} \|\zeta_{\e_n}*\mu_{\e_n}\|_{L^2(\Rd)} \geq \|\mu\|_{L^2(\Rd)},
\]
so that $\mu \in L^2(\Rd)$. Furthermore, up to another subsequence, we may assume that $(\zeta_{\e_n}*\mu_{\e_n})_n$ converges weakly in $L^2$. Also, since $\zeta_{\e_n}*\mu_{\e_n} \wsto \mu$, for all $f \in C_\mathrm{c}^\infty(\Rd)$ we have
\bes
	\lim_{n\to\infty} \int f \,d \zeta_{\e_n}*\mu_{\e_n} = \int f \,d\mu,
\ees
so $(\zeta_{\e_n}*\mu_{\e_n})_n$ converges weakly in $L^2$ to $\mu$. By the Banach--Saks theorem (c.f. \cite[Section 38]{Riesz}), up to taking a further subsequence of $(\zeta_{\e_n}*\mu_{\e_n})_n$, the Ces\`aro mean $(v_k)_k$ defined by 
\bes
	v_k := \frac{1}{k} \sum_{i=1}^k \zeta_{\e_i}*\mu_{\e_i} \quad \mbox{for all $k\in\N$},
\ees
converges to $\mu$ strongly in $L^2$. Finally, for any $f \in C_\mathrm{c}^\infty(\Rd)$, this ensures
\begin{align*}
	\left| \int f (v_k)^2\,d\mathcal{L}^d - \int f \mu^2\,d\mathcal{L}^d \right| &\leq \int |f| |v_k-\mu||v_k +\mu|\,d\mathcal{L}^d\\
	&\leq  \|f\|_{L^\infty(\Rd)}  \|v_k -\mu\|_{L^2(\Rd)} \|v_k + \mu\|_{L^2(\Rd)},
\end{align*}
so that
\be \label{dist conv of square} 
	\lim_{k \to \infty} \int f (v_k)^2\,d\mathcal{L}^d = \int f \mu^2 \, d\mathcal{L}^d.
\ee
We now use this stronger notion convergence to conclude our proof of part (\ref{liminf condition 2}). Since $m>2$ and 
\[  
	\| \varphi_{\e_n} * \mu_{\e_n} \|_{L^{m-1}(\mu_{\e_n};\Rd)}^{m-1} = (m-1) \F^m_{\e_n}(\mu_{\e_n})  < C \quad \text{for all $n \in \N$}, 
\]
by part (\ref{weakcpt}) of Proposition \ref{AGSthm}, up to another subsequence,  there exists $w \in L^1(\mu;\Rd)$ so that for all $f \in C_\mathrm{c}^\infty(\Rd)$,
\begin{align} 
 	\lim_{n\to\infty} \int f (\varphi_{\e_n} * \mu_{\e_n}) \,d \mu_{\e_n} = \int f w \,d \mu . \label{weak conv prob}
\end{align}
Furthermore, recalling the definition of the regularized energy and applying \cite[Theorem 5.4.4(ii)]{AGS},
\[ 
	\liminf_{\e \to 0} \mathcal{F}^m_\e(\mu_\e) = \lim_{n\to\infty} \mathcal{F}^m_{\e_n}(\mu_{\e_n}) = \lim_{n\to\infty} \frac{1}{m-1} \int (\varphi_{\e_n}*\mu_{\e_n})^{m-1} \,d \mu_n  \geq \frac{1}{m-1} \int w^{m-1} \,d\mu . 
\]
Therefore, to finish the proof, it suffices to show that $w(x) \geq \mu(x)$ for $\mu$-almost every $x \in \Rd$. By Lemma \ref{move mollifier prop} and the fact that $\zeta_{\e_n} *\zeta_{\e_n} = \varphi_{\e_n}$ for all $n\in\N$, there exists $p>0$ and $C_\zeta >0$ so that for all $f \in C_\mathrm{c}^\infty(\Rd)$,
\begin{align*} 
	&\left|\int f (\varphi_{\e_n} * \mu_{\e_n}) \,d \mu_{\e_n}  - \int f (\zeta_{\e_n} * \mu_{\e_n})^2 \,d\mathcal{L}^d\right|\\
	 &= \left|\int \zeta_{\e_n} *(f \mu_{\e_n}) \,d\zeta_{\e_n} *\mu_{\e_n} - \int (\zeta_{\e_n} * \mu_{\e_n})f \,d\zeta_{\e_n} * \mu_{\e_n} \right| \\
	&\leq \e_n^{p} \|\grad f \|_{L^\infty(\Rd)} \left( \|(\zeta_{\e_n}*\mu_{\e_n})\|_{L^2(\Rd)}^2 + C_\zeta \right)
\end{align*}
Combining this with equation \eqref{weak conv prob}, we obtain
\be \label{first zeta mu convergence}
	\lim_{n\to\infty} \int f (\zeta_{\e_n} *\mu_{\e_n})^2\, d\mathcal{L}^d = \int f w \,d \mu.
\ee
Finally, using equation \eqref{dist conv of square} and the definition of $(v_k)_k$ as a sequence of convex combinations of the family $\{\zeta_{\e_i} *\mu_{\e_i}\}_{i\in\{1,\dots,k\}}$, for all $f \in C_\mathrm{c}^\infty(\Rd)$ with $f \geq 0$ we have
\begin{align*}
	\int f \mu^2 \,d\mathcal{L}^d &= \lim_{k \to\ \infty} \int f(v_k)^2 \,d\mathcal{L}^d = \lim_{k \to \infty} \ird f \left( \frac{1}{k}\sum_{n=1}^k \zeta_{\e_n}*\mu_{\e_n}(x) \right)^2\,dx\\
	&\leq  \lim_{k \to \infty} \frac{1}{k} \sum_{n=1}^k  \int f \left( \zeta_{\e_n}*\mu_{\e_n} \right)^2\,d\mathcal{L}^d. 
\end{align*}
Since the limit in \eqref{first zeta mu convergence} exists, it coincides with its Ces\`aro mean on the right-hand side of the above equation. Thus, for all $f \in C_\mathrm{c}^\infty(\Rd)$ with $f \geq 0$,
\[ 
	\int f \mu^2\,d\mathcal{L}^d \leq \int f w \,d \mu. 
\]
This gives $w(x) \geq \mu(x)$ for $\mu$-almost every $x \in \Rd$, which completes the proof.
\end{proof}

Now, we add a confining drift or interaction potential to our internal energies, so that energy minimizers exist and we may apply the previous $\Gamma$-convergence result to conclude that minimizers converge to minimizers. For the remainder of the section we consider energies of the form $\E_\e^m$ given in Definition \ref{full energies}, with the following additional assumptions on $V$ and $W$ to ensure that the energy is confining.
\begin{as}[confining assumptions]  
	The potentials $V$ and $W$ are bounded below and one of the following additional assumptions holds:
\begin{align*}
\tag{CV} &V \text{ has compact sublevel sets}; \label{V} \\
\tag{CV$'$} &V(x) \geq C_0|x|^2 + C_1 \text{ for all $x \in \Rd$ for some $C_0 >0, C_1 \in \R$}; \label{V'}\\
\tag{CW} &V=0\mbox{ and }W \text{ is radial satisfying } \lim_{|x| \to \infty} W(x) = +\infty. \label{W}
\end{align*}
\end{as}

Under these assumptions, the regularized energies $\E_\e^m$ are lower semicontinuous with respect to weak-$^*$ convergence ($m>1$) and Wasserstein convergence ($m=1$), where for the latter we assume $\varphi$ is a Gaussian (c.f. Proposition \ref{lower semicontinuity}, and \cite[Lemma 5.1.7]{AGS}, \cite[Lemma 3.4]{McCann} and \cite[Lemma 2.2]{SimioneSlepcevTopaloglu}).

\begin{remark}[tightness of sublevels] \label{sublevel remark}
	Assumptions \eqref{V} and \eqref{V'} ensure that the set $\{ \mu \in \P(\Rd) \mid \int V \,d \mu \leq C \}$ is tight for all $C>0$; c.f. \cite[Remark 5.1.5]{AGS}. Likewise, Assumption \eqref{W} on $W$ ensures that the set $\{ \mu \in \P(\Rd) \mid \int W*\mu \,d \mu \leq C \}$ is tight up to translations for all $C >0$; c.f. \cite[Theorem 3.1]{SimioneSlepcevTopaloglu}.
\end{remark}

We now prove existence of minimizers of $\E_\e^m$, for all $\e>0$.

\begin{prop} \label{minimizersexist}
	Let $\e>0$. For $m>1$, if either Assumption \eqref{V} or \eqref{W} holds, then minimizers of $\E^m_\e$ over $\P(\Rd)$ exist. For $m=1$, if \eqref{V'} holds and $\varphi$ is a Gaussian, then minimizers of $\E^1_\e$ over $\P_2(\Rd)$ exist.
\end{prop}

\begin{proof}
	First suppose $m >1$, so that $\F_\e \geq 0$ and $\E^m_\e$ is bounded below. By Remark \ref{sublevel remark}, if \eqref{V} holds, then any minimizing sequence of $\E^m_\e$ has a subsequence that converges in the weak-$^*$ topology. Likewise, if (\ref{W}) holds, then any minimizing sequence of $\E^m_\e$ has a subsequence that, up to translation, converges in the weak-$^*$ topology. By lower semicontinuity of $\E^m_\e$, the limits of minimizing sequences are  minimizers of $\E^m_\e$.

Now, suppose $m =1$. By Proposition \ref{relative sizes lemma}, for all $\delta >0$ and $\mu\in\P_2(\Rd)$,
\[ 
	\F^m_\e(\mu)  \geq  -(2\pi/\delta)^{d/2} - 2\delta (M_2(\mu)+ \e^2 M_2(\zeta)),
\]
Consequently, by the assumption in \eqref{V'} and the fact that $W$ is bounded below by, say, $\tilde C \in\R$, we can choose $\delta = C_0/2$ and obtain
\begin{align} \label{existence ineq}
	\tilde C + C_0M_2(\mu)+C_1 - \left(4\pi/C_0\right)^{d/2} - C_0 \e^2 M_2(\zeta) \leq \E^m_\e(\mu) \quad \mbox{for all $\mu \in \P_2(\Rd)$},
\end{align}
Hence  any minimizing sequence $(\mu_n)_n \subset \P_2(\Rd)$ has uniformly bounded second moment. Thus, $(\mu_n)_n$ has a subsequence that converges in the weak-$^*$ topology to a limit  with finite second moment. By the lower semicontinuity of $\E^m_\e$ the limit must be a minimizer of $\E^m_\e$. 
\end{proof}

Finally, we conclude that minimizers of the regularized energy converge to minimizers of the unregularized energy.

\begin{thm}[minimizers converge to minimizers] \label{minimizers converge theorem}
Suppose $m>1$. If  Assumption \eqref{V} holds, then for any sequence $(\mu_\e)_\e \subset \P(\Rd)$ such that $\mu_\e$ is a minimizer of $\E^m_\e$ for all $\e>0$, we have, up to a subsequence, $\mu_\e \wsto \mu$, where $\mu$ is minimizes $\E^m$. Alternatively, if  Assumption \eqref{W} holds, then we have $\mu_\e \wsto \mu$, up to a subsequence and translation, where again $\mu$ minimizes $\E^m$.

Now suppose $m =1$. If Assumption \eqref{V'} holds and $\varphi$ is a Gaussian, then for any sequence $(\mu_\e)_\e \subset \P_2(\Rd)$ such that $\mu_\e$ is a minimizer of $\E^1_\e$ for all $\e>0$, we have, up to a subsequence, $\mu_\e \wsto \mu$, where $\mu$ minimizes $\E^1$.
%
\end{thm}

\begin{proof}
	The proof is classical. We include it for completeness.
	
	We only prove the result under Assumptions \eqref{V}/\eqref{V'} since the argument for (\ref{W}) is analogous. For any $\e>0$, since $\mu_\e$ is a minimizer of $\E^m_\e$ we have that $\E^m_\e(\mu_\e) \leq \E^m_\e(\nu)$ for all $\nu \in\P(\Rd)$ if $m>1$, and for all $\nu \in \P_2(\Rd)$ if $m=1$. Taking the infimum limit of the left-hand side and the supremum limit of the right-hand side, Theorem \ref{Gamma convergence theorem2}(\ref{limsup condition 2}) ensures that 
\be \label{minimizers ineq}
	\liminf_{\e \to 0} \E^m_\e(\mu_\e) \leq \limsup_{\e \to 0} \E^m_\e(\nu) \leq \E^m(\nu) . 
\ee
Since $\E^m$ is proper there exists $\nu \in \P(\Rd)$ if $m>1$ and $\nu \in \P_2(\Rd)$ if $m=1$ so that the right-hand side is finite. Thus, up to a subsequence, we may assume that $ \{\E^m_\e(\mu_\e)\}_\e$ is uniformly bounded. When $m >1$, $\F_\e(\mu) \geq 0$ for all $\e \geq 0$, and this implies that $\{\int V \,d \mu_\e\}_\e$ is uniformly bounded, so $\{\mu_\e\}_\e$ is tight. When $m=1$, the inequality in \eqref{existence ineq} ensures that $\{M_2(\mu_\e)\}_\e$ is uniformly bounded, so again $\{\mu_\e\}_\e$ is tight. Thus, up to a subsequence, $(\mu_\e)_\e$ converges weakly-$^*$ to a limit $\mu \in \P(\Rd)$ if $m>1$ and $\mu \in \P_2(\Rd)$ if $m=1$. By Theorem \ref{Gamma convergence theorem2}(\ref{liminf condition 2}) and the inequality in \eqref{minimizers ineq}, we obtain
\[ 
	\E^m(\mu) \leq  \liminf_{\e \to 0} \E^m_\e(\mu_\e) \leq \E^m(\nu)
\]
for all $\nu \in \P(\Rd)$ if $m>1$ and for all $\nu \in \P_2(\Rd)$ if $m=1$. Therefore, $\mu$ is a minimizer of $\E^m$.
\end{proof}

\begin{remark}[convergence of minimizers]
	One the main difficulties for improving the topology in which the convergence of the minimizers happen is that we do not control $L^m$-norms of the regularized minimizing sequences due to the special form of our regularized energy. This is the main reason we only get weak-$^*$ convergence in the previous result and the main obstacle to improve results for the $\Gamma$-convergence of gradient flows, as we shall see in the next section.
\end{remark}


\section{$\Gamma$-convergence of gradient flows} \label{Gamma convergence GF section} 

We now consider gradient flows of the regularized energies $\E_\e^m$, as in Definition \ref{full energies}, for $m \geq 2$ and prove that, under sufficient regularity assumptions, gradient flows of the regularized energies converge to gradient flows of the unregularized energy as $\e \to 0$. For simplicity of notation, we often write $\E_\e^m$ and $\F_\e^m$ for $\e \geq 0$ when we refer jointly to the regularized and unregularized energies.

We begin by showing that the gradient flows of the regularized energies are well-posed, provided that $V$ and $W$ satisfy the following convexity and regularity assumptions.

\begin{as}[convexity and regularity of $V$ and $W$] \label{convexityregularity}
The potentials $V, W \in C^1(\Rd)$ are semiconvex, with at most quadratic growth, and $W$ is even. Furthermore,  there exist $C_0, C_1 >0$ so
\[
	|W(x)|, |\grad V(x)|, |\grad W(x)| \leq C_0 + C_1|x|^{m-1} \quad \mbox{for all $x \in \Rd$}.
 \]
\end{as}

\begin{remark}[$\omega$-convexity]
More generally, our results naturally extend to drift and interaction energies that are merely $\omega$-convex; see \cite{CraigNonconvex}. However, given that the main interest of the present work is approximation of diffusion, we prefer the simplicity of Assumption (\ref{convexityregularity}), as it allows us to focus our attention on the regularized internal energy.
\end{remark}
\begin{prop} \label{wellposedness}
	Let $\e\geq0$ and $m\geq2$. Suppose $\E^m_\e$ is as in Definition \ref{full energies} and $V$ and $W$ satisfy Assumption \ref{convexityregularity}. Then, for any $\mu_0\in \overline{D(\E_\e^m)}$, there exists a unique gradient flow of $\E_\e^m$ with initial datum $\mu_0$.
\end{prop}

\begin{proof}
	It suffices to verify that $\E_\e^m$ is proper, coercive, lower semicontinuous with respect to $2$-Wasserstein convergence, and semiconvex along generalized geodesics; c.f. \cite[Theorem 11.2.1]{AGS}. (See also \cite[Equation (2.1.2b)]{AGS} for the definition of coercive.)  If $\e >0$, then $\F^m_\e$ is finite on all of $\P_2(\Rd)$, and if $\e=0$, then $\F^m$ is proper. Thus, our assumptions on $V$ and $W$ ensure that $\E_\e^m$ is proper.
Clearly $\F^m_\e$ is bounded below. Hence, since the semiconvexity of $V$ and $W$ ensures that their negative parts have at most quadratic growth, $\E_\e^m$ is coercive.

For $\e >0$, Proposition \ref{lower semicontinuity} ensures that $\F_\e^m$ is lower semicontinuous with respect to weak-$^*$ convergence, hence also $2$-Wasserstein convergence. For $\e = 0$, the unregularized internal energy $\F^m$ is also lower semicontinuous with respect to weak-$^*$ and $2$-Wasserstein convergence \cite[Lemma 3.4]{McCann}. Since $V$ and $W$ are lower semicontinuous and their negative parts have at most quadratic growth, the associated potential and interaction energies are lower semicontinuous with respect to $2$-Wasserstein convergence \cite[Lemma 5.1.7, Example 9.3.4]{AGS}. Therefore, $\E_\e^m$ is lower semicontinuous for all $\e \geq 0$.

	For $\e>0$,  Proposition \ref{prop:conv} ensures that $\F_\e^m$ is semiconvex along generalized geodesics in $\P_2(\R^d)$. For $\e=0$, the unregularized internal energy $\F^m$ is convex \cite[Theorem 2.2]{McCann}. For $V$ and $W$ semiconvex, the corresponding drift $\int V \,d \mu$ and interaction $(1/2)\int (W*\mu) \,d \mu$ energies are semiconvex \cite[Proposition 9.3.2]{AGS}, \cite[Remark 2.9]{5person}. Therefore, the resulting regularized energy $\E^m_\e$ is semiconvex. 	
\end{proof}

In the case $\e =0$, gradient flows of the energies $\E^m$ are characterized as solutions of the partial differential equation (\ref{W2PDE1}); c.f. \cite[Theorems 10.4.13 and 11.2.1]{AGS}, \cite[Theorem 2.12]{5person}. Now, we show that gradient flows of the regularized energies $\E_\e^m$ can also be characterized as solutions of a partial differential equation.

\begin{prop} \label{characterizationGF}
	Let $\e>0$ and $m\geq2$. Suppose $\E_\e^m$ is as in Definition \ref{full energies} and $V$ and $W$ satisfy Assumption \ref{convexityregularity}. Then, $\mu_\e \in AC^2_\loc((0,+\infty); \P_2(\Rd))$ is the gradient flow of $\E^m_\e$ if and only if $\mu_\e$ is a weak solution of the continuity equation with velocity field 
\begin{align} \label{velocityfield}
	v = - \grad V - \grad W*\mu_\e -\grad \varphi_\e* \left( (\varphi_\e*\mu_\e)^{m-2} \mu_\e \right) - (\varphi_\e*\mu_\e)^{m-2} \grad \varphi_\e * \mu_\e \,.
	\end{align}
Moreover, $\int_0^T  \|v(t)\|^2_{L^2(\mu_\e;\Rd)} \,dt <\infty$ for all $T >0$.
\end{prop}

\begin{proof}
	Suppose $\mu_\e \in AC^2_\loc((0,+\infty); \P_2(\Rd))$ is the gradient flow of $\E^m_\e$. Then, by Definition \ref{gradientflowdef} and Corollary \ref{full subdiff char}, $\mu_\e$ is a weak solution to the continuity equation with velocity field (\ref{velocityfield}). Conversely, suppose $\mu_\e$ is a weak solution to the continuity equation with velocity field (\ref{velocityfield}). By Corollary \ref{full subdiff char}, $-v(t) \in \partial \E(\mu(t)) \cap \Tan_{\mu(t)}\P_2(\R^d)$ for almost every $t \in (0,\infty)$. Furthermore, since $\int_0^T  \|v(t)\|^2_{L^2(\mu_\e;\Rd)}\,dt <\infty$ for all $T >0$, $\mu_\e \in AC^2_\loc((0,+\infty); \P_2(\Rd))$ by \cite[Theorem 8.3.1]{AGS}.
\end{proof}

A consequence of the previous proposition is that, for the regularized energies $\E^m_\e$, particles remain particles, i.e. a solution of the gradient flow with initial datum given by a finite sum of Dirac masses remains a sum of Dirac masses, and the evolution of the trajectories of the particles is given by a system of ordinary differential equations.
\begin{cor} \label{particles well posed}
	Let $\e>0$ and $m\geq2$, and let $V$ and $W$ satisfy Assumption \ref{convexityregularity}. Fix $N \in \N$. For $i \in \{1, \dots, N\}:=I$, fix $X_i^0 \in \Rd$ and $m_i \geq 0$ satisfying $\sum_{i \in I} m_i = 1$.
	Then the ODE system
\begin{align} \label{ODEsystem}
	\begin{cases}
		 \dot{X}_i(t) =-\grad V(X_i(t))- \sum_{j \in I} \grad W(X_i(t)-X_j(t))m_j- \grad \frac{\delta \F_\e^m}{\delta \mu} (\Sigma_j \delta_{X_j(t)} m_j ), & t\in[0,T],\\
		 X_i(0) = X_i^0,
	\end{cases}
\end{align}
is well-posed for all $T>0$. Furthermore, $\mu_\e = \sum_{i \in I} \delta_{X_i(\cdot)} m_i$ belongs to $AC^2([0,T];\P_2(\R^d))$ and is the gradient flow of $\E_\e^m$ with initial conditions $\mu_\e(0) := \sum_{i \in I} \delta_{X_i^0}m_i$.
\end{cor}

\begin{proof}
To see that (\ref{ODEsystem}) is well-posed, first note that the function
\begin{align*}
	(y_1,\dots,y_N) \mapsto &\grad \frac{\delta \F_\e^m}{\delta \mu} (\Sigma_j \delta_{y_j} m_j )\\
		&= \sum_{j \in I}   \left( \left( \sum_{k \in I}\varphi_\e(y_j - y_k) m_k \right)^{m-2} + \left( \sum_{k \in I}\varphi_\e(y_i - y_k) m_k\right)^{m-2}  \right) \grad \varphi_\e(y_i - y_j) m_j \end{align*}
is Lipschitz. Likewise, Assumption \ref{convexityregularity} ensures $y_i \mapsto \grad V(y_i)$ and $y_i \mapsto \sum_{j\in I} \grad W(y_i-y_j)$ are continuous and one-sided Lipschitz. Therefore, the ODE system \eqref{ODEsystem} is well-posed forward in time.
	
	Now, suppose $(X_i)_{i=1}^N$ solves (\ref{ODEsystem}) with initial data $(X_i^0)_{i =1}^N$ on an interval $[0,T]$, for some fixed $T$. We abbreviate by $v_i = v_i(X_1, X_2, \dots, X_N)$ the velocity field for $X_i$ in (\ref{ODEsystem}). For any test function $\varphi \in C_\mathrm{c}^\infty(\Rd \times (0,T))$, the fundamental theorem of calculus ensures that, for all $i \in I$,
\begin{align*}
	\int_0^T \left( \grad \varphi(X_i(t),t) \dot{X}_i(t) + \partial_t \varphi( X_i(t),t) \right) \,dt = -\varphi(X_i(0),0).
\end{align*}
Combining this with \eqref{ODEsystem}, we obtain
\begin{align*}
	\int_0^{T} \partial_t \varphi(X_i(t),t) \,dt  + \varphi(X_i^0,0) - \int_0^{T} \grad \varphi(X_i(t),t) v_i(t) \,dt = 0
\end{align*}
Multiplying both sides by $m_i$, summing over $i$, and taking $\mu_\e = \sum_{i\in I} \delta_{X_i(\cdot)} m_i$ for $t\in[0,T]$ gives
\begin{align*}
	\int_0^{T} \ird\partial_t  \varphi(t,x)\,d \mu_\e(t,x) dt + \int_\Rd \varphi(0,x)\,d \mu_\e(0,x) +\int_0^{T} \int_{\Rd} \grad \varphi(t,x) v(t,x) \,d \mu_\e(t,x) \,dt  =0,
\end{align*}
for $v$ as in (\ref{velocityfield}). Therefore, $\mu_\e$ is a weak solution of the continuity equation with velocity field $v$. Furthermore, for all $T >0$
\begin{align*} 
	\int_0^T  \|v(t)\|^2_{L^2(\mu_\e;\Rd)}\,dt \leq 2\max_{(i,j,k) \in I^3} &\Bigg[ \int_0^T \left(|\grad V(X_i(t))|^2 + |\grad W(X_i(t)-X_j(t))|^2\right) \,dt \\
	&\quad + \int_0^T \Big( \left|(\varphi_\e(X_j(t)-X_k(t))^{m-2}+(\varphi_\e(X_i(t)-X_k(t))^{m-2} \right|^2\\
	&\phantom{{}=={}} \times \left|\grad \varphi_\e(X_i(t)-X_j(t)) \right|^2 \Big) \,dt \Bigg] < \infty,
\end{align*}
by the continuity of $\grad V$, $\grad W$, and $\varphi_\e$. Therefore, by Proposition \ref{characterizationGF}, we conclude that $\mu_\e \in AC^2([0,T];\P_2(\R^d))$ and $\mu_\e$ is the gradient flow of $\E^m_\e$.
\end{proof}

We now turn to the $\Gamma$-convergence of the gradient flows of the regularized energies, using the scheme introduced by Sandier--Serfaty \cite{SaSe} and then generalized by Serfaty \cite{Serfaty}, which provides three sufficient conditions for concluding convergence. We will use the following variant of Serfaty's result, which allows for slightly weaker assumptions on the gradient flows of the regularized energies, but follows from the same argument as Serfaty's original result. (See also Remark \ref{metric slope remark} on the correspondence between Wasserstein gradient flows and curves of maximal slope.)

\begin{thm}[c.f. {\cite[Theorem 2]{Serfaty}}] \label{Serfaty theorem}
	Let $m\geq2$. Suppose that, for all $\e>0$, $\mu_\e$ belongs to $AC^2([0,T];\P_2(\R^d))$ and is a gradient flow of $\E_\e^m$ with well-prepared initial data, i.e., 
\begin{align} \tag{S0} \label{S0}
	\mu_\e(0) \wsto \mu(0), \quad \lim_{\e \to 0} \E_\e^m(\mu_\e(0)) = \E^m(\mu(0)), \quad \mu(0) \in D(\E^m).
\end{align}
Suppose further that there exists a curve $\mu$  in $\P_2(\Rd)$ such that, for almost every $t \in[0,T]$, $ \mu_\e(t) \wsto \mu(t)$ and
\begin{enumerate}[(S1)]
	\item \label{cond:md} $\displaystyle \liminf_{\e\to 0} \int_0^t |\mu_\e'|(s)^2 \,d s \geq \int_0^t|\mu'|(s)^2\,d s$,\\
	\item \label{cond:liminf} $\displaystyle \liminf_{\e\to0} \E^m_\e(\mu_\e(t)) \geq \displaystyle \E^m(\mu(t))$,\\
	\item \label{cond:slopes} $\displaystyle \liminf_{\e\to 0} \int_0^t |\partial \E^m_\e|^2(\mu_\e(s)) \,ds \geq \int_0^t |\partial \E^m|^2(\mu(s))\,ds$.
\end{enumerate}
Then $\mu \in AC^2([0,T];\P_2(\R^d))$, and $\mu$ is a gradient flow of $\E^m$.
\end{thm}

For simplicity of notation, in what follows we shall at times omit dependence on time when referring to curves in the space of probability measures.

In order to apply Serfaty's scheme in the present setting to obtain $\Gamma$-convergence of the gradient flows, a key assumption  is that  the following quantity is bounded uniformly in $\e>0$ along the gradient flows $\mu_\e$ of the regularized energies $\E^m_\e$:
\bes 
	\| \mu_\e\|_{BV_\e^m} :=  \ird\ird  \zeta_\e(x-y) \left| (\nabla \zeta_\e * p_\e)(x)  + (\nabla \zeta_\e * \mu_\e) (x) (\varphi_\e*\mu_\e)(y)^{m-2} \right| \,d\mu_\e(y) \,dx \,,
\ees
where we use the abbreviation $p_\e :=(\varphi_\e*\mu_\e)^{m-2}\mu_\e$. This quantity differs from $\norm{\grad \delta \F_\e^m/\delta \mu_\e}_{L^1(\mu_\e;\Rd)}$ merely by the placement of the absolute value sign:
\begin{align}
	\| \mu_\e\|_{BV_\e^m} &\geq  \ird \left| \ird  \zeta_\e(x-y)  (\nabla \zeta_\e * p_\e)(x)  + (\nabla \zeta_\e * \mu_\e) (x) (\varphi_\e*\mu_\e)(y)^{m-2}dx \right| d \mu_\e(y) \nonumber\\
	&= \int  \left| (\nabla \varphi_\e * p_\e)+ (\nabla \varphi_\e * \mu_\e) (\varphi_\e*\mu_\e)^{m-2} \right|  d \mu_\e = \left\|\grad \frac{\delta \F_\e^m}{\delta \mu_\e} \right\|_{L^1(\mu_\e;\Rd)}.\label{BV heuristic 1} 
\end{align}

Serfaty's scheme allows one to assume, without loss of generality, that $|\F^m_\e|(\mu_\e)$ is bounded uniformly in $\e>0$ for almost every $t\in [0,T]$, and H\"older's inequality ensures that $|\F^m_\e|(\mu_\e) = \norm{\grad \delta \F_\e^m/\delta \mu_\e}_{L^2(\mu_\e;\Rd)} \geq \norm{\grad \delta \F_\e^m/\delta \mu_\e}_{L^1(\mu_\e;\Rd)}$; see Proposition \ref{subdiffchar}. Consequently, we miss the bound we require on $\| \mu_\e\|_{BV_\e^m}$ merely by placement of the absolute value sign in inequality (\ref{BV heuristic 1}).

Still, $\| \mu_\e\|_{BV_\e^m}$ has a useful heuristic interpretation. Through the proof of Theorem \ref{Gamma GF theorem}, we obtain
\begin{align} 
	\liminf_{\e \to 0} \int_0^T \left\|\grad \frac{\delta \F_\e^m}{\delta \mu_\e}\right\|_{L^1(\mu_\e;\Rd)}\,dt &\geq \frac{m}{m-1} \int_0^T \left\|\grad \mu(t)^{m-1} \right\|_{L^1(\mu(t);\Rd)} \,dt \nonumber\\
	&= \int_0^T \ird \left| \grad \mu(t,x)^{m} \right|\, dx \,dt;\label{BV heuristic 2} 
\end{align}
see the inequality \eqref{weak convergence of velocities} and Proposition \ref{AGSthm}. Consequently, one may think of $\| \mu_\e\|_{BV_\e^m} $ as a nonlocal approximation of the $L^1$-norm of the gradient of $\mu^m$.

We begin with a technical lemma we shall use to prove the convergence of the gradient flows.

\begin{lem} \label{technical move convolution}
	Let $\e>0$ and $m \geq 2$, and let $T >0$ and $\mu_\e \in AC^2([0,T];\P_2(\R^d))$. Then for any Lipschitz function $f: [0,T]\times \Rd \to \R$ with constant $L_f>0$, there exists $r>0$ so that
\begin{align*}
&\norm{\left[(\zeta_\e *(f \mu_\e)) - f( \zeta_\e * \mu_\e) \right](\grad\zeta_\e*p_\e) + \left[(\zeta_\e*(f p_\e)) - f(\zeta_\e* p_\e)\right](\grad \zeta_\e*\mu_\e)}_{L^1([0,T] \times \Rd)} \\
 & \leq  \e^r L_f  \left( \int_0^T\| \mu_\e(t)\|_{BV_\e^m}\,dt + 2 C_\zeta\|\grad \zeta\|_{L^1(\Rd)} T^{1/(m-1)} \left(  \int_0^T \F^m_\e(\mu_\e(t))\,dt \right)^{\frac{m-2}{m-1}} \right),
  \end{align*}
where $C_\zeta>0$ is as in Assumption \ref{mollifierAssumption}.
\end{lem}

\begin{proof}
We argue similarly as in Lemma \ref{move mollifier prop}. Let $f: [0,T]\times \Rd \to \R$ be Lipschitz with constant $L_f>0$. Then,
\begin{align*}
	& \int \big| \left[(\zeta_\e *(f \mu_\e)) - f( \zeta_\e * \mu_\e) \right](\grad\zeta_\e*p_\e) + \left[(\zeta_\e*(f p_\e)) - f(\zeta_\e* p_\e)\right](\grad \zeta_\e*\mu_\e) \big| \,d\mathcal{L}^d\\
	&= \ird \left| \ird \zeta_\e(x-y)[f(y)-f(x)]  \left[ (\grad \zeta_\e*p_\e)(x) + (\grad \zeta_\e*\mu_\e)(x) (\varphi_\e*\mu_\e)(y)^{m-2} \right] \,d \mu_\e(y)  \right| \,dx \\
	&\leq L_f\ird\ird \zeta_\e(x-y) |x-y|\left| (\grad \zeta_\e*p_\e)(x) + (\grad \zeta_\e*\mu_\e)(x)(\varphi_\e*\mu_\e)(y)^{m-2} \right| \,d \mu_\e(y) \,dx\,.
\end{align*}
By Assumption \ref{mollifierAssumption}, $C_\zeta$ is so that $\zeta(x) \leq C_\zeta |x|^{-q}$ for $q >d+1$ for all $x \in \Rd$. Choose $\bar r$ so that
\begin{align} \label{pdef}
0< \bar r < \frac{q-(d+1)}{q-1} . 
\end{align}
Now, we break the integral with respect to $d \mu_\e(y)$ above into integrals over the domain $B_{\e^{\bar r}}(x)$ and $\Rd \setminus B_{\e^{\bar r}}(x)$, bounding the above quantity by
\begin{align*}
	&\, L_f \int_\Rd \int_{B_{\e^{\bar r}}(x)} \zeta_\e(x-y) |x-y|\left| (\grad \zeta_\e*p_\e)(x) + (\grad \zeta_\e*\mu_\e)(x)(\varphi_\e*\mu_\e)(y)^{m-2} \right| \,d \mu_\e(y) \,dx\\
	&\phantom{{}={}} +L_f   \ird \int_{\Rd \setminus B_{\e^{\bar r}}(x)} \zeta_\e(x-y) |x-y|\left| (\grad \zeta_\e*p_\e)(x) + (\grad \zeta_\e*\mu_\e)(x)(\varphi_\e*\mu_\e)(y)^{m-2} \right| \,d \mu_\e(y) \,dx , \\
 	&=: I_1 + I_2 
\end{align*}
First, we consider $I_1$. Since, in the integral, $|x-y| < \e^{\bar r}$, we obtain
\[ I_1 < \e^{\bar r} L_f \| \mu_\e\|_{BV_\e^m} . \]
Now, we consider $I_2$. We apply the inequality in \eqref{zeta bound} to obtain $\zeta_\e(x-y)|x-y| \leq C_\zeta \e^{\tilde{r}}$ with $\tilde{r}:={\bar r}(1-q)+q-d$ in the integral---the inequality in \eqref{pdef} ensures $\tilde{r}>1$.
Consequently,
\begin{align*}
	I_2 &\leq \e^{\tilde{r}} L_f C_\zeta \left( \int |\grad \zeta_\e *p_\e| \,d\mathcal{L}^d \int d\mu_\e + \int| \grad \zeta_\e *\mu_\e |\,d\mathcal{L}^d \int p_\e\,d\mathcal{L}^d \right)\\
	&\leq 2 \e^{\tilde{r}} L_f C_\zeta \|\grad \zeta_\e\|_{L^1(\Rd)} \int p_\e\,\mathcal{L}^d \leq 2 \e^{\tilde{r}-1} L_f C_\zeta\|\grad \zeta\|_{L^1(\Rd)}  \F^m_\e(\mu)^{(m-2)/(m-1)} ,
\end{align*}
where, in the last inequality, we use that $\|\grad \zeta_\e\|_{L^1(\Rd)} = \|\grad \zeta\|_{L^1(\Rd)}/\e$ and, by Jensen's inequality for the concave function $s^{(m-2)/(m-1)}$,
\begin{align} \label{pepsbd}
	\int p_\e\,d\mathcal{L}^d = \int (\varphi_\e*\mu_\e)^{m-2} d \mu_\e \leq \left( \int (\varphi_\e*\mu_\e)^{m-1} d \mu_\e\right)^{(m-2)/(m-1)} = \F^m_\e(\mu_\e)^{(m-2)/(m-1)} .
\end{align}
Since $0\leq (m-2)/(m-1) <1$, Jensen's inequality gives
\begin{align} \label{Jensen in t}
	\int_0^T \F^m_\e(\mu_\e(t))^{(m-2)/(m-1)} \,dt \leq T \left( \frac{1}{T} \int_0^T \F^m_\e(\mu_\e(t)) \,dt\right)^{(m-2)/(m-1)} .
\end{align}
This gives the result by taking $r:=\min ({\bar r},\tilde r -1)$.
\end{proof}

With this technical lemma in hand, we now turn to the $\Gamma$-convergence of the gradient flows.

\begin{thm} \label{Gamma GF theorem}
	Let $m \geq 2$, and let $V$ and $W$ be as in Assumption \ref{convexityregularity}. Fix $T>0$ and suppose that $\mu_\e \in AC^2([0,T];\P_2(\R^d))$ is a gradient flow of $\E^m_\e$ for all $\e>0$ satisfying
\begin{align} \tag{A0} \label{A0}
	\sup_{\e >0} M_2(\mu_\e(0)) < +\infty, \quad \mu_\e(0) \wsto \mu(0), \quad \lim_{\e \to 0} \E^m_\e(\mu_\e(0)) = \E^m(\mu(0)),
\end{align}
for some $\mu(0) \in D(\E^m)$. Furthermore, suppose that the following hold:
\begin{enumerate}[(\text{A}1)]	
	\label{extra assumptions b}
	\item $\sup_{\e >0} \int_0^T \| \mu_\e(t) \|_{BV_\e^m}dt  < \infty $; \label{extra assumptions c}
	\item there exists $\mu\: [0,T] \to \P_2(\R^d)$ such that $\zeta_\e*\mu_\e(t) \to \mu(t)$ in $L^1([0,T];L^m_\loc(\Rd))$ as $\e\to0$, and $\sup_{\e >0} \int_0^T \|\zeta_\e* \mu_\e(t)\|^m_{L^m(\R^d)} \,dt <  \infty $. \label{extra assumptions a}
\end{enumerate}
Then $\mu_\e(t) \wsto \mu(t)$ for almost every $t \in [0,T]$, $\mu \in AC^2([0,T];\P_2(\R^d))$, and $\mu$ is the gradient flow of $\E^m$ with initial data $\mu(0)$.
\end{thm}

\begin{proof}
First, we note that $\mu_\e(t) \wsto \mu(t)$ for almost every $t \in [0,T]$. This follows from (A\ref{extra assumptions a}), which ensures $\zeta_\e*\mu_\e(t) \to \mu(t)$ in $L^1([0,T];L^m_\loc(\Rd))$, hence $\zeta_\e*\mu_\e(t) \to \mu(t)$ in distribution for almost every $t \in [0, T]$. Then, since  $\zeta_\e*\mu_\e(t) - \mu_\e(t) \to 0$ in distribution for all $t \in [0,T]$, we obtain $\mu_\e(t) \to \mu(t)$ in distribution. Finally, since weak-* convergence and convergence in distribution are equivalent when $\mu_\e$ and $\mu$ are both probability measures \cite[Remark 5.1.6]{AGS}, we obtain $\mu_\e(t) \wsto \mu(t)$ for almost every $t \in [0,T]$.

It remains to verify conditions (\ref{S0}), (S\ref{cond:md}), (S\ref{cond:liminf}), and (S\ref{cond:slopes}) from Theorem \ref{Serfaty theorem}. Item (\ref{S0}) holds by assumption (\ref{A0}).
Item (S\ref{cond:md}) follows by the same argument as in \cite[Theorem 5.6]{CraigTopaloglu}. Item (S\ref{cond:liminf}) is an immediate consequence of the fact that $\mu_\e(t) \wsto \mu(t)$ for almost every $t \in [0,T]$, our main $\Gamma$-convergence Theorem \ref{Gamma convergence theorem2}, and the lower semicontinuity of the potential and interaction energies with respect to weak-$^*$ convergence \cite[Lemma 5.1.7]{AGS}. 

We devote the remainder of the proof to showing Condition (S\ref{cond:slopes}). 
We shall use the following fact throughout: combining Assumption (A\ref{extra assumptions a}) with Proposition \ref{relative sizes lemma} implies that 
\begin{equation}\label{auxxx} 
	\sup_{\e > 0} \int_0^T \F^m_\e(\mu_\e(t))\,dt  \leq  \sup_{\e > 0}  \frac{1}{m-1} \int_0^T \|\zeta_\e*\mu_\e(t)\|_{L^m(\R^d)}^m\,dt <\infty . 
\end{equation}
 To prove (S\ref{cond:slopes}) we may assume, without loss of generality, that $\liminf_{\e \to 0} \int_0^T|\partial \E^m_\e|(\mu_\e(t))^2\,dt$ is finite, so by Fatou's lemma
\begin{align} \label{subdiff bound pointwise}
	 \infty > \liminf_{\e \to 0} \int_0^T|\partial \E^m_\e|(\mu_\e(t))^2 \,dt\geq \int_0^T \liminf_{\e \to 0}  |\partial \E^m_\e|(\mu_\e(t))^2 \,dt, 
 \end{align}
so $\liminf_{\e\to0} |\partial \E^m_\e|(\mu_\e(t)) < \infty$ for almost every $t \in [0,T]$. In particular, up to taking subsequences, we may assume that, for almost every $t \in [0,T]$, $\{|\partial \E^m_\e|(\mu_\e(t))\}_\e$ is bounded uniformly in $\e>0$. By Corollary \ref{full subdiff char},
\bes
	|\partial \E_\e^m|(\mu_\e) = \left\|\grad V + \grad W*\mu_\e + \grad \frac{ \delta \F^m_\e}{\delta \mu_\e}(\mu_\e) \right\|_{L^2(\mu_\e;\Rd)} .
\ees
Furthermore, note that if 
\begin{align} \label{regularity of mu}
	\mu^{m} \in W^{1,1}(\Rd) \text{ and }\grad \mu^{m} +  \grad V \mu + (\grad W*\mu) \mu= \xi \mu \text{ for some }\xi \in L^2(\mu;\Rd) ,
\end{align}
then $|\partial \E^m|(\mu) = \| \xi \|_{L^2(\mu;\Rd)}$; c.f. \cite[Theorem 10.4.13]{AGS}. Thus, to prove (S\ref{cond:slopes}) it suffices to show that
\be\label{eq:ineq-slopes-Serfaty}
	\liminf_{\e \to 0} \int_0^T \int \left| \grad V + \grad W*(\mu_\e(t)) + \grad \frac{ \delta \F^m_\e}{\delta \mu_\e}(\mu_\e(t)) \right|^2 d \mu_\e(t) \,dt \geq   \int_0^T \int | \xi(t)  |^2 \,d \mu(t)\,dt  , 
\ee
when (\ref{regularity of mu}) holds for almost every $t \in[0,T]$. Furthermore, the inequality in \eqref{eq:ineq-slopes-Serfaty} is, by Proposition \ref{AGSthm}(\ref{weaklsc}), a consequence of
\begin{align}  \label{weak convergence of velocities}
\lim_{\e \to 0} \int_0^T \int f(t) \left( \grad V + \grad W*\mu_\e(s) + \grad \frac{ \partial \F^m_\e}{\partial \mu_\e}(s) \right) d \mu_\e (s)ds = \int_0^T \int f(t) \xi(t) \,d \mu(t) \,ds,
\end{align}
for all $f \in C_\mathrm{c}^\infty([0,T] \times\Rd)$. Observe that Proposition \ref{AGSthm} is stated for probability measures---we can easily rescale $d\mu_\e \otimes d\mathcal{L}^d$ to be a probability measure by diving the above equations by $T>0$.

First, we address the terms with the drift and interaction potentials $V$ and $W$. Combining Assumption \ref{convexityregularity} on $V$ and $W$ with Assumption (A\ref{extra assumptions b}) on $\mu_\e$ ensures that $|\grad V|$ is uniformly integrable in $d \mu_\e \otimes d\mathcal{L}^d$ and $(x,y) \mapsto |\grad W(x-y)|$ is uniformly integrable $d \mu_\e \otimes d \mu_\e \otimes d\mathcal{L}^d$.Therefore, by \cite[Lemma 5.1.7]{AGS}, $(\mu_\e)_\e$ converging weakly-$^*$ to $\mu$ ensures that 
 \[ 
 	\lim_{\e\to 0} \int_0^T \int f(t) \left( \grad V + \grad W*(\mu_\e(t))  \right) \,d \mu_\e(t)\,dt = \frac{m}{m-1} \int_0^T \ird f(t) \big( \grad V + \grad W*(\mu(t))\big)\,d\mu(t)\,dt. 
\]
Now we deal with proving the diffusion part of \eqref{regularity of mu} (that is, for almost every $t \in [0,T]$, we have $\mu(t)^m \in W^{1,1}(\Rd)$ and $\grad \mu(t)^m = \eta(t) \mu(t)$ for $\eta \in L^2(\mu;\Rd)$), and with proving that
\begin{align} \label{GFgoal1}
	\lim_{\e \to 0} \int_0^T \int f(t) \grad \frac{ \delta \F^m_\e}{\delta \mu_\e}(\mu_\e(t)) \,d \mu_\e(t)\,dt =  \int_0^T \int f(t) \eta(t) \,d \mu(t)\,dt ,
\end{align}
Recalling the abbreviation $p_\e := (\varphi_\e *\mu_\e)^{m-2} \mu_\e$, we rewrite the inner integral on the left-hand side of (\ref{GFgoal1}) as
\begin{align*}
 	\int f \grad \frac{\partial \F^m_\e}{\partial \mu_\e} \,d \mu_\e &= \int f \left( (\grad \varphi_\e* p_\e) + (\varphi_\e* \mu_\e)^{m-2} (\grad \varphi_\e * \mu_\e) \right) \,d \mu_\e \\
 &= \int  (\zeta_\e *(f \mu_\e)) (\grad\zeta_\e*p_\e) + (\zeta_\e*(f p_\e)) (\grad \zeta_\e*\mu_\e) \,d\mathcal{L}^d.
\end{align*}
Applying Lemma \ref{technical move convolution} together with \eqref{auxxx} and (A3), and integrating by parts, we obtain
\begin{align*}
	\lim_{\e \to 0} \int_0^T \int f(t)  \grad \frac{\delta \F^m_\e}{\delta \mu_\e}(\mu_\e(t)) \, d \mu_\e(t)\,dt &= \lim_{\e \to 0} \int_0^T \int f(t)( \zeta_\e * (\mu_\e(t))) (\grad\zeta_\e*(p_\e(t)))\,d\mathcal{L}^d\,dt\\
	&\phantom{{}={}} + \int_0^T \int f(t)(\zeta_\e* (p_\e(t))) (\grad \zeta_\e*(\mu_\e(t))) \,d\mathcal{L}^d\,dt \\
	& = - \lim_{\e \to 0}  \int_0^T \int \grad f(t) ( \zeta_\e * (\mu_\e(t))) (\zeta_\e*(p_\e(t))) \,d\mathcal{L}^d\,dt \\
	&=- \lim_{\e \to 0}  \int_0^T \int \zeta_\e*(\grad f(t) ( \zeta_\e * (\mu_\e(t)))) p_\e(t)\,d\mathcal{L}^d\,dt  .
\end{align*} 
Now we move $\grad f$ out of the convolution. By Lemma \ref{move mollifier prop}, there exists $p>0$ so
\begin{align*}
	\left| \int \zeta_\e*(\grad f ( \zeta_\e * \mu_\e)) p_\e\,d\mathcal{L}^d \right.-&\left. \int \grad f (\zeta_\e*(  \zeta_\e * \mu_\e)) p_\e \,d\mathcal{L}^d\right| \\
& \leq \e^{p}  \|\grad f \|_{L^\infty([0,T]\times\Rd)} \left( \int (\varphi_\e*\mu_\e)^{m-1} d \mu_\e + C_\zeta \int p_\e\,d\mathcal{L}^d \right)  \\
&  \leq \e^{p} \|\grad f \|_{L^\infty([0,T]\times\Rd)} \left( \F^m_\e(\mu_\e) + C_\zeta \F^m_\e(\mu_\e)^{(m-2)/(m-1)} \right),
\end{align*}
where we again use (\ref{pepsbd}). 
Using the inequality in \eqref{Jensen in t} and that $\{\int_0^T \F^m_\e(\mu_\e(t))\,dt\}_\e$ is uniformly bounded in $\epsilon$,
\begin{align} \label{almostdone}
	-\lim_{\e \to 0} \int_0^T \int f(t) \grad \frac{\delta \F^m_\e}{\delta \mu_\e}(\mu_\e(t)) \,d \mu_\e(t)\,dt &=  \lim_{\e \to 0} \int_0^T \int \grad f(t) (\varphi_\e*\mu_\e) p_\e\,d\mathcal{L}^d \,dt\\
	&= \lim_{\e \to 0} \int_0^T \ird \grad f(t) (\varphi_\e*\mu_\e(t))^{m-1} d \mu_\e(t)\,dt.
\end{align}

To conclude the proof, we aim to apply Proposition \ref{AGSthm}\eqref{strongcty}, and we begin by verifying the hypotheses of this proposition. First, note that since $\zeta_\e*\mu_\e \to \mu$ in $L^1([0,T];L^m_\loc(\Rd))$ for $m \geq 2$ as $\e\to0$, we also have $\zeta_\e*\mu_\e \to \mu$ in $L^1([0,T];L^2_\loc(\Rd))$. Let $w_\e = \varphi_\e*\mu_\e$. By definition, $\int w_\e d \mu_\e = \int (\zeta_\e*\mu_\e)^2\,d\mathcal{L}^d$. Thus, Assumption (A\ref{extra assumptions a}) and the fact that $\zeta_\e*\mu_\e(\Rd)= 1$ imply
\[ 
	\sup_{\e >0} \int_0^T \int |\zeta_\e*\mu_\e(t)|^2\,d\mathcal{L}^d\,dt < \infty,
\]
so that $w_\e \in L^1([0,T],L^1(\mu_\e;\Rd))$. Furthermore, for any $h \in L^\infty([0,T];W^{1,\infty}(\Rd))$, the mollifier exchange lemma \ref{move mollifier prop} and the convergence of $\zeta_\e*\mu_\e$ to $\mu$ in $L^1([0,T];L^2_\loc(\Rd))$ give
\begin{align} \label{weak convergence of w eps}
	\int_0^T \int h(t) w_\e(t) \,d \mu_\e(t) &= \int_0^T \int \zeta_\e*(h \mu_\e(t)) \,d\zeta_\e *(\mu_\e(t))\,dt \nonumber\\
	& = \int_0^T \int h(t) (\zeta_\e*\mu_\e(t))^2\,d\mathcal{L}^d\,dt\\
	&\phantom{{}={}}+ \e^p \|\grad h\|_{ L^\infty([0,T];W^{1,\infty}(\Rd))} \left( \int_0^T \int\|\zeta_\e*(\mu_\e(t))\|_{L^2(\Rd)}^2\,d\mathcal{L}^d\,dt + C_\zeta \right)\nonumber\\
	 &\longrightarrow \int_0^T \int h(t) \mu(t)^2 \,d\mathcal{L}^d\,dt, \nonumber
\end{align}
as $\e\to0$. Thus, $w_\e \in L^1([0,T];L^1(\mu_\e;\Rd))$ converges weakly to $\mu \in L^1([0,T];L^1(d \mu))$ in the sense of Definition \ref{weakvaryingdef} as $\e\to0$. As before, while this definition is stated for probability measures, we can easily rescale $d\mu_\e \otimes d\mathcal{L}^d$ to be a probability measure by diving the above equations by $T>0$.

We now seek to show that, for all $g \in C_\mathrm{c}^\infty([0,T]\times\Rd)$,
\[ 
	\lim_{\e \to 0} \int_0^T \int g(t) |w_\e(t)|^{m-1} \,d \mu_\e(t)\,dt = \int_0^T \int g(t) |\mu(t)|^{m-1} \,d \mu(t) . 
\]
When $m=2$, this follows from equation (\ref{weak convergence of w eps}). Suppose $m>2$. Let $\kappa\: \Rd\to\R$ be a smooth cutoff function with $0 \leq \kappa \leq 1$, $\|\grad \kappa\|_{L^\infty(\Rd)} \leq 1$, $\|D^2 \kappa\|_{L^\infty(\Rd)} \leq 4$, $\kappa(x) = 1$ for all $|x| < 1/2$ and $\kappa(x) = 0$ for all $|x| > 2$. Given $R>0$, define $\kappa_R := \kappa(\cdot/R)$, so that $\|\grad \kappa_R\|_{L^\infty(\Rd)} \leq 1/R$. 
Then, by Jensen's inequality for the convex function $s \mapsto s^{m-1}$, Lemma \ref{move mollifier prop}, and Assumption (A\ref{extra assumptions a}), 
\begin{align*}
	\limsup_{\e \to 0} \int_0^T \int  | \kappa_R w_\e(t)|^{m-1} \,d \mu_\e(t)\,d\mathcal{L}^d\,dt &\leq  \limsup_{\e \to 0} \int_0^T \int (\zeta_\e*(\mu_\e(t)))^{m-1} \zeta_\e*(\kappa_R^{m-1} \mu_\e(t))\,d\mathcal{L}^d\,dt \\
	&\leq \limsup_{\e \to 0} \int_0^T \int \kappa_R^{m-1} (\zeta_\e*(\mu_\e(t)))^m \,d\mathcal{L}^d\,dt\\
	&= \int_0^T \int (\kappa_R \mu(t))^{m-1} \,d \mu(t)\,dt .
\end{align*}
Combining this with (\ref{weak convergence of w eps}), where we may choose $h = \kappa_R g$ for any $g \in C_\mathrm{c}^\infty(\Rd)$, we have that $(\kappa_R w_\e)_\e$ converges strongly in $L^{m-1}(\mu_\e;\Rd)$ to $\kappa_R \mu \in L^{m-1}(\mu;\Rd)$ as $\e\to0$, in the sense of Definition \ref{weakvaryingdef}. Finally, we may apply Proposition \ref{AGSthm}\eqref{strongcty} to conclude that for all $g \in C_\mathrm{c}^\infty([0,T]\times\Rd)$,
\[ \lim_{\e \to 0} \int_0^t \ird g  | \kappa_R w_\e|^{m-1} d \mu_\e = \int_0^t \ird g  |\kappa_R \mu|^{m-1} d \mu . \]
Taking $g = \grad f$, choosing $R>1$ so that $\kappa_R \equiv 1$ on the support of $\grad f$, and combining the above equation with equation (\ref{almostdone}), we obtain
\begin{align} \label{finalequalities1}
	\lim_{\e \to 0} \int_0^T \int f(t) \grad \frac{\delta \F^m_\e}{\delta \mu_\e}(\mu_\e(t)) \,d \mu_\e(t)\,dt = - \int_0^T \int \grad f(t) \mu(t)^m \,d\mathcal{L}^d\,dt.
\end{align}

We now prove that $\mu$ has the necessary regularity. In particular, we show that for almost every $t \in [0,T]$, we have $\mu^m \in W^{1,1}(\Rd)$ and $\grad \mu^m = \eta \mu$ for $\eta \in L^2(\mu;\Rd)$. Inequality (\ref{subdiff bound pointwise}) ensures that, up to subsequences $\{\int_0^t |\partial \F^m_\e|^2(\mu_\e(t))\,dt\}_\e$ is bounded uniformly in $\e>0$. Thus, by H\"older's inequality, there exists $C>0$ so that 
\begin{align*}
	C> \int_0^T \left\|\grad \frac{\delta \F^m_\e}{\delta \mu_\e}(\mu_\e(t)) \right\|^2_{L^2(\mu_\e;\Rd)} \,dt&\geq \int_0^T \left\|\grad \frac{\delta \F^m_\e}{\delta \mu_\e}(\mu_\e(t)) \right\|^2_{L^1(\mu_\e;\Rd)} \,dt\\
	&\geq T\left( \frac{1}{T} \int_0^T  \left\|\grad \frac{\delta \F^m_\e}{\delta \mu_\e}(\mu_\e(t)) \right\|_{L^1(\mu_\e;\Rd)} \,dt \right)^{2}, 
\end{align*}
for all $\e>0$. Combining this with \eqref{finalequalities1} gives
\[
	C T \|f \|_{L^\infty(\Rd)}  \geq  \limsup_{\e \to 0} \|f \|_{L^\infty([0,T]\times\Rd)} \int_0^T \norm{\grad \frac{\delta \F^m_\e}{\delta \mu_\e}(\mu_\e(t))}_{L^1(\mu_\e;\Rd)}  \geq \int_0^T \int f(t) \grad (\mu(t)^m) \,d\mathcal{L}^d\,dt.\]
Hence $\grad (\mu^m)$ has finite measure on $[0,T]\times \Rd$, so we may rewrite (\ref{finalequalities1}) as
\begin{align} \label{finalequalities2}
	\lim_{\e \to 0} \int_0^t \int f \grad \frac{\delta \F^m_\e}{\delta \mu_\e}(\mu_\e(t)) \,d \mu_\e(t)\,dt = - \int_0^t \int  f(t) \,d  \grad (\mu(t)^m) \,dt.
\end{align}
By another application of H\"older's inequality, this guarantees
\begin{align*}
	\sqrt{C}  \left( \int_0^t \|f(t) \|^2_{L^2(\mu;\Rd)} \,dt\right)^{1/2} &\geq  \limsup_{\e \to 0}\int_0^t \|f(t) \|_{L^2(\mu_\e;\Rd)}  \left\| \grad \frac{\delta \F^m_\e}{\delta \mu_\e}(\mu_\e(t)) \right\|_{L^2(\mu_\e;\Rd)}\\
	& \geq \int_0^t \int f(t)  d\grad (\mu(t)^m) \,dt.
\end{align*}
Riesz representation theorem then ensures that there exists $\eta \in L^2([0,t];L^2(\mu;\Rd))$ so that $\eta \mu = \grad (\mu^m)$. In particular, this implies $\grad( \mu(t)^m) \in L^1(\Rd)$ for almost every $t \in [0,T]$, so $\mu^m \in W^{1,1}(\Rd)$ for almost every $t \in [0,T]$ and we may rewrite (\ref{finalequalities2}) as
\bes
	\lim_{\e \to 0} \int_0^T \ird f(t) \grad \frac{\delta \F^m_\e}{\delta \mu_\e}(\mu_\e(t))\,d \mu_\e(t)\,dt = - \int_0^T \int  f(t) \eta \,d \mu(t)\,dt,
\ees
which completes the proof. 
\end{proof}

We conclude this section by showing that, in the case when $m=2$ and for $V, W \in C^2(\Rd)$ with bounded Hessians, whenever the initial data of the gradient flows have finite second moments and internal energies, we automatically obtain Assumptions (A\ref{extra assumptions b})--(A\ref{extra assumptions a}). Consequently, in this special case, we are able to conclude the convergence of the gradient flows without these additional assumptions.
\begin{cor} \label{m2 theorem}
	Let $\e>0$ and $m=2$. In addition to satisfying Assumption \ref{convexityregularity}, assume that $V,W \in C^2(\Rd)$ have bounded Hessians $D^2 V$ and $D^2 W$. Fix $T>0$, and suppose $\mu_\e \in AC^2([0,T];\P_2(\R^d))$ is a gradient flow of $\E_\e^m$ satisfying 
\begin{align} 
	&\mu_\e(0) \wsto \mu(0) , \quad  \lim_{\e \to 0} \E_\e^m(\mu_\e(0)) = \E^m(\mu(0)), \quad \mu(0) \in D(\E^m), \label{initial data convergence} \\
& \sup_{\e >0 } M_2(\mu_\e(0)) < \infty, \quad \sup_{\e >0 } \textstyle \int \mu_\e(0) \log(\mu_\e(0))\,d\mathcal{L}^d < +\infty. \label{bdd entropy}
\end{align}
Then, there exists $\mu \in AC^2([0,T];\P_2(\R^d))$ such that  
\[ 
	\mu_\e(t) \wsto \mu(t) \text{ and }\zeta_\e *\mu_\e(t) \xrightarrow{L^2(\Rd)} \mu(t) \text{ for all } t \in [0,T] ,
\] 
and $\mu$ is the gradient flow of $\E^m$ with initial data $\mu(0)$.
\end{cor}

\begin{remark}[Previous work, $m=2$]
The above theorem generalizes a result by Lions and Mas-Gallic \cite{LionsMasGallic} on a numerical scheme for the porous medium equation $\partial_t \mu = \Delta \mu^2$ on a bounded domain with periodic boundary conditions to equations of the form (\ref{W2PDE1}) on Euclidean space.
\end{remark}

\begin{proof}[Proof of Corollary \ref{m2 theorem}]
First, we show that $\sup_{\e >0} \| \zeta_\e*(\mu_\e(0))\|_{L^2(\Rd)} < \infty$. The fact that $D^2V$ and $D^2 W$ are bounded ensures $|V|$ and $|W|$ grow at most quadratically. Combining this with equations \eqref{initial data convergence}--\eqref{bdd entropy}, which ensure $\{\E^m_\e(\mu(0))\}_\e$ and $\{M_2(\mu_\e(0))\}_\e$ are bounded uniformly in $\e>0$, we obtain
\begin{align*}
	\sup_{\e >0} \|\zeta_\e*(\mu_\e(0))\|_{L^2(\Rd)}^2&= \sup_{\e >0} \F^2_\e(\mu_\e(0))\\
	& = \sup_{\e >0} \left(\E_\e^2(\mu_\e(0)) -\textstyle \int V d \mu_\e(0) - \frac12 \int W*(\mu_\e(0)) d \mu_\e(0)\right) < +\infty .
\end{align*}
Furthermore, since the energy $\F^2_\e$ decreases along solutions to the gradient flow, we have
\begin{align} \label{m2 energy bound} \sup_{\e >0} \|\zeta_\e*(\mu_\e(t))\|_{L^2(\Rd)}^2 \leq \sup_{\e >0} \|\zeta_\e*(\mu_\e(0))\|_{L^2(\Rd)}^2 < \infty \quad \text{ for all } t\in[0,T] . 
\end{align}

Next, we show that our assumption that the initial data has bounded entropy (\ref{bdd entropy}) ensures $\int_0^t \|\grad \zeta_\e*(\mu_\e(s))\|_{L^2(\Rd)}^2\,ds < C(1+T)+M_2(\mu_\e(t))$ for all $t \in [0,T]$, for some $C>0$ depending on $d$, $V$, $W$ and $\sup_{\e >0} \int \log \mu_\e(0)\,d\mu_\e(0)$. Formally differentiating the entropy $\F^1(\mu) = \int \log(\mu)\,d\mu$ along the gradient flows $\mu_\e$, we expect that, for all $t \in [0,T]$,
\[ 
	\frac{d}{dt} \left[\F^1(\mu_\e(t))\right] = -2 \int |\grad \zeta_\e*(\mu_\e(t))|^2\,d\mathcal{L}^d + \int \Delta V \,d\mu_\e(t) + \int  \Delta W* (\mu_\e(t)) \,d\mu_\e(t) . 
\]
Hence, for any $t \in[0,T]$,
\begin{align*}
	\F^1(\mu_\e(t)) - \F^1(\mu_\e(0)) &= -2 \int_0^t \int |\grad \zeta_\e*(\mu_\e(s))|^2\,d\mathcal{L}^d\,ds\\
	&\phantom{{}={}} + \int_0^t \int \Delta V \,d\mu_\e(s)\,ds + \int_0^t \int \Delta W* (\mu_\e(s)) \,d\mu_\e(s)\,ds \\
&\leq - 2 \int_0^t \int |\grad \zeta_\e*(\mu_\e(s))|^2\,d\mathcal{L}^d\,ds +t\left( \|D^2 V\|_{L^\infty(\Rd)} + \|D^2 W\|_{L^\infty(\Rd)}\right)
\end{align*}
This computation can be made rigorous by first proving the analogous inequality along \emph{discrete time gradient flows} using the flow interchange method of Matthes, McCann, and Savar\'e \cite[Theorem 3.2]{MMS} and then sending the timestep to zero to recover the above inequality in continuous time. Thus, there exists $K_0>0$ depending on $V, W$ and $\sup_{\e >0} \F^1(\mu_\e(0)) $ so that, for all $t\in[0,T]$,
\[ 
	\int_0^t \| \grad \zeta_\e*(\mu_\e(s)) \|_{L^2(\Rd)}^2\,ds \leq -\F^1(\mu_\e(t)) + K_0(1+t) .\]
Finally, by a Carleman-type estimate \cite[Lemma 4.1]{CPSW}, we have $\F^1(\nu) \geq -(2 \pi)^{d/2} - M_2(\nu)$ for any $\nu \in \P_2(\Rd)$. Therefore,
\begin{align} \label{H1 to second moment}
	\int_0^t \| \grad \zeta_\e*(\mu_\e(s)) \|_{L^2(\Rd)}^2\,ds \leq M_2(\mu_\e(t)) + C(1+t) .
\end{align}
Now, we  use this estimate to show that $\{M_2(\mu_\e(t))\}_\e$ is uniformly bounded in $\e$ for all $t \in[0,T]$. Let $\kappa$ be a smooth cutoff function with $0 \leq \kappa \leq 1$, $\|\grad \kappa\|_{L^\infty(\Rd)} \leq 1$, $\|D^2 \kappa\|_{L^\infty(\Rd)} \leq 4$, $\kappa(x) = 1$ for all $|x| < 1/2$ and $\kappa(x) = 0$ for all $|x| > 2$. Given $R>0$, define $\kappa_R(x) = \kappa(x/R)$, so that $\|\grad \kappa_R\|_{L^\infty(\Rd)} \leq 1/R$ and $\|D^2 \kappa_R\|_{L^\infty(\Rd)} \leq 4/R^2$. Then there exists $C_\kappa>0$ so that for all $R>1$, $|\grad (\kappa_R(x) x^2)| \leq C_\kappa|x|$ and $|D^2 (\kappa_R(x) x^2)| \leq C_\kappa$ for all $x\in\Rd$. By Proposition \ref{characterizationGF}, $\mu_\e$ is a weak solution of the continuity equation. Therefore choosing $\kappa_R(x) |x|^2$ as our test function, we obtain, for all $t \in[0,T]$,
\begin{align*}
 &\ird \kappa_R(x)|x|^2 \,d \mu_\e(t,x)-  \ird \kappa_R(x) |x|^2 \,d \mu_\e(0,x)\\
 &= -2  \int_0^t \ird \grad (\kappa_R(x) x^2)\left(\grad \varphi_\e*(\mu_\e(s)) + \grad V(x) + \grad W*(\mu_\e(s))(x) \right) \,d \mu_\e(s,x) .
\end{align*}
Since $D^2 V$ and $D^2 W$ are bounded, $|\grad V|$ and $|\grad W|$ grow at most linearly. Consequently, there exists $C'>0$, depending on $V$, $W$, and $C_\kappa$ so that
\[
	-2 \int_0^t \ird \grad (\kappa_R(x) x^2) (\grad V(x) + \grad W*(\mu_\e(s))(x) \,d \mu_\e(s,x) \leq C' \left( 1 + \int_0^t M_2(\mu_\e(s))\,ds \right) . 
\]
Likewise, by Lemma \ref{technical move convolution}, there exists $r>0$ so that, for all $t\in[0,T]$,
\begin{align*} 
	&-2\int_0^t \ird \grad (\kappa_R(x) x^2) \grad \varphi_\e*(\mu_\e(s))(x) \,d \mu_\e(s,x)\\
	& = -2 \int_0^t\ird \zeta_\e*(\grad (\kappa_R(x) x^2) \mu_\e(s)) \grad \zeta_\e*(\mu_\e(s))(x)\,dx\,ds \\
	& \leq -2\int_0^t \ird \grad (\kappa_R(x) x^2)\zeta_\e*(\mu_\e(s))(x)\grad \zeta_\e*(\mu_\e(s))(x)\,dx\,ds\\
	&\phantom{{}={}}  + \e^r C_\kappa\left(\int_0^t\| \mu_\e(s)\|_{BV_\e^m}\,ds + 2 t \|\grad \zeta\|_{L^1(\Rd)} \right) \\
	& = \int_0^t \ird \Delta (\kappa_R(x) x^2)(\zeta_\e*(\mu_\e(s)))^2\,dx\,ds + \e^r C_\kappa\left(\int_0^t\| \mu_\e(s)\|_{BV_\e^m}\,ds + 2 t \|\grad \zeta\|_{L^1(\Rd)} \right) \\
	& \leq C_\kappa \int_0^t \F_\e^2(\mu_\e(s))\,ds + 2\e^r C_\kappa\left(\int_0^t\| \zeta_\e* (\mu_\e(s))\|_{L^2(\Rd)} \|\grad \zeta_\e*(\mu_\e(s))\|_{L^2(\Rd)}\,ds + 2 t\|\grad \zeta\|_{L^1(\Rd)} \right) \\
& \leq C_\kappa t \F_\e^2(\mu_\e(0)) + 2\e^r C_\kappa\left(\sqrt{t \F_\e^2(\mu_\e(0))}\sqrt{ M_2(\mu_\e(t)) + C(1+t) }+ 2 t \|\grad \zeta\|_{L^1(\Rd)} \right) \\
&\leq C'' \left(1+ t +\e^r M_2(\mu_\e(t))\right)
\end{align*}
for $C''$ depending on $C_\kappa$, $\sup_{\e >0} \F_\e^2(\mu_\e(0))$, and $\|\grad \zeta\|_{L^1(\Rd)}$. In the second inequality, we use that 
\begin{align} \label{BV bound}
\| \mu_\e\|_{BV_\e^m} \leq 2 \|(\grad \zeta_\e*\mu_\e)(\zeta_\e*\mu_\e)\|_{L^1(\Rd)}\leq \| \zeta_\e* \mu_\e\|_{L^2(\Rd)} \|\grad \zeta_\e*\mu_\e\|_{L^2(\Rd)}
\end{align}
Therefore, there exists $C''>0$ so that, for all $t \in[0,T]$,
\begin{align*}
	\ird \kappa_R(x) |x|^2 \,d \mu_\e(t,x)  \leq M_2(\mu_\e(0)) + C' \left( t + \int_0^t M_2(\mu_\e(s))\,ds \right)+ C'' \left(1+ t +\e^r M_2(\mu_\e(t))\right).
\end{align*}
As the right-hand side is independent of $R>1$, by sending $R \to +\infty$ by the dominated convergence theorem we obtain that for $\e^r<1/(2C'')$, 
\[ 
	M_2(\mu_\e(t)) \leq 2C' \left(t + \int_0^t M_2(\mu_\e(s))\,ds \right) + 2C''(t+1).
\]
Therefore, by Gronwall's inequality, there exists $\tilde{C}$ depending on $C'$, $C''$ and $T$ (and independent of $\e$) so that 
\begin{align} \label{m2 second moment} 
	M_2(\mu_\e(t))< \tilde{C}  \quad \text{ for all }t \in [0,T].
\end{align}
We may combine this with the inequality in \eqref{H1 to second moment} to obtain, for all $t\in[0,T]$,
\begin{align} \label{m2 H1}
	\int_0^t \| \grad \zeta_\e*(\mu_\e(s))\|_{L^2(\Rd)}^2\,ds \leq \tilde{C} + C(1+t) \quad \text{ for } t \in [0,T].
\end{align}

We now use these results to verify the assumptions of Theorem \ref{Gamma GF theorem} hold, so that we may apply this result to conclude convergence of the gradient flows. Assumption (A\ref{extra assumptions b}) is a consequence of the inequality in \eqref{m2 second moment}. Assumption (A\ref{extra assumptions c}) is a consequence of the inequalities in \eqref{m2 energy bound}, \eqref{BV bound} and \eqref{m2 H1}. 

It remains to show Assumption (A\ref{extra assumptions a}). First, note that since $\sup_{\e >0} \|\zeta_\e*\mu_\e\|_{L^\infty([0,T]\times \Rd)} < \infty$, every subsequence of $(\zeta_\e*\mu_\e)_\e$ has a further subsequence, which we also denote by $(\zeta_\e*\mu_\e)_\e$, that converges weakly in $L^2([0,T]\times \Rd)$ to some $\nu$ as $\e\to0$, and for which $\zeta_\e*\mu_\e(t) \rightharpoonup \nu(t)$ weakly in $L^2(\Rd)$ for all $t \in [0,T]$. By uniqueness of limits and \eqref{initial data convergence}, we have $\nu(0) = \mu(0)$ almost everywhere.

Next, note that \eqref{m2 energy bound} and \eqref{m2 H1} ensure that $\sup_{\e>0}\|\zeta_\e*\mu\|_{L^2([0,T];H^1(\Rd))}<\infty$. In particular we have $\sup_{\e>0}\|\kappa_R \zeta_\e*\mu\|_{L^2([0,T];H^1(\Rd))}<\infty$ for the smooth cutoff function $\kappa_R$, $R>1$. Therefore, by the Rellich--Kondrachov Theorem (c.f. \cite[Section 5.7]{Evans}), for almost every $t \in [0,T]$, up to another subsequence,  $(\kappa_R \zeta_\e*\mu_\e(t))_\e$ converges strongly in $L^2(\Rd)$ to some $\nu_R(t)$. In particular, for any $f \in C_\mathrm{c}^\infty(B_{R/2}(0))$,
\[
	\int f \,d\nu(t) =  \lim_{ \e \to 0} \int f \,d\zeta_\e *\mu_\e(t) = \int f\,d \nu_R(t) \quad \text{ for all } t \in[0,T],
\]
so $\nu = \nu_R$ almost everywhere in $B_{R/2}(0)$. Since $R>1$ is arbitrary, this shows that for all $t \in [0,T]$, $\zeta_\e*\mu_\e(t) \to \nu(t)$ strongly in $L^2_\loc(\Rd)$. Finally, using again that $\{\|\zeta_\e*\mu_\e(t)\|_{L^2(\Rd)}\}_t$ is bounded uniformly in $t \in[0,T]$, the dominated convergence theorem ensures that $\zeta_\e*\mu_\e(t) \to \nu(t)$ in $L^1([0,T];L^2_\loc(\Rd)$ as $\e\to0$. This completes the proof of assumption (A\ref{extra assumptions a}).

As we have now verified the conditions of Theorem \ref{Gamma GF theorem}, we now conclude that $\mu_\e(t) \wsto \nu(t)$ for almost every $t \in [0,T]$, for some $\nu \in AC^2([0,T];\P_2(\R^d))$ which is the gradient flow of $\E^2$ with initial data $\mu(0)$. By Proposition \ref{wellposedness}, the gradient flow of $\E^2$ with initial data $\mu(0)$ is unique. Thus, since any subsequence of $(\mu_\e)_\e$ has a further subsequence which converges to $\nu$, the full sequence must converge to $\mu$, which gives the result.
\end{proof}


\section{Numerical results} \label{numerics section}

\subsection{Numerical method and convergence} \label{numericalmethod}
We now apply the theory of regularized gradient flows developed in the previous sections to develop a blob method for diffusion, allowing us to numerically simulate solutions to partial differential equations of Wasserstein gradient flow type (\ref{W2PDE1}). We begin by describing the details of our numerical scheme and applying Theorem \ref{Gamma GF theorem} to prove its convergence, under suitable regularity assumptions. 

\begin{thm} \label{numerics convergence}
Assume $m\geq2$ and $V$ and $W$ satisfy Assumption \ref{convexityregularity}.
Suppose $\mu(0) \in D(\E^m)$ is compactly supported in $B_R(0)$, the ball of radius $R$ centered at the origin. For fixed grid spacing $h>0$, define the grid indices $Q_R^h := \{ i \in \mathbb{Z}^d  : |ih| \leq R \}$ and approximate $\mu(0)$ by the following sequence of measures:
\begin{align} \label{initial data discretization}
 \mu_\e(0) := \sum_{i \in Q_R^h} \delta_{ih} m_i , \quad m_i = \int_{Q_i} \,d\mu(0),\;i \in Q_R^h,
\end{align}
where $Q_i$ is the cube centered at $ih$ of side length $h$. Next, for $\e>0$, define the evolution of these measures by
\begin{align} \label{particle gradient flow solution}
\mu_\e(t) = \sum_{i \in Q_R^h} \delta_{X_i(t)} m_i, \quad t\in[0,T],
\end{align}
where $\{X_i(t)\}_{i \in Q_R^h}$ are solutions to the ODE system \eqref{ODEsystem} on a time interval $[0,T]$ with initial data $X_i(0) = ih$.
If $h = o(\e)$ as $\e \to 0$ and Assumptions (A\ref{extra assumptions b})--(A\ref{extra assumptions a}) from Theorem \ref{Gamma GF theorem} hold, then $(\mu_\e(t))_\e$ converges in the weak-$^*$ topology to $\mu(t)$ as $\e\to0$ for almost every $t \in [0,T]$, where $\mu(t)$ is the unique solution of \eqref{W2PDE1} with initial datum $\mu(0)$.
\end{thm}
\begin{proof}
By Corollary \ref{particles well posed}, $\mu_\e \in AC^2([0,T];\P_2(\R^d))$ is the gradient flow of $\E^m_\e$ with initial condition $\mu_\e(0)$ for all $\e >0$. To apply Theorem \ref{Gamma GF theorem} and obtain the result, it remains to show that Assumption \eqref{A0} holds. In particular, we must show that, assuming $h = o(\e)$,
\begin{align*}
	&\lim_{\e \to 0} \left(\int V \,d\mu_\e(0)  + \frac{1}{2} \int (W*(\mu_\e(0)))\,d \mu_\e(0) +  \F^m_\e(\mu_\e(0)) \right)\\
	&= \int V \,d\mu(0)  + \frac{1}{2}  \int (W*(\mu(0)))\,d \mu(0) +  \F^m(\mu(0)) . 
\end{align*}
Define $T: \Rd \to \Rd$ by $T(y) = ih$ for $y \in Q_i$ and $i \in Q_R^h$. Then $T$ is a transport map from $\mu(0)$ to $\mu_\e(0)$ and $|T(y)-y| \leq h$ for all $y \in \Rd$.
By construction, 
\[ 
W_2(\mu_\e(0),\mu(0)) \leq \left\{ \left| T(y)  - y \right| \mid y \in \supp \mu(0) \right\} \leq h ,
\]
so $\mu_\e(0) \wsto \mu(0)$ as $\e \to 0$ (and so, as $h\to0$). Likewise, for all $\e,h>0$, $\supp \mu_\e(0) \subseteq B_R(0)$. Consequently, since $V$ and $W$ are continuous,
\[ 
\lim_{\e \to 0} \int V d\mu_\e(0)  + \frac{1}{2} \int (W*(\mu_\e(0)))\,d \mu_\e(0) = \int V d\mu(0)  + \frac{1}{2}  \int (W*(\mu(0)))\,d \mu(0) . 
\]

Thus, it remains to show that 
\[ 
\lim_{\e\to 0}  \F^m_\e(\mu_\e(0)) =  \F^m(\mu(0)). 
\] 
By Theorem \ref{Gamma convergence theorem2}, we have that $\liminf_{\e\to 0} \F^m_\e(\mu_\e(0)) \geq \F^m(\mu_\e(0)).$ By Proposition \ref{relative sizes lemma}, for all $\e>0$ we have
\[ 
\F^m_\e(\mu_\e(0)) \leq \F^m(\mu_\e(0)) = \tfrac{1}{m-1} \|\zeta_\e*\mu_\e\|_{L^m(\R^d)}^m . 
\]
Consequently, to show that $\limsup_{\e \to 0} \F^m_\e(\mu_\e(0)) \leq \F^m(\mu(0)) = \|\mu(0)\|_{L^m(\R^d)}^m/(m-1)$, it suffices to show that $\zeta_\e*\mu_\e(0) \to \mu(0)$ in $L^m$ as $\e\to0$. 

For simplicity of notation, we suppress the dependence on time and show $\zeta_\e*\mu_\e \to \mu$ in $L^m$ as $\e\to0$. By the assumptions that $\mu \in D(\E^m)$ with compact support and $V$ and $W$ are continuous, we have $\mu \in L^m(\Rd)$. Consequently $\zeta_\e*\mu \to \mu$ in $L^m$ as $\e\to0$, and it is enough to show that $\zeta_\e*\mu_\e - \zeta_\e*\mu \to 0$ in $L^m$. Using that $T$ is a transport map from $\mu_\e$ to $\mu$,
\begin{align*}
|\zeta_\e*\mu_\e(x) - \zeta_\e*\mu(x)| &= \left| \ird \zeta_\e(x-T(y)) -  \zeta_\e(x-y)  \,d\mu(y) \right| \\ 
&\leq \int_0^1 \ird \left| \grad \zeta_\e(x- (1-\alpha) T(y) - \alpha y ) \right||T(y)-y| \,d\mu(y) d \alpha \\
&\leq h \int_0^1 \ird \left| \grad \zeta_\e(x- (1-\alpha) T(y) - \alpha y ) \right| \,d\mu(y) d \alpha.
\end{align*}

Combining the decay of $\grad \zeta$ from Assumption \ref{mollifierAssumption} with the fact that $\grad \zeta$ is continuous, there exists $C>0$ so that $|\grad \zeta(x)| \leq C (1_{B}(x) +  |x|^{-q'}1_{\Rd \setminus B}(x))$, where $B = B_1(0)$ is the unit ball centered at the origin.
Note that if $|x-y| \geq 2h$, then for all $\alpha \in [0,1]$, $|x- (1-\alpha) T(y) - \alpha y| \geq |x-y| -h \geq |x-y|/2$ and $|x- (1-\alpha) T(y) - \alpha y| \leq 3|x-y|/2$. Thus, by the assumptions on our mollifier, we have
\begin{align*}
 &\left| \grad \zeta_\e(x- (1-\alpha) T(y) - \alpha y ) \right| \\
  &\quad \leq  \frac{C}{ \e^{d+1}} \Bigg[ 1_B\left(\frac{x- (1-\alpha) T(y) - \alpha y}{\e} \right)\\
  &\qquad  +  \e^{q'} \left|x- (1-\alpha) T(y) - \alpha y  \right|^{-q'} 1_{\Rd \setminus B} \left(\frac{x- (1-\alpha) T(y) - \alpha y}{\e} \right) \Bigg] \\
 &\quad \leq  \frac{C}{ \e^{d+1}} \left( 1_B\left(\frac{|x-y|}{2\e} \right)  + \left(  \frac{2\e}{3} \right)^{q'} \left|x -y  \right|^{-q'} 1_{B \setminus \Rd} \left(\frac{3|x-y|}{2\e} \right)  \right) .
\end{align*}

Thus, taking the $L^m$-norm with respect to $x$, doing a change of variables, and applying Minkowski's inequality, we obtain
\begin{align*}
&\|\zeta_\e*\mu_\e-\zeta_\e*\mu\|_{L^m(\R^d)}\\
& \leq h \|\grad \zeta_\e\|_\infty \left\|  \int_{B_{2h}(x)} \,\mu(y) \right\|_{m} \\
&\qquad + \frac{C h}{\e^{d+1}} \left\| \int_{B_{2h}(x)^c}\left( 1_B\left(\frac{|x-y|}{2\e} \right)  + \left(  \frac{2\e}{3} \right)^{q'} \left|x -y  \right|^{-q'} 1_{B \setminus \Rd} \left(\frac{3|x-y|}{2\e} \right) \right) \,d\mu(y) \right\|_{m} \\
&=  h \|\grad \zeta_\e\|_\infty \left\|  \int_{B_{2h}(0)} \,\mu(x-w)dw \right\|_{m} \\
&\qquad + \frac{C h}{\e^{d+1}} \left\| \int_{B_{2h}(0)^c}\left( 1_B\left(\frac{|w|}{2\e} \right)  +  \left(  \frac{2\e}{3} \right)^{q'} \left|w  \right|^{-q'} 1_{B \setminus \Rd} \left(\frac{3|w|}{2\e} \right) \right) \,\mu(x-w) dw \right\|_{m} \\
&\leq c   \| \mu\|_{m} \left( \frac{h^{d+1}}{\e^{d+1}} +\frac{h}{\e} \right) ,
\end{align*}
where $c>0$ depends on $C, \|\grad \zeta\|_\infty$, and the space dimension. 
Therefore, provided that $h = o(\e)$ as $\e \to 0$, we obtain that $\zeta_\e*\mu_\e - \zeta_\e*\mu \to 0$ in $L^m$.
\end{proof}

\begin{remark}[compact support of initial data]
In Theorem \ref{numerics convergence}, we assume that the initial datum of the exact solution $\mu(0) \in D(\E^m)$ is compactly supported. More generally, under the same assumptions on $V$, $W$, and $m$, given any $\nu_0 \in D(\E^m) \cap \P_2(\Rd)$ without compact support, 
there exists $\tilde{\nu}_0 \in D(\E^m)$ with compact support such that $\nu_0$ and $\tilde{\nu}_0$ are arbitrarily close in the Wasserstein distance. Furthermore, by the contraction inequality for gradient flows of $\E^m$, the solution $\nu$ with initial data $\nu_0$ and the solution $\tilde{ \nu}$ with initial data $\tilde{\nu}_0$ satisfy
\[ W_2(\nu(t),\tilde{\nu}(t)) \leq C W_2(\nu_0,\tilde{\nu}_0) \text{ for all }t \in [0,T], \]
where $C>0$ depends on $T$ and the semiconvexity of $V$ and $W$ \cite[Theorem 11.2.1]{AGS}. In this way, any solution of (\ref{W2PDE1}) with initial datum in $ D(\E^m) \cap \P_2(\Rd)$ can be approximated by a solution with compactly supported initial datum, so that our assumption of compact support in Theorem \ref{numerics convergence} is not restrictive.
\end{remark}

\begin{remark}[Assumptions (A\ref{extra assumptions b})--(A\ref{extra assumptions a})]  \label{justifyassumptions}
In Theorem \ref{numericalmethod}, we proved that, as long as Assumptions (A\ref{extra assumptions b})--(A\ref{extra assumptions a}) from Theorem \ref{Gamma GF theorem} hold along the particle solutions $\{\mu_\e\}_\e$, then any limit of these particle solutions must be the corresponding gradient flow of the unregularized energy.
Verifying these conditions analytically can be challenging; see Theorem \ref{m2 theorem}. However, numerical results can provide confidence that these conditions hold along a given particle approximation.

A sufficient condition for Assumption (A\ref{extra assumptions b}) is that the $(m-1)$th moment of the particle solution 
\[ \sum_{i \in Q_R^h} |X_i(t)|^{m-1} m_i \]
is bounded uniformly in $t, \e$, and $h$. In particular, this is satisfied if the particles remain compactly supported in a ball.

A sufficient condition for Assumption (A\ref{extra assumptions c}) is that
\begin{align} \label{simple nonlocal sobolev}
	\int |\grad \zeta_\e*p_\e| \,d\zeta_\e * \mu_\e + \int |\grad \zeta_\e *\mu_\e| \,d\zeta_\e*p_\e,
\end{align}
with $p_\e=(\varphi_\e*\mu_\e)^{m-2}\mu_\e$, remains bounded uniformly in $t$, $\e$, and $h$. In fact, for purely diffusive problems, we observe that this quantity is not only bounded uniformly in $\e$ and $h$, but decreases in time along our numerical solutions; see Figure \ref{Nonlocal Sobolev 1D} below. For the nonlinear Fokker--Planck equation, we observe that this quantity is bounded uniformly in $\e$ and $h$ and converges to the corresponding norm of the steady state as $t \to \infty$; see Figure \ref{FPfig} below.

A sufficient condition for Assumption (A\ref{extra assumptions a}) is that the blob solution converges to a limit in $L^1$ and $L^\infty$, uniformly on bounded time intervals. Again, we observe this numerically, in both one and two dimensions, and both for purely diffusive equations and the nonlinear Fokker--Planck equation; see Figures \ref{1Dloglog}--\ref{FPfig} below. 
In this way, Assumptions (A\ref{extra assumptions b})--(A\ref{extra assumptions a}) may be verified numerically in order to give confidence that the limit of any blob method solution is, in fact, the correct exact solution.
\end{remark}

\subsection{Numerical implementation}

We now describe the details of our numerical implementation. In all of the numerical examples which follow, our mollifiers $\zeta_\e$ and $\varphi_\e$ are given by Gaussians,
\bes
	\zeta_\e(x) = \frac{1}{(4\pi \epsilon^2)^{d/2}} e^{-|x|^2/4 \epsilon^2} \ , \quad \varphi_\e(x) = \zeta_\e*\zeta_\e(x) = \frac{1}{(8\pi \epsilon^2)^{d/2}} e^{-|x|^2/8 \epsilon^2}, \qquad x \in \Rd, \e>0.
\ees
In addition to Gaussian mollifiers, we also performed numerical experiments with a range of compactly supported and oscillatory mollifiers and observed similar results. In practice, Gaussian mollifiers provided the best balance between speed of computation and speed of convergence.

We construct our numerical particle solutions $\mu_\e(t)$ as described in Theorem \ref{numerics convergence}. As a mild simplification, we consider the mass of each particle to be given by $m_i = \mu(0,ih)h^d$, where $\mu(0,ih)$ is the value of the initial datum $\mu(0)$ at the grid point $ih$. For the numerical examples we consider, in which $\mu(0)$ is a continuous function, the rate of convergence is indistinguishable from defining $m_i$ as in \eqref{initial data discretization}.

The system of ordinary differential equations that prescribes the evolution of the particle locations (c.f. \eqref{ODEsystem} and \eqref{particle gradient flow solution}) can be solved numerically in a variety of ways, and we observe nearly identical results independent of our choice of ODE solver. In analogy with previous work on blob methods in the fluids case \cite{AndersonGreengard}, we find that the numerical error due to the choice of time discretization is of lower order than the error due to the regularization and spatial discretization. We implement the blob method in Python, using the Numpy, SciPy, and Matplotlib libraries \cite{NumPy,SciPy,matplotlib}. In particular, we compute the evolution of the particle trajectories via the SciPy implementation of the Fortran VODE solver \cite{VODE}, which uses either a backward differentiation formula (BDF) method or an implicit Adams method, depending on the stiffness of the problem. 

Our convergence result, Theorem \ref{numerics convergence}, requires that $h = o(\e)$ as $\e \to 0$. Numerically, we observe the fastest rate of convergence with $\epsilon = h^{1- p}$, for $0 < p \ll 1$, as $h \to 0$. Since computational speed decreases as $p$ approaches 0, we take $\epsilon = h^{0.99}$ in the following simulations. In these examples, we discretize the initial data on a line ($d=1$) or square of sidelength $5.0$ ($d=2$), centered at $0$.

Finally, to visualize our particle solution (\ref{particle gradient flow solution}) and compare it to the exact solutions in $L^p$-norms, we construct a blob solution obtained by convolving the particle solution with a mollifier,
\begin{align} \label{blob gradient flow solution}
\tilde{\mu}_\epsilon(t,\cdot) := \varphi_\e* \mu_\e(t,\cdot)  = \sum_{i \in Q_R^h}\varphi_\e(\cdot - x_i) m_i, \qquad t\in[0,T]
\end{align}
By Lemma \ref{weakst convergence mollified sequence}, if $\mu_\e \wsto \mu$ as $\e\to0$, where $\mu$ is the exact solution, then we also have $\tilde{\mu}_\e \wsto \mu$. Consequently our convergence result, Theorem \ref{numerics convergence}, also applies to this blob solution.

We measure the accuracy of our numerical method with respect to the $L^1$-, $L^\infty$-, and Wasserstein metrics. To compute the $L^1$- and $L^\infty$-errors, we take the difference between the exact solution and the blob solution (\ref{blob gradient flow solution}) and evaluate discrete $L^1$- and $L^\infty$-norms using the following formulas:
\[ \|f\|_{L^1(Q_R^h)} = \sum_{i \in Q_R^h} | f(ih) | h^d , \quad  \|f\|_{L^\infty(Q_R^h)} = \max_{i \in Q_R^h} | f(ih)|,  \quad \mbox{for a given function $f\:\Rd\to\R$} . \] 

We compute the Wasserstein distance between our particle solution $\mu_\e$ in \eqref{particle gradient flow solution} and the exact solution $\mu$ in one dimension using the formula
\begin{align} \label{compute W2 1D}
W_2(\mu_\e, \mu) = \left( \int_0^1 |F_{\mu_\e}^{-1}(s) - F_{\mu}^{-1}(s)|^2 \,ds \right)^{1/2} ,
\end{align}
where $F_{\mu_\e}^{-1}$ and $F_\mu^{-1}$ are the generalized inverses of the cumulative distribution functions of $\mu$ and $\mu_\e$, respectively; c.f. \cite[Theorem 6.0.2]{AGS}. We evaluate the integral in \eqref{compute W2 1D} numerically using the SciPy implementation of the Fortran library QUADPACK \cite{QUADPACK}. In two dimensions, we compute the Wasserstein error by discretizing the exact and blob solutions as piecewise constant functions on a fine grid and then using the Python Optimal Transport library to compute the discrete Wasserstein distance between them.  In particular, we use the Earth Mover's Distance function in this library, which is based on the network simplex algorithm introduced by Bonneel, van de Panne, Paris, and Heidrich \cite{Bonneel}.

\subsection{Simulations} \label{simulations}
Using the method described in the previous section, we now give several examples of numerical simulations. We consider initial data given by linear combinations of Gaussian and Barenblatt profiles, which we denote as follows:
\begin{center}
\begin{figure}[!ht] 
\centering
\textbf{Heat and Porous Medium Equations: Fundamental Solution} 

\smallskip
Exact vs. Numerical Solution, $h = 0.02$, varying $m$ \\ \smallskip
\includegraphics[scale=.81]{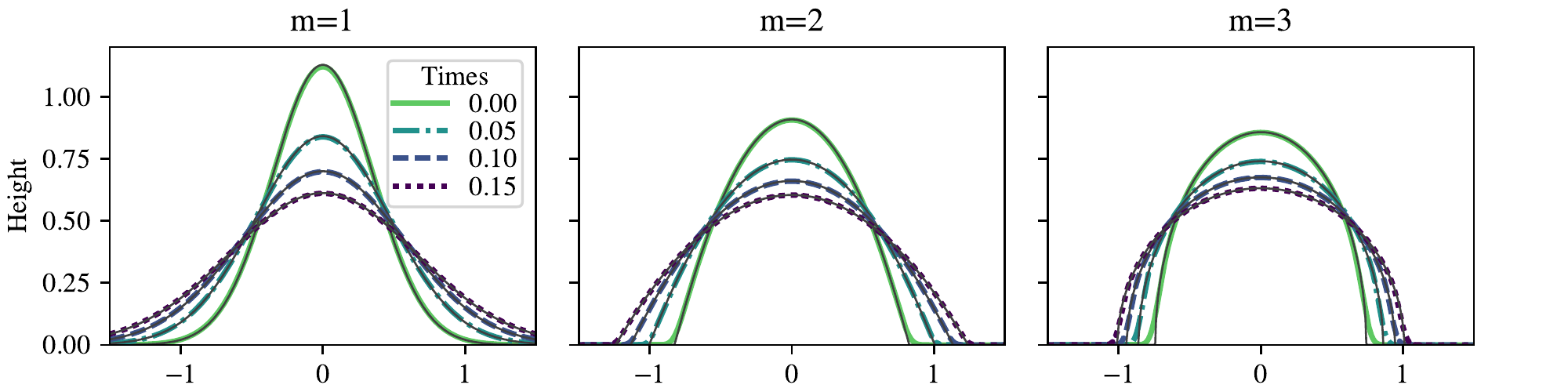}

{\footnotesize Position} \hspace{3.6in} \

Exact vs. Numerical Solution, varying $h$, $m = 3$ \\ \smallskip
\includegraphics[scale=.81]{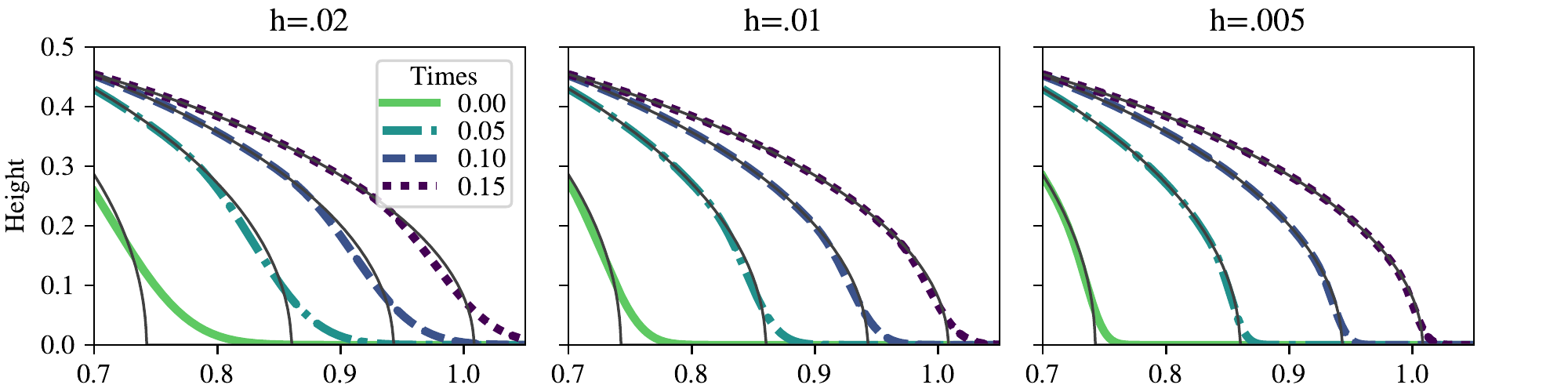}
{\footnotesize Position} \hspace{3.6in} \
	\caption{Comparison of exact and numerical solutions to the heat ($m=1)$ and porous medium ($m=2,3$) equations. Numerical solutions are plotted with thick lines, and exact solutions are plotted with thin lines.}
	\label{exactnumericaldensity}
\end{figure}
\end{center}
\bes
	\psi_{m}(\tau,x) = \begin{cases} \frac{1}{(4\pi \tau)^{d/2}} e^{-|x|^2/4 \tau} &\text{ for } m=1, \\
 \tau^{-d\beta}(K - \kappa \tau^{-2\beta} |x|^2)_+^{1/(m-1)} & \text{ for } m>1, \end{cases} \quad x \in \Rd,
\ees
with
\bes
	\beta = \frac{1}{2+d(m-1)} \quad \text{ and } \quad \kappa = \frac{\beta}{2} \left( \frac{m-1}{m} \right),
\ees
and $ K = K(m,d)$ chosen so that $\int \psi_{m}(\tau,x) dx = 1$.
 
In Figure \ref{exactnumericaldensity}, we compare exact and numerical solutions to the heat and porous medium equations ($V=W=0$, $m=1,2,3$), with initial data given by a Gaussian ($m=1$) or Barenblatt ($m=2, 3$) function with scaling $\tau = 0.0625$. The top row shows the evolution of the density on a large spatial scale, at which the exact and numerical solutions are visually indistinguishable for $m=1$ and $m=2$. However, for $m=3$ the fat tails of the numerical simulation peel away from the exact solution at small times. The second row depicts the numerical simulations for $m=3$ on a smaller spatial scale, illustrating how the tails of the numerical simulation converge to the exact solution as the spacing of the computational grid is refined.

\begin{center}
\begin{figure}[!ht] 
\centering
\textbf{Heat and Porous Medium Equations: Double Bump Initial Data}  \\ \ \\
\includegraphics[trim={.25cm 0cm .6cm .5cm},clip,scale=.85]{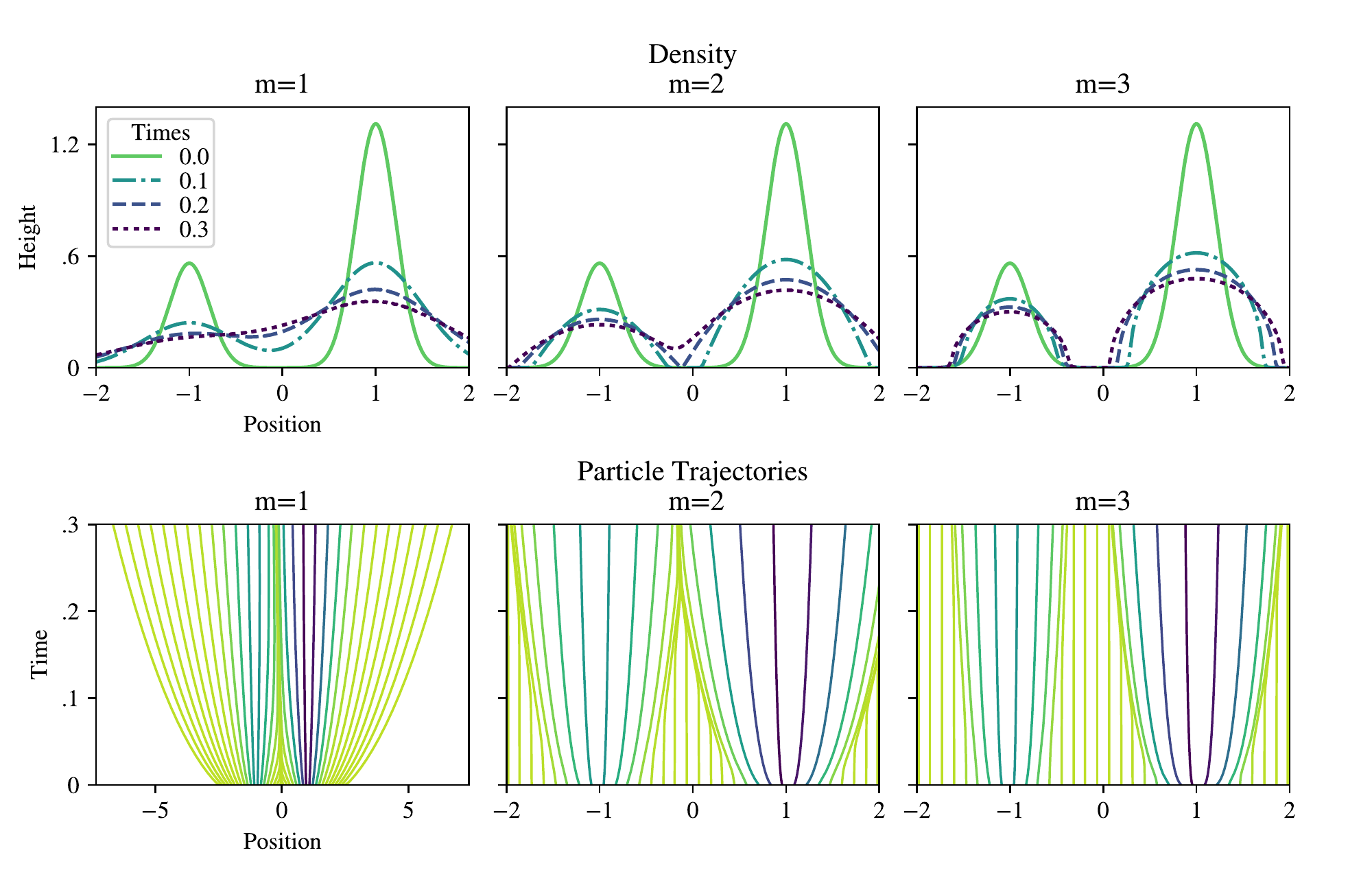}
\includegraphics[trim={.25cm .4cm .6cm 0cm},clip,scale=.85]{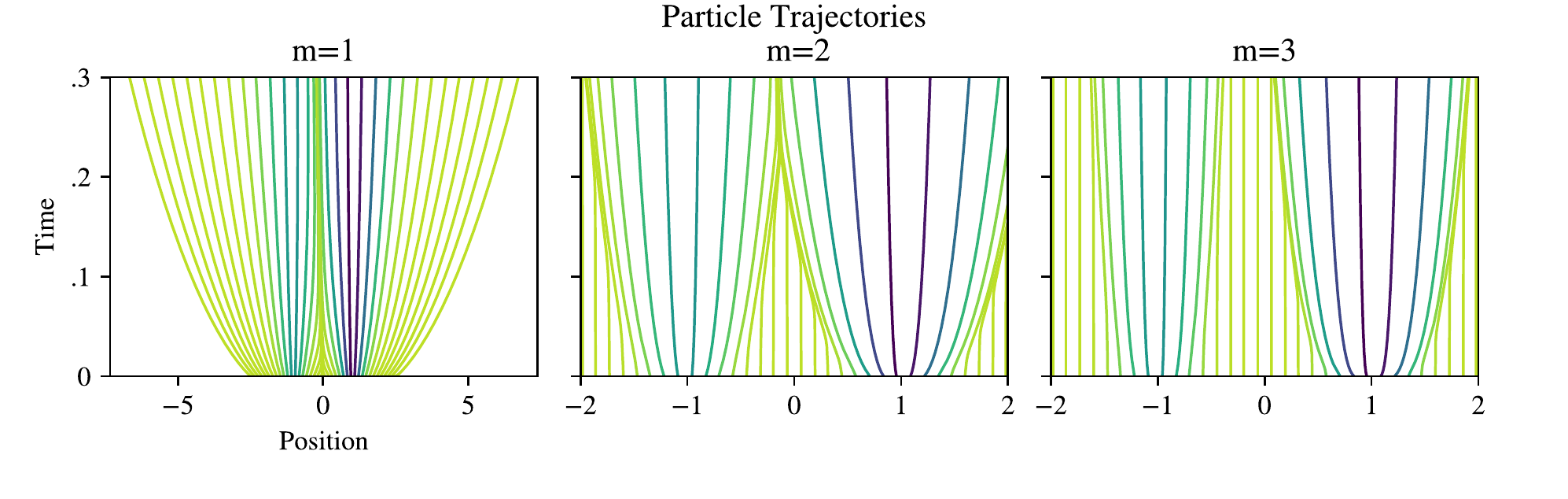}
	\caption{Numerical simulation of the one-dimensional heat and porous medium equations. \textbf{Top: }Evolution of the blob density $\rho^h_\epsilon$. \textbf{Bottom:} Evolution of the particle trajectories $x_i$, with colors indicating relative mass of each particle.  \ \\ \
}
\label{doublebumpdiffusion}
\end{figure}
\end{center}

 In Figure \ref{doublebumpdiffusion}, we compute solutions of the one-dimensional heat and porous medium equations ($V=W=0$, $m=1,2,3$), illustrating the role of the diffusion exponent $m$. The initial data is given by a linear combination of Gaussians, $\rho_0(\cdot) = 0.3 \psi_1(\cdot+1,0.0225)+0.7 \psi_1(\cdot-1, 0.0225)$, and the grid spacing is $h = 0.01$. For $m=1$, the infinite speed of propagation of support of solutions to the heat equation is reflected both at the level of the density, for which the gap between the two bumps fills quickly, and also in the particle trajectories, which quickly spread to fill in areas of low mass. In contrast, for $m=2$ and $m=3$, we observe finite speed of propagation of support, as well as the emergence of Barenblatt profiles as time advances.

\begin{center}
\begin{figure}[!ht] 
\centering
\textbf{Heat and Porous Medium Equations: Evolution of Nonlocal Sobolev Norm} \\ \ \\
\begin{flushleft} \hspace{1.1in} Fundamental Solutions \hspace{1.32in} Double Bump Initial Data \hspace{.1in} \end{flushleft}

\includegraphics[trim={0cm .4cm 0cm .4cm},clip,scale=.85]{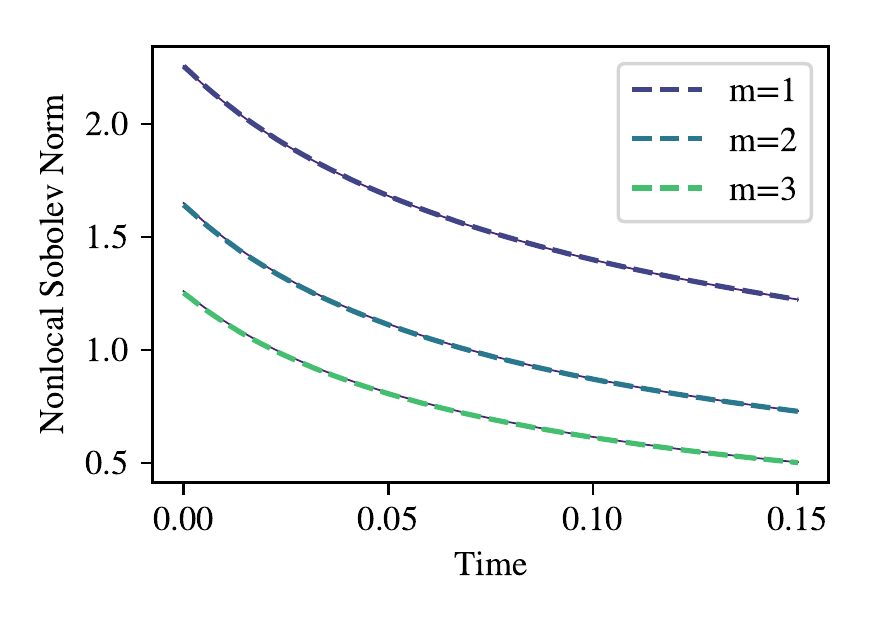}
\includegraphics[trim={.7cm .4cm .3cm .4cm},clip,scale=.85]{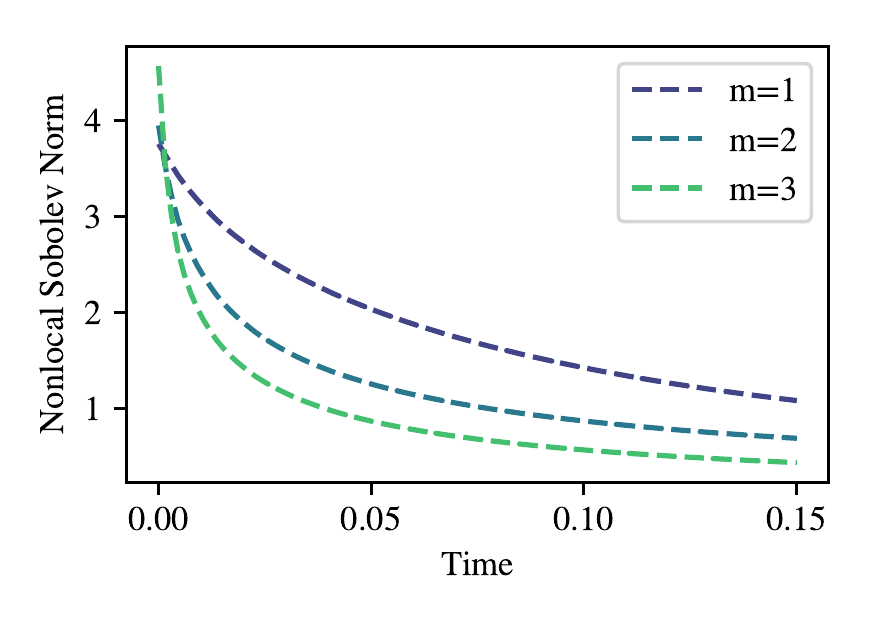}
\caption{ \textbf{Left:} Comparison of nonlocal Sobolev norm \eqref{simple nonlocal sobolev} along numerical solutions (dashed line) with the value of $\|\grad \mu^m \|_{L^1(\Rd)}$ along exact solutions $\mu$ (solid line).  \textbf{Right:} Evolution of nonlocal Sobolev norm along the numerical solutions. } \label{Nonlocal Sobolev 1D}\end{figure}
\end{center}

In Figure \ref{Nonlocal Sobolev 1D}, we compute the evolution of the nonlocal Sobolev norm (\ref{simple nonlocal sobolev}) along the numerical solutions from Figures \ref{exactnumericaldensity} and \ref{doublebumpdiffusion}. In both cases, we observe that the quantity converges as $h \to 0$ and decreases in time. This gives further credence to the heuristic that the nonlocal Sobolev norm is an approximation of the $L^1$-norm of the gradient of the $m$th power of the exact solution, which does decrease in time along the exact solution; see \eqref{BV heuristic 1} and \eqref{BV heuristic 2}. In particular, this provides numerical evidence that Assumption (A\ref{extra assumptions c}) from our main convergence theorem, Theorem \ref{Gamma GF theorem}, is satisfied.

\begin{center}
\begin{figure}[!ht]
\centering
\textbf{Convergence Analysis: One-Dimensional Diffusion}
\includegraphics[trim={0cm .3cm 0cm .8cm},clip,scale=.93]{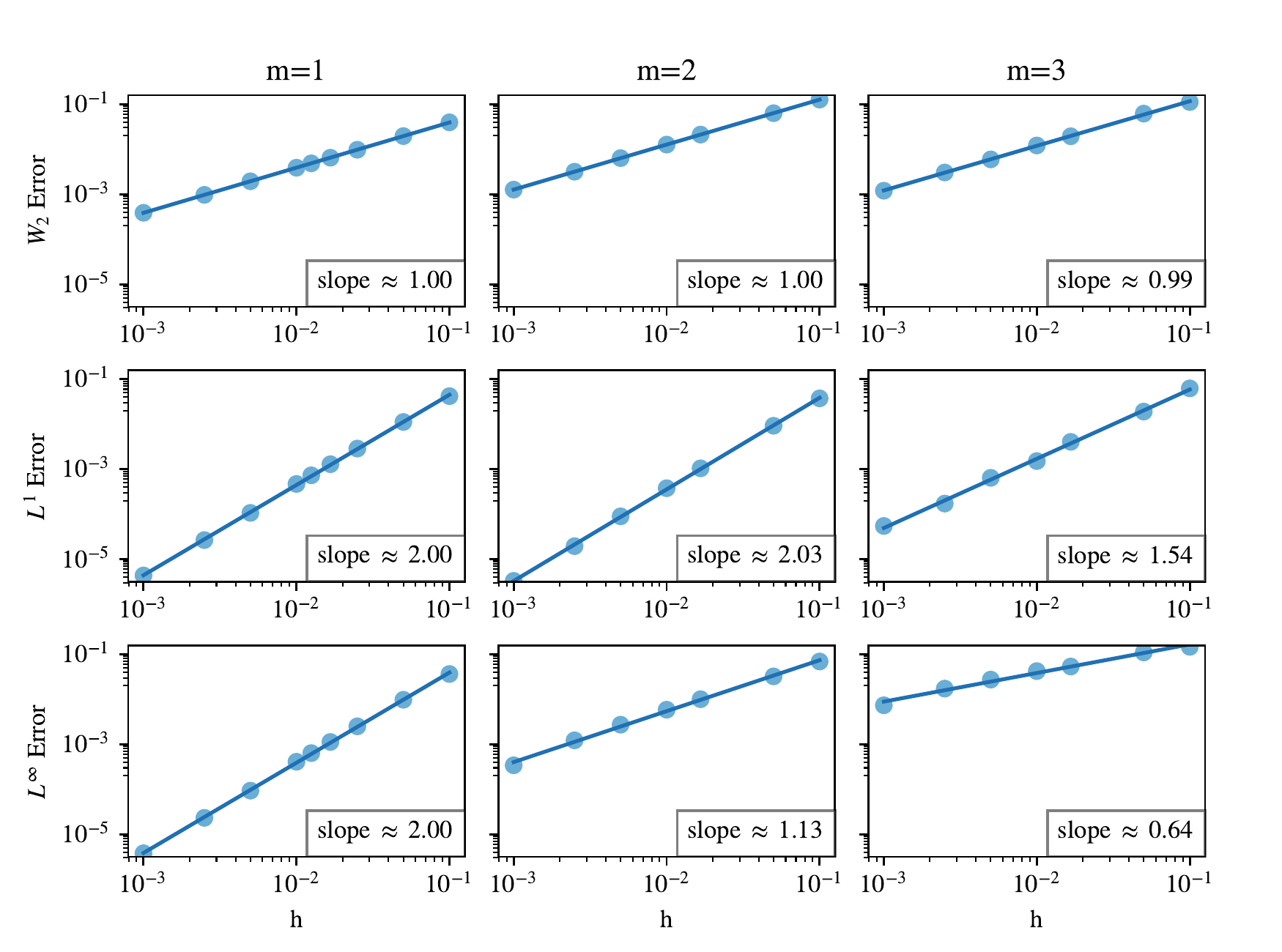}
	\caption{Rate of convergence of blob method for one-dimensional heat and porous medium equations. \label{1Dloglog}}
\end{figure}
\end{center}

\begin{center}
\begin{figure}[!ht]
\centering
\textbf{Convergence Analysis: Two-Dimensional Diffusion}\vspace{0.2cm}
\includegraphics[trim={0cm .3cm 0cm .2cm},clip,scale=.93]{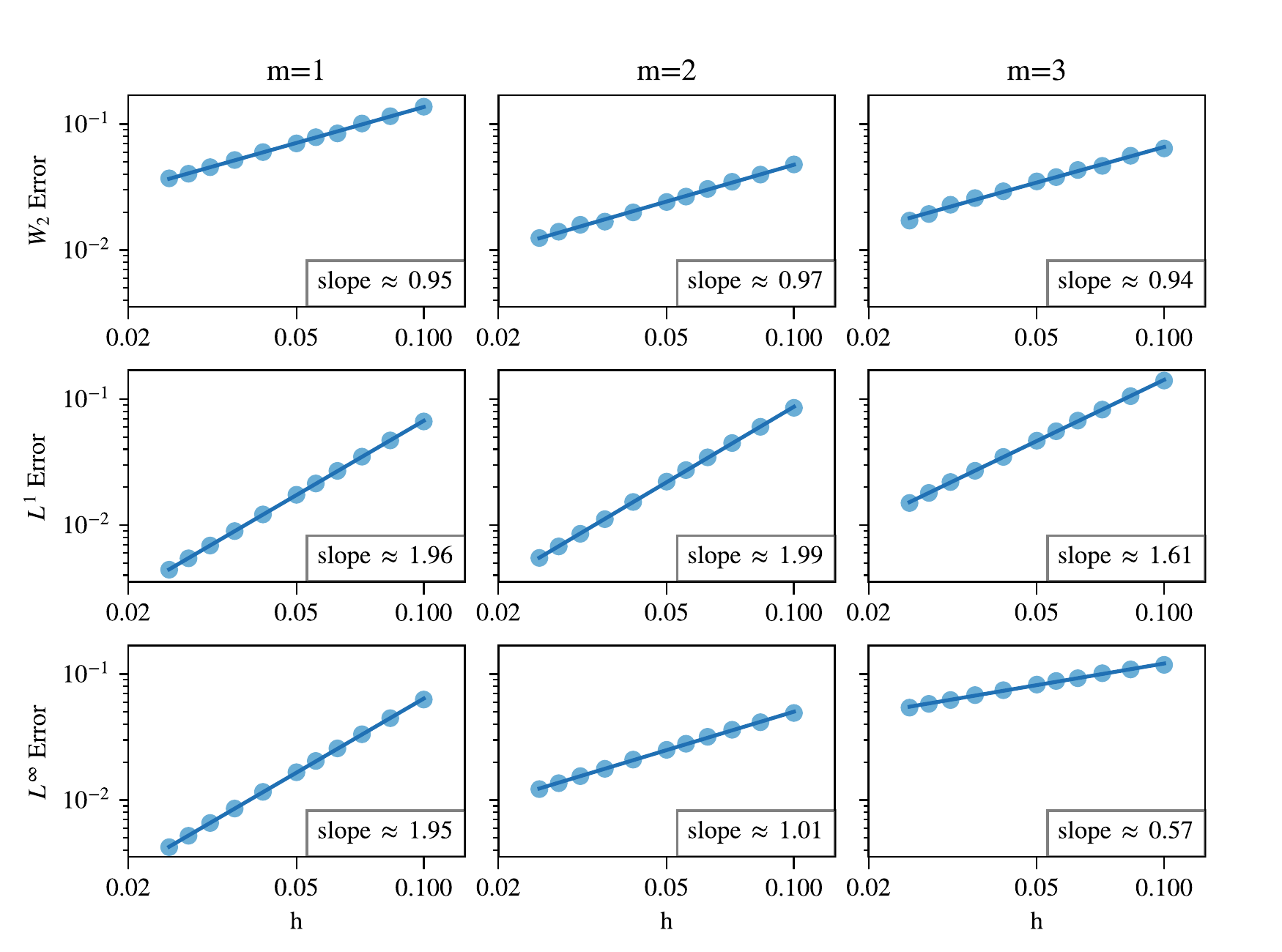}
	\caption{Rate of convergence of blob method for two-dimensional heat and porous medium equations. } \label{2Dloglog}
\end{figure}
\end{center}

In Figure \ref{1Dloglog}, we analyze the rate of convergence of our numerical scheme in one dimension. We compute the error between numerical and exact solutions of the heat and porous medium equations ($m =1, 2,3$) in Figure \ref{exactnumericaldensity} at time $t = 0.05$, with respect to the $2$-Wasserstein distance, $L^1$-norm, and $L^\infty$-norm and examine the scaling of the error with the grid spacing $h$. (Recall that $\e = h^{0.99}$ throughout.) Plotting the errors on a logarithmic scale, we observe that the Wasserstein error depends linearly on the grid spacing for all values of $m$. The $L^1$-norm scales quadratically for $m=1$ and $2$ and superlinearly for $m=3$. Finally, the $L^\infty$-error scales quadratically for $m=1$, superlinearly for $m=2$, and sublinearly for $m=3$. This deterioration of the rate of $L^\infty$-convergence for $m=3$ is due to the sharp transition at the boundary of the exact solution; see the second row of Figure \ref{exactnumericaldensity}. 
In Figure \ref{2Dloglog}, we perform the same analysis on the rate of convergence of our method in two dimensions and observe similar rates of convergence as in the one-dimensional case.

\begin{center}
\begin{figure}[!ht]
\centering
\textbf{Fokker--Planck: Two Dimensions}\\  
Rate of Convergence to Steady State

\includegraphics[scale=.8]{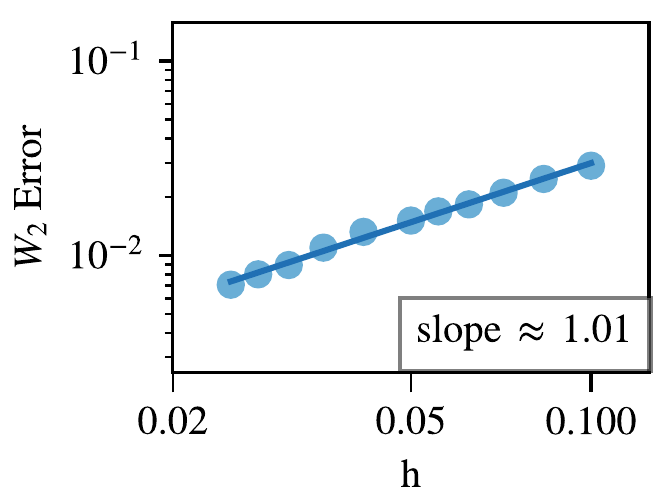}
\includegraphics[trim={.1cm -.1cm .1cm .01cm},clip,scale=.8]{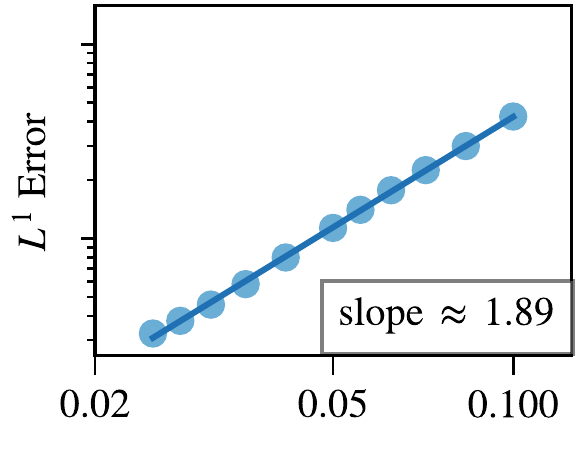}
\includegraphics[scale=.8]{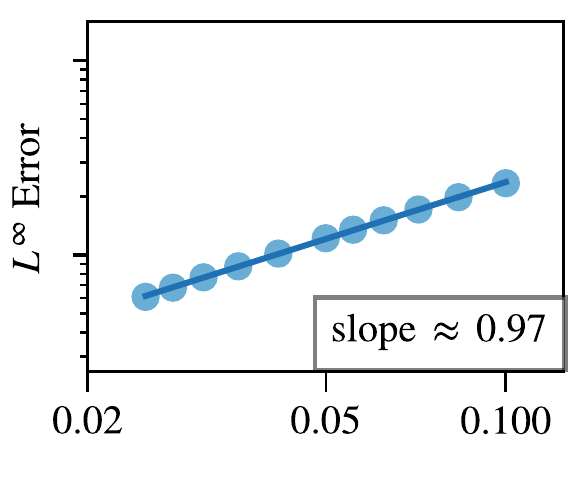}

Evolution of Density: Barenblatt and double bump initial data

{\begin{flushleft} \hspace{1.35in} \footnotesize t = 0.0 \hspace{1.35in} t = 0.6 \hspace{1.35in} t = 1.2 \hspace{-0.3in} \end{flushleft}}

\includegraphics[trim={0cm .3cm 0cm .0cm},clip,scale=.85]{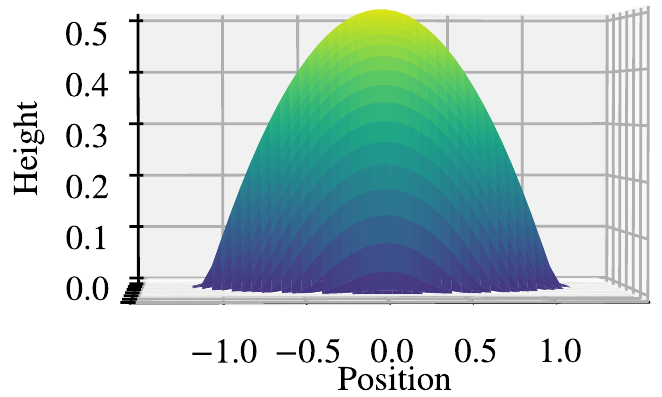}
\includegraphics[trim={0cm .3cm 0cm .0cm},clip,scale=.85]{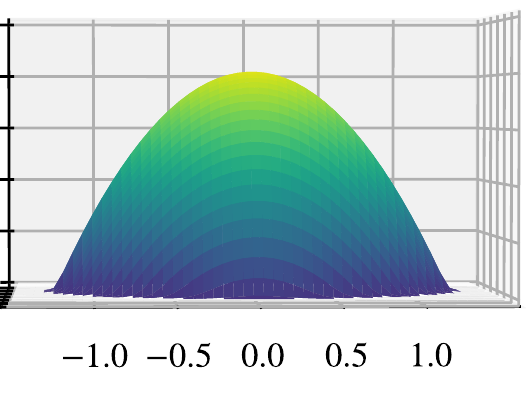}
\includegraphics[trim={0cm .3cm 0cm .0cm},clip,scale=.85]{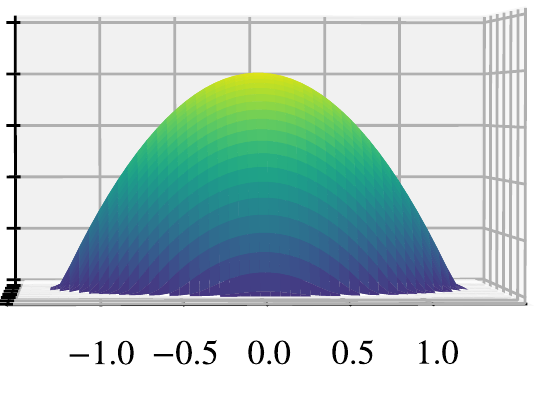}
\vspace{-.1cm}
{ \begin{flushleft}  \hspace{1.35in} \footnotesize t = 0.0 \hspace{1.35in} t = 0.6 \hspace{1.35in} t = 1.2 \hspace{-0.3in} \end{flushleft}}
\includegraphics[scale=.85]{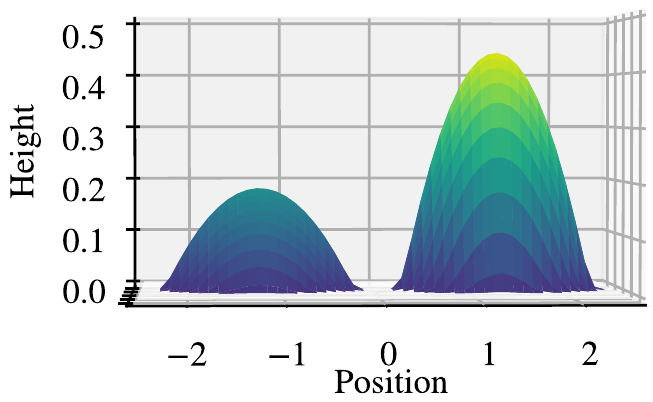}
\includegraphics[scale=.85]{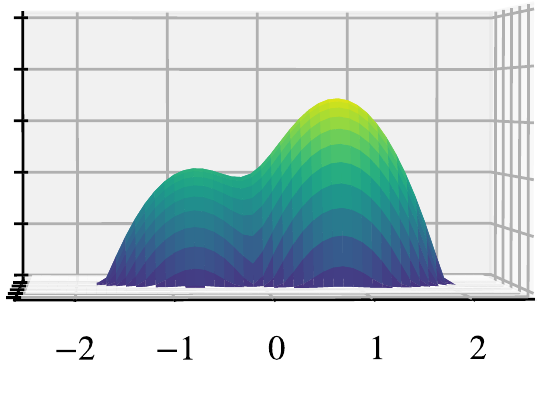}
\includegraphics[scale=.85]{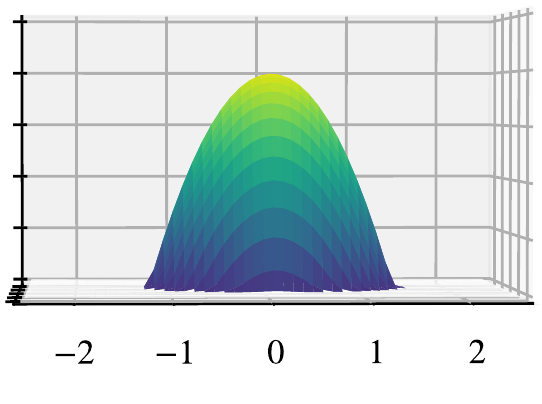}

Evolution of Nonlocal Sobolev Norm: Barenblatt and double bump initial data
\begin{flushleft} \footnotesize \hspace{1.4in} Barenblatt initial data \hspace{1.4in} Double bump initial data \hspace{.1in} \end{flushleft}
\includegraphics[trim={0cm .5cm 0cm .4cm},clip,scale=.85]{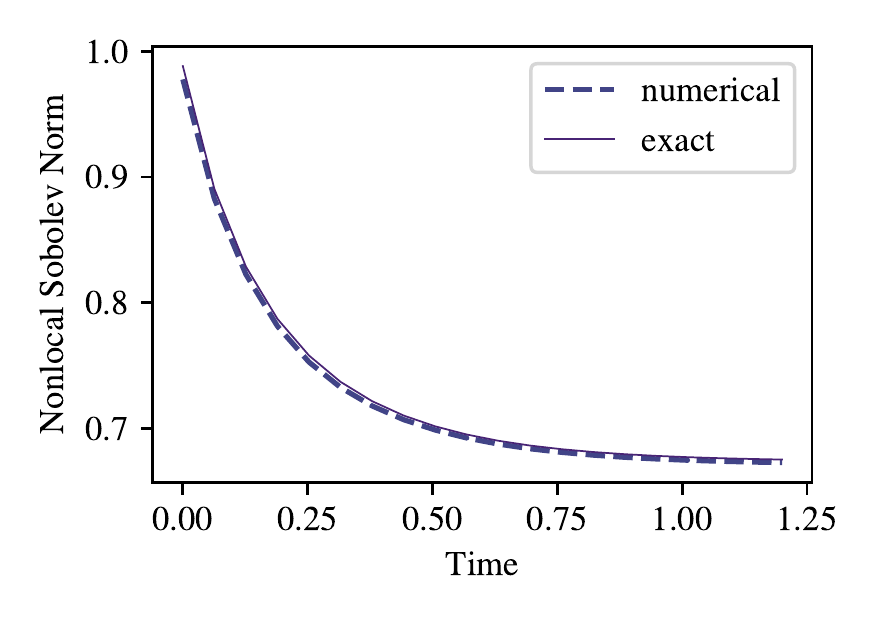}
\includegraphics[trim={.7cm .5cm .3cm .4cm},clip,scale=.85]{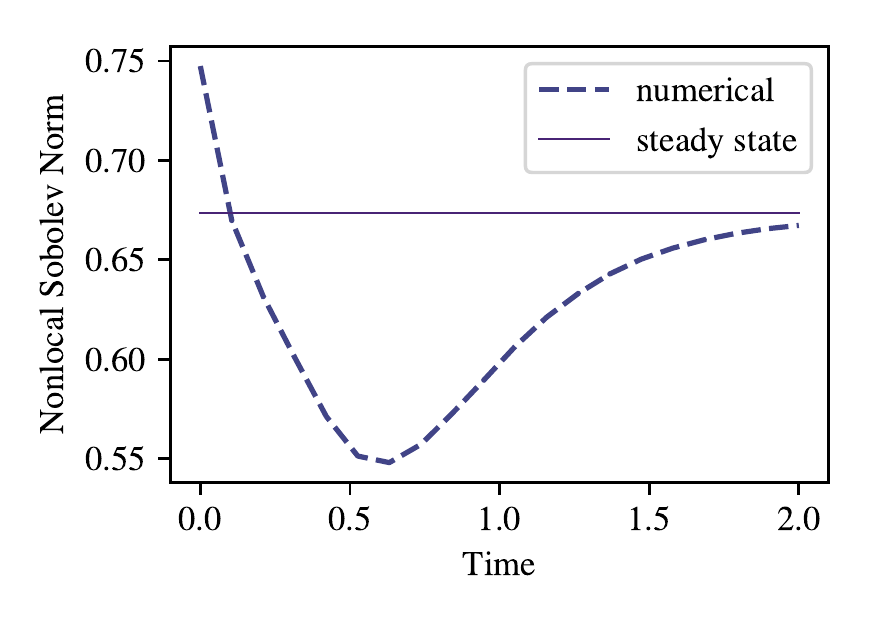}
\vspace{-.3cm}
	\caption{\textbf{Top row:} Error between numerical solutions and steady state.  \textbf{Middle rows:} Snapshots of the evolution towards steady state. \textbf{Bottom left:} Comparison of nonlocal Sobolev norm \eqref{simple nonlocal sobolev} along numerical solution from second row (dashed line) with $\|\grad \mu^m \|_{L^1(\Rd)}$ along exact solution $\mu$ (solid line). \textbf{Bottom right:} Comparison of nonlocal Sobolev norm along numerical solution from third row (dashed line) with $\|\grad \mu^m \|_{L^1(\Rd)}$ evaluated at steady state $\mu$ (solid line).}
	\label{FPfig}
\end{figure}
\end{center}

In Figure \ref{FPfig}, we simulate solutions to the nonlinear Fokker--Planck equation ($V(\cdot) = \abs{\cdot}^2/2$, $W=0$, $m=2$) and consider the rate of convergence to the steady state of the equation, $\psi_2(0.25,x)$.  In the top row, we compute the error between the numerical solution at time $t=1.2$ and the steady state with respect to the Wasserstein, $L^1$-, and $L^\infty$-norms for various choices of grid spacing $h$. We consider solutions with Barenblatt initial data ($m=2$, $\tau = 0.15$). We plot the error's dependence on $h$ with a logarithmic scale and compute the slope of the line of best fit to determine the scaling relationship between the error and $h$. We observe similar rates of convergence as in the case of the heat and porous medium equations; see Figure \ref{2Dloglog}. In the middle rows,  we give snapshots of the evolution of the blob method solution, as it converges to the steady state. We consider Barenblatt initial data ($m=2$, $\tau = 0.15$) and double bump initial data given by a linear combination of Barenblatts, $\rho_0(x) = 0.7\psi_2(x-(1.25,0),0.1)+0.3\psi_2(x+(1.25,0),0.1)$. The grid spacing is $h = 0.02$. In the bottom row, we compute the evolution of the nonlocal Sobolev norm (\ref{simple nonlocal sobolev}) along the numerical solutions from the middle rows. In both cases, we observe that this quantity converges for $h$ small. For Barenblatt initial data, it decreases in time along the numerical solution and agrees well with the value of $\|\grad \mu^2\|_{L^1(\Rd)}$ along the exact solution $\mu$. For the double bump initial data, it remains bounded in time, converging asymptotically to  $\| \grad (\psi_2(0.25, \cdot))^2 \|_{L^1}$, where $\psi_2(0.25, \cdot)$ is the steady state. Again, this supports the interpretation of the nonlocal Sobolev norm as an approximation of the $L^1$-norm of the gradient of the $m$th power of the exact solution and provides numerical evidence for Assumption (A\ref{extra assumptions c}) from Theorem \ref{Gamma GF theorem}.

In the remaining numerical examples, we apply our method to simulate solutions of Keller--Segel type equations, with the interaction potential $W$ given by $2 \chi \log\abs{\cdot}$ for $\chi>0$. In one dimension, the derivative of this potential is not integrable, and we remove its singularity it setting it equal to $2 \chi/\e$ for all $x\in\Rd$ such that $|x| < \e$. In two dimensions, the gradient of this potential is integrable, and we regularize it by convolving it with a mollifier $\varphi_\e$, as done in previous work by the second author on a blob method for the aggregation equation \cite{CB}.

\begin{center}
\begin{figure}[!ht] 
\centering
\textbf{One-Dimensional Keller--Segel Equation: Blow-up} \\

\begin{flushleft} \hspace{2.2cm} Evolution of Second Moment \hspace{2.4cm} Evolution of Particle Trajectories\end{flushleft}
\includegraphics[trim={.4cm .5cm .4cm .4cm},clip,scale=.85]{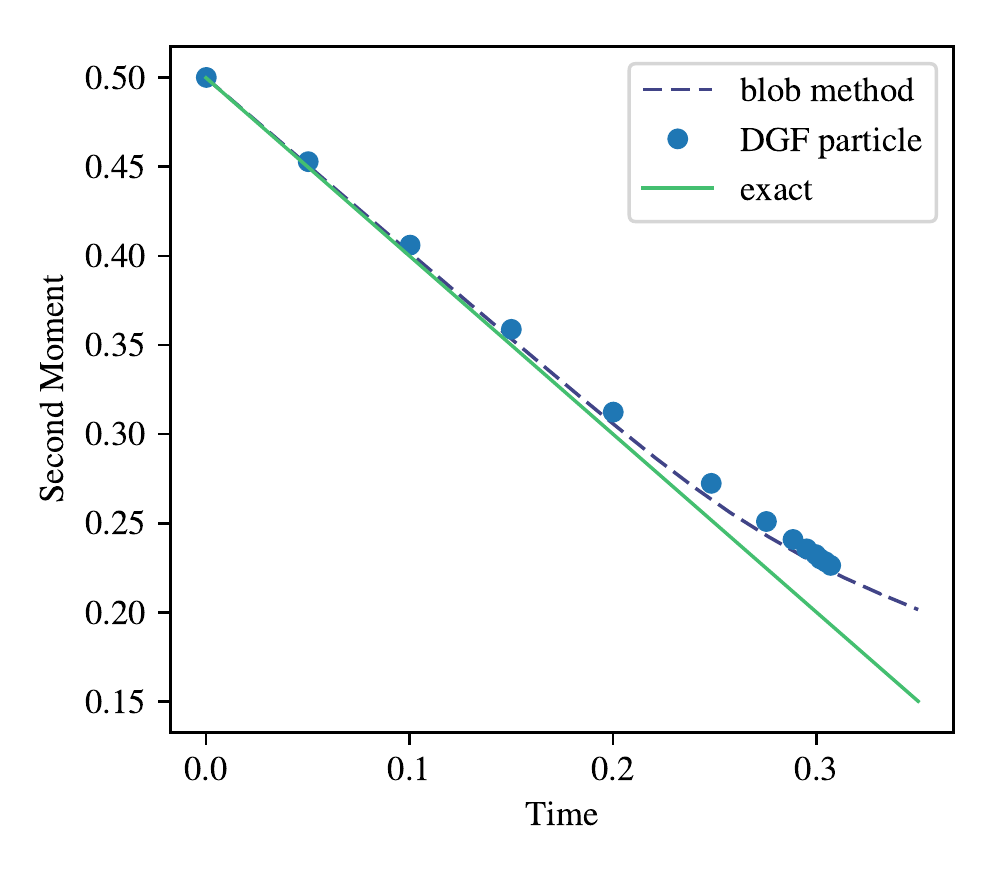}
\includegraphics[trim={0cm .2cm .5cm .9cm},clip,scale=.85]{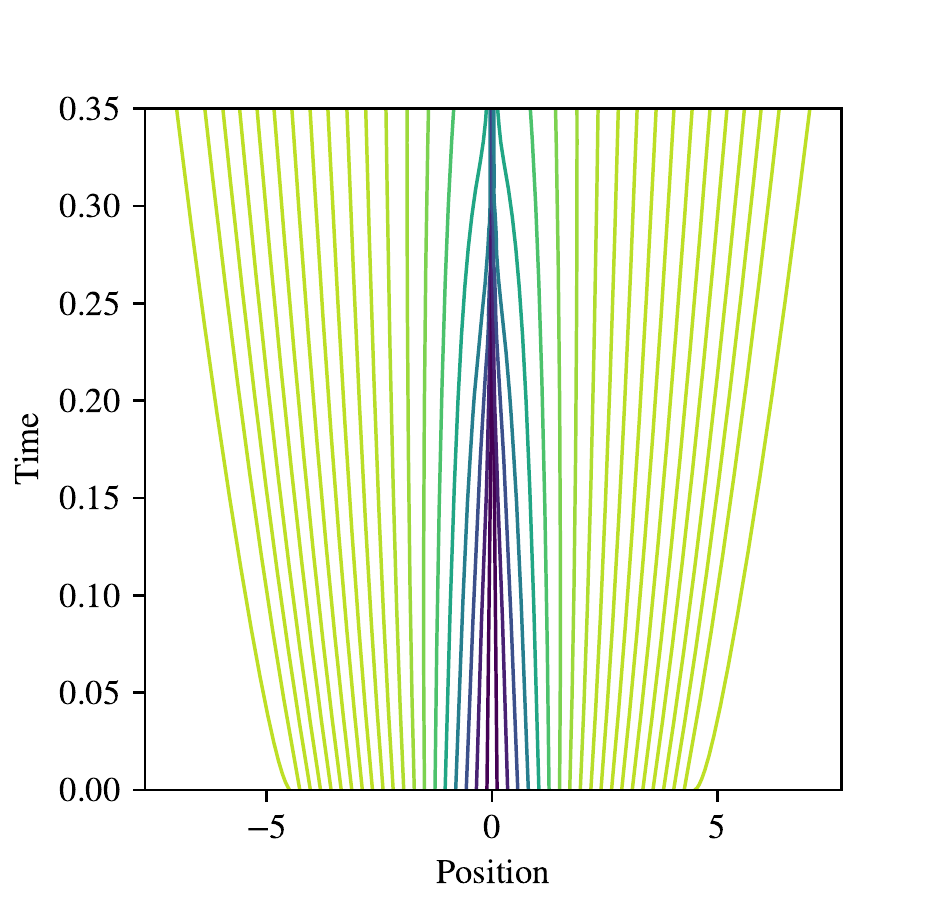}
\vspace{-.25cm}
\caption{\textbf{Left:} Comparison of the evolution of the second moment along exact solutions (solid line) with blob method solutiosn (dashed line) and previous numerical results by the DGF particle method \cite{CHPW}. \textbf{Right:} Evolution of particle trajectories, with colors indicating the relative mass of each particle.} \label{1DKellerSegel}
\end{figure}
\end{center}

\begin{center}
\begin{figure}[t] 
\centering
\textbf{One-Dimensional Nonlinear Keller--Segel Equation: Convergence to Steady State} \\

\begin{flushleft} \hspace{2.3cm} Evolution of Second Moment \hspace{2cm} Evolution of Particle Trajectories \end{flushleft}
\vspace{-.1cm}
\includegraphics[trim={.2cm .2cm .4cm .4cm},clip,scale=.85]{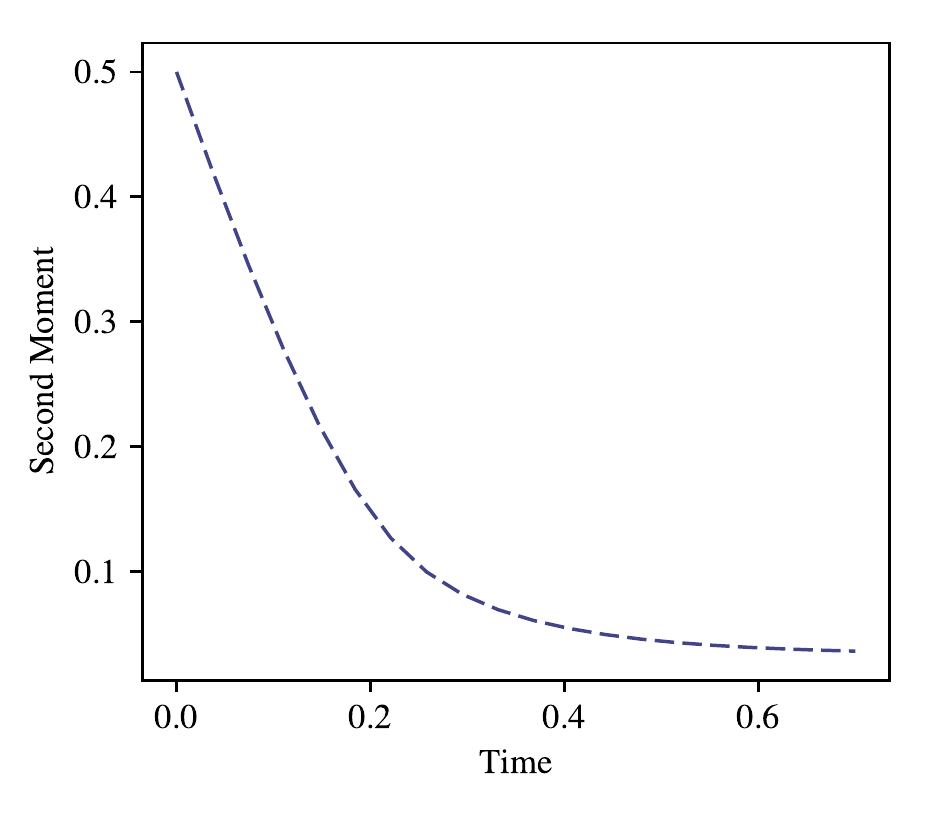}
\includegraphics[trim={0cm 0cm .5cm .8cm},clip,scale=.85]{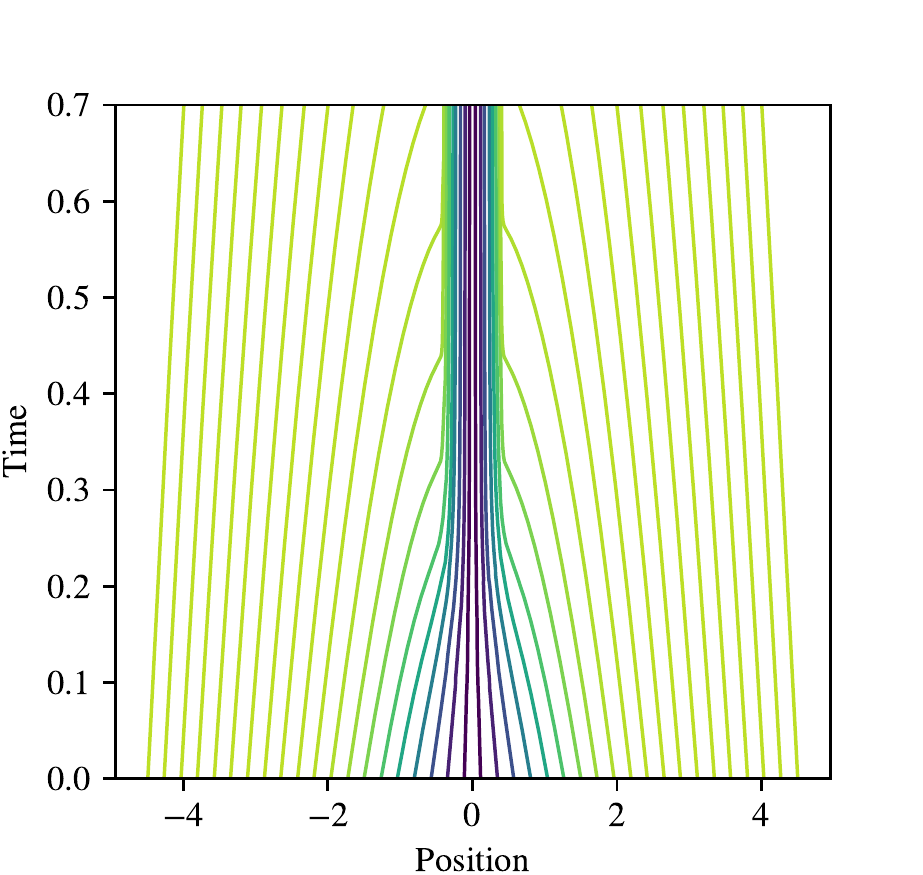} 
\vspace{-.25cm}
\caption{\textbf{Left:} Evolution of the second moment. \textbf{Right:} Evolution of particle trajectories, with colors indicating the relative mass of each particle.} 
\label{1DKellerSegelm2}
\end{figure}
\end{center}

In Figure \ref{1DKellerSegel}, we consider the one-dimensional variant of the Keller--Segel equation ($V=0$, $W(\cdot) = 2 \chi \log\abs{\cdot}$, $m=1$) studied in \cite{CPS}. Its interest is that it has a defined critical value $\chi$ for unit mass leading to the dichotomy of blow-up versus global existence. For $\chi = 1.5$ and initial data of mass one, solutions blow up in finite time. We consider initial data given by a Gaussian $\psi_1(\tau,\cdot)$,  $\tau = 0.25$, discretized on the interval $[-4.5,4.5]$ with grid spacing $h = 0.009$. We compare the evolution of the second moment of our blob method solutions with the second moment of the exact solution. We also compare our results with those obtained in previous work via a one-dimensional Discrete Gradient Flow (DGF) particle method \cite{CPSW,CHPW}. By refining our spatial grid with respect to the DGF particle method, we observe modest improvements. (Alternative simulations, with similar spatial and time discretizations as used in the DGF method, yielded similar results as obtained by DGF.) The blow-up of solution is not only evident in the second moment, which converges to zero linearly in time, but also in the evolution of the particle trajectories. In particular, we observe particle trajectories merging on several occasions as time advances.

In Figure \ref{1DKellerSegelm2}, we consider a nonlinear variant of the Keller--Segel equation ($V=0$, $W(\cdot) = 2 \chi \log\abs{\cdot}$, $m=2$) in one dimension, with initial data and discretization as in Figure \ref{1DKellerSegel}. We observe the convergence to a steady state both at the level of the second moment and the particle trajectories.

\begin{center}
\begin{figure}[t]
\centering
\textbf{Two-Dimensional Keller--Segel Equation: Evolution of Density}
\vspace{0cm}
\begin{flushleft} \hspace{1.15in} t = 0.0 \hspace{1.15in} t = 0.075 \hspace{1.15in} t = 0.15 \hspace{.1in} \end{flushleft}
\vspace{-.1cm}

{\footnotesize Subcritical Mass $=7\pi$}

\hspace*{-1.2cm}
\includegraphics[trim={.83cm 1cm 1.6cm 1cm},clip,scale=.85]{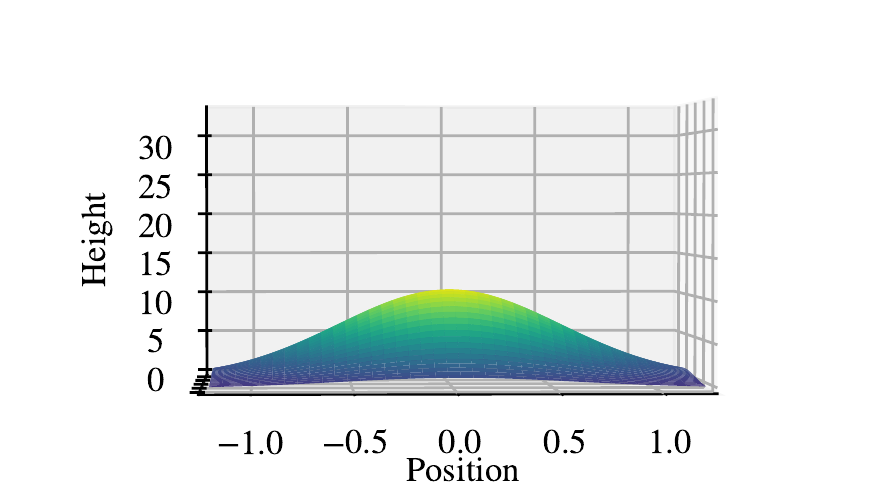}
\includegraphics[trim={2cm 1cm 1.6cm 1cm},clip,scale=.85]{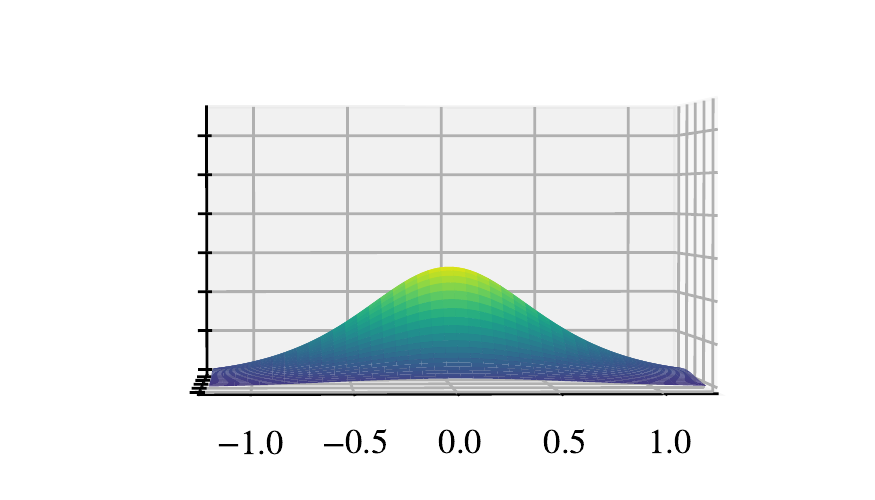}
\includegraphics[trim={2cm 1cm 1.6cm 1cm},clip,scale=.85]{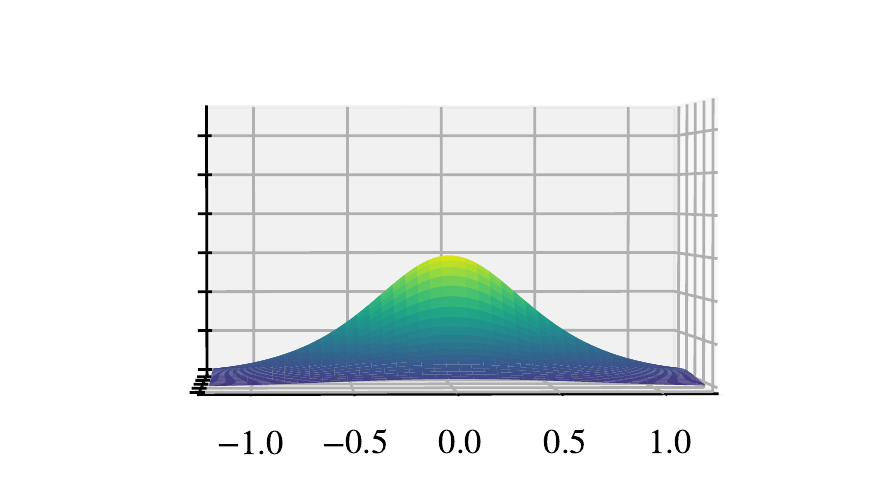}

\vspace{-.2cm}
{\footnotesize Critical Mass $=8\pi$}

\includegraphics[trim={2.03cm 1cm 1.6cm 1cm},clip,scale=.85]{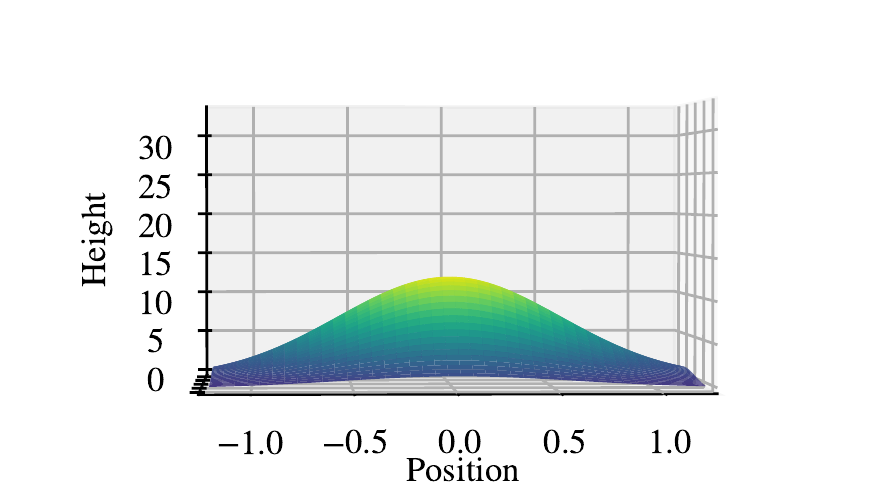}
\includegraphics[trim={2cm 1cm 1.6cm 1cm},clip,scale=.85]{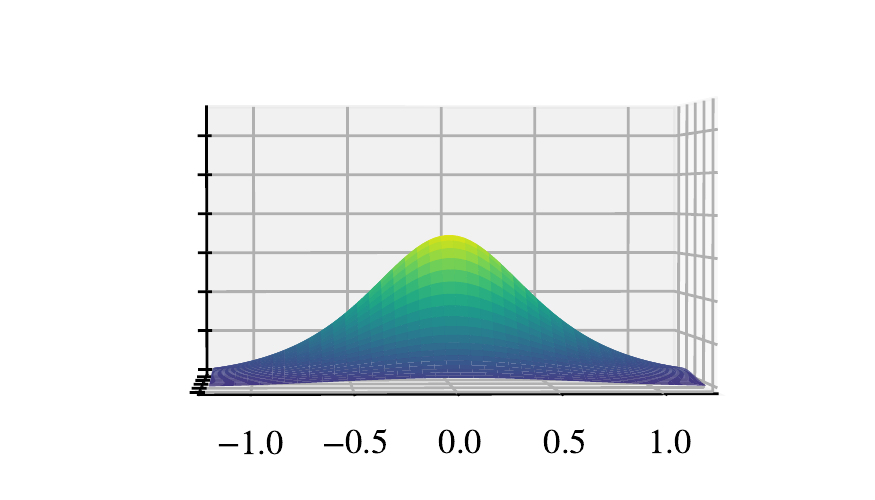}
\includegraphics[trim={2cm 1cm 1.6cm 1cm},clip,scale=.85]{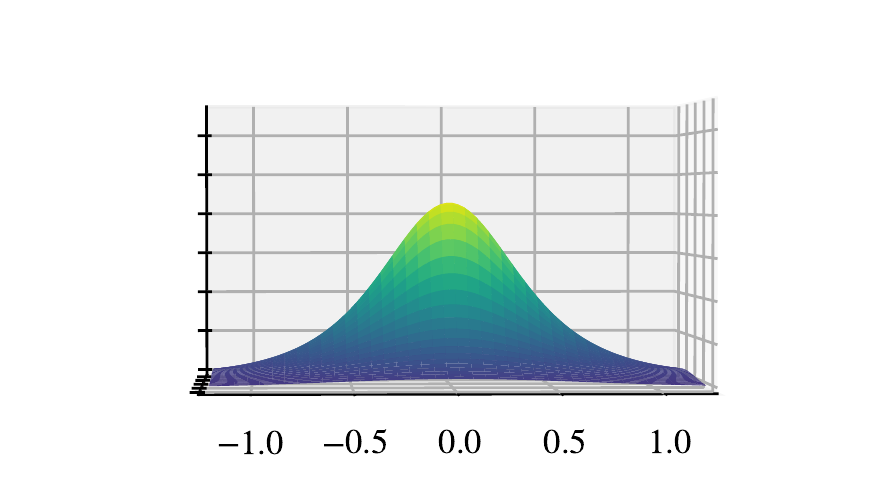}

\vspace{-.2cm}
{\footnotesize Supercritical Mass $=9\pi$}

\includegraphics[trim={2.03cm .1cm 1.6cm 1cm},clip,scale=.85]{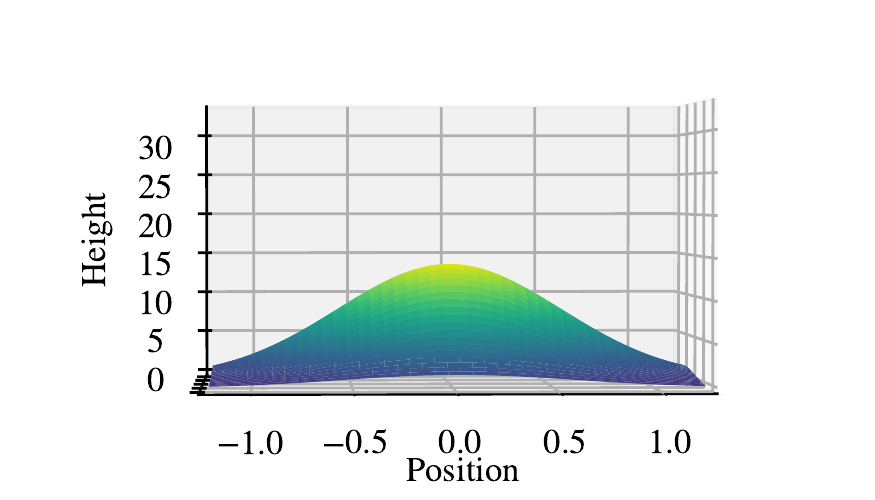}
\includegraphics[trim={2cm .1cm 1.6cm 1cm},clip,scale=.85]{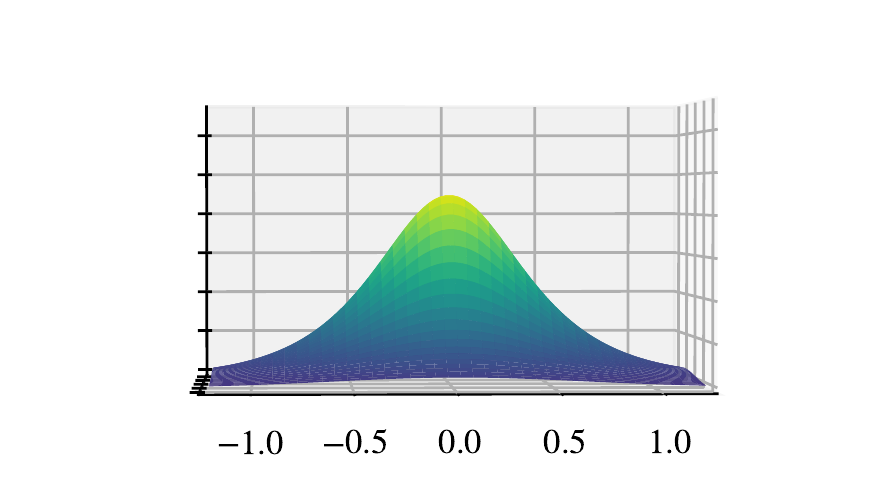}
\includegraphics[trim={2cm .1cm 1.6cm 1cm},clip,scale=.85]{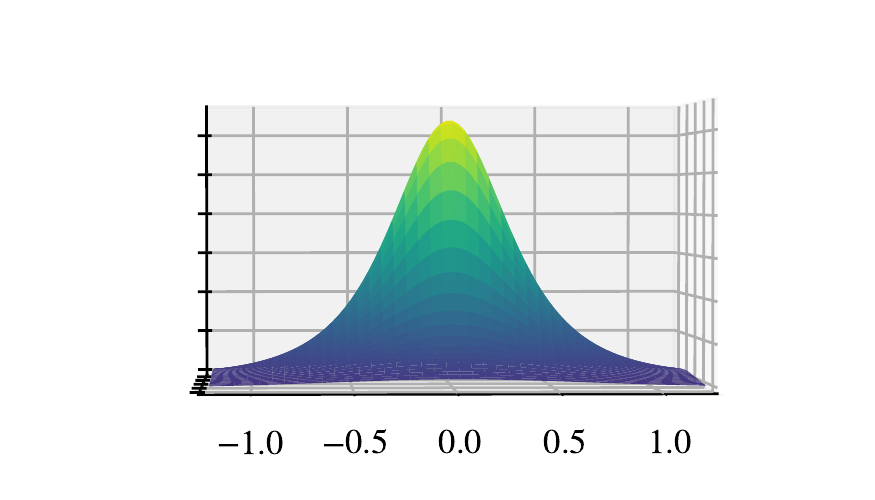}
\vspace{-.4cm}
\caption{Evolution of numerical solutions for the two-dimensional Keller--Segel equation with subcritical, critical, and supercritical initial data. } \label{2DKSden}
\end{figure}
\end{center}

\begin{center}
\begin{figure}[t]
\centering
\textbf{Two-Dimensional Keller--Segel Equation: Analysis of Blowup Behavior} 

\begin{flushleft} \hspace{1cm} Evolution of Second Moment, $h = 0.0\bar{3}$ \hspace{1.25cm} Convergence of Second Moment \end{flushleft} 

\vspace{-.1cm}
\hspace*{-1.9cm}
\begin{minipage}{.43\textwidth}
\begin{flushright}
\includegraphics[trim={.4cm 1.25cm .4cm .4cm},clip,scale=.85]{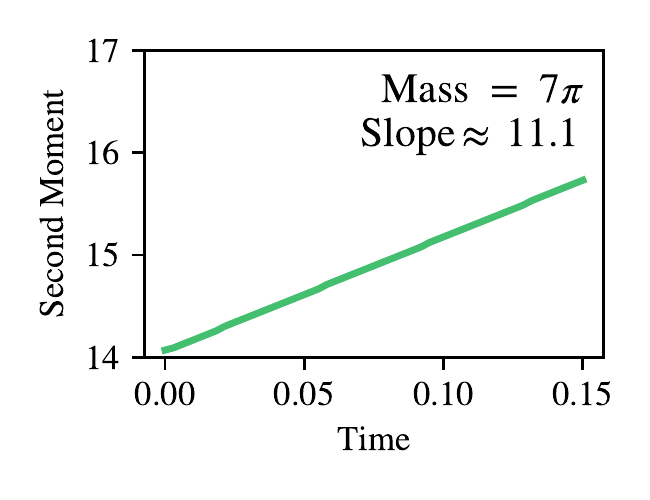}
\includegraphics[trim={.4cm 1.25cm .4cm .1cm},clip,scale=.85]{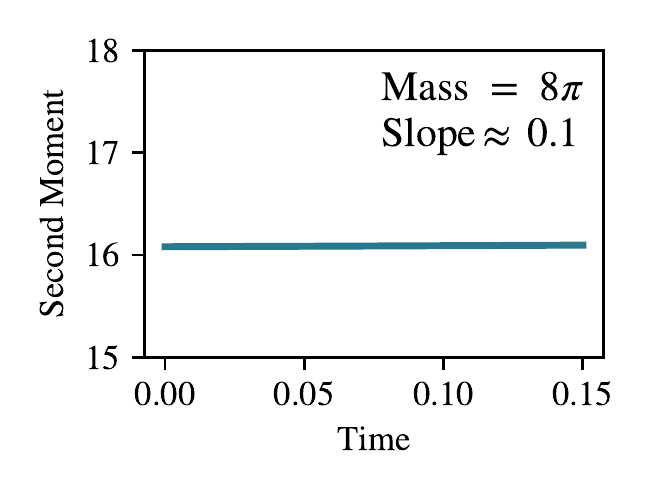}
\includegraphics[trim={.4cm .4cm .4cm .1cm},clip,scale=.85]{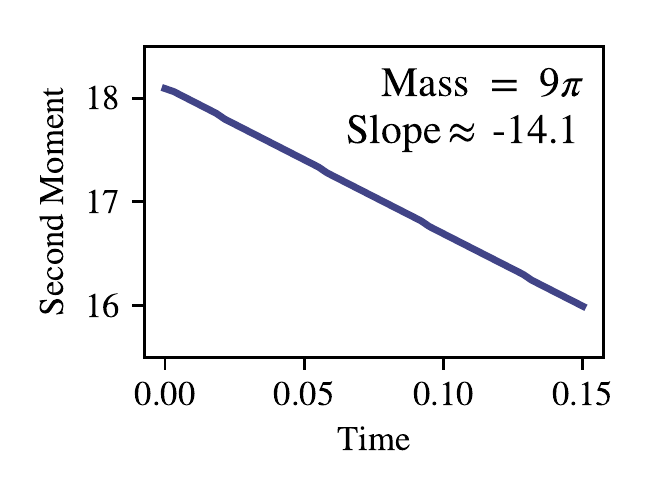}
\end{flushright}
\end{minipage} \qquad 
\begin{minipage}{.43\textwidth}
\vspace{.2cm}
\centering

\begin{flushleft}
\includegraphics[trim={0cm .4cm .3cm .4cm},clip,scale=.85]{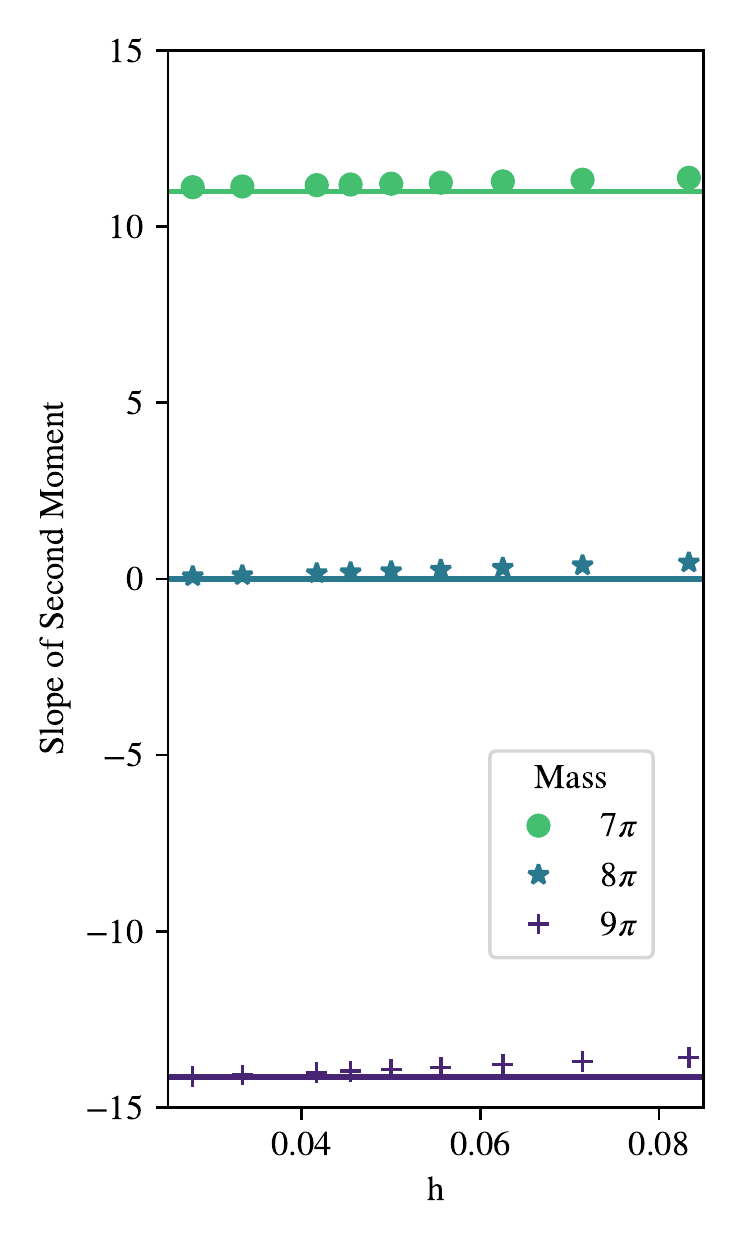}
\end{flushleft}
\end{minipage}
\vspace{-.5cm}
\caption{\textbf{Left:} Evolution of second moment of numerical solutions. \textbf{Right:} Convergence of slope of second moment to theoretically predicted slope (solid line).}  \label{2DKSmom}
\end{figure}
\end{center}

\begin{center}
\begin{figure}[t] 
\centering
\textbf{Two-Dimensional Keller--Segel Equation: Blowup with Supercritical Mass $9 \pi$} 

\begin{flushleft} \hspace{2cm} Evolution of Second Moment \hspace{2.5cm} Evolution of Particle Trajectories \end{flushleft}
\vspace{-.1cm}
\includegraphics[trim={1.5cm 1.95cm .8cm .9cm},clip,scale=.85]{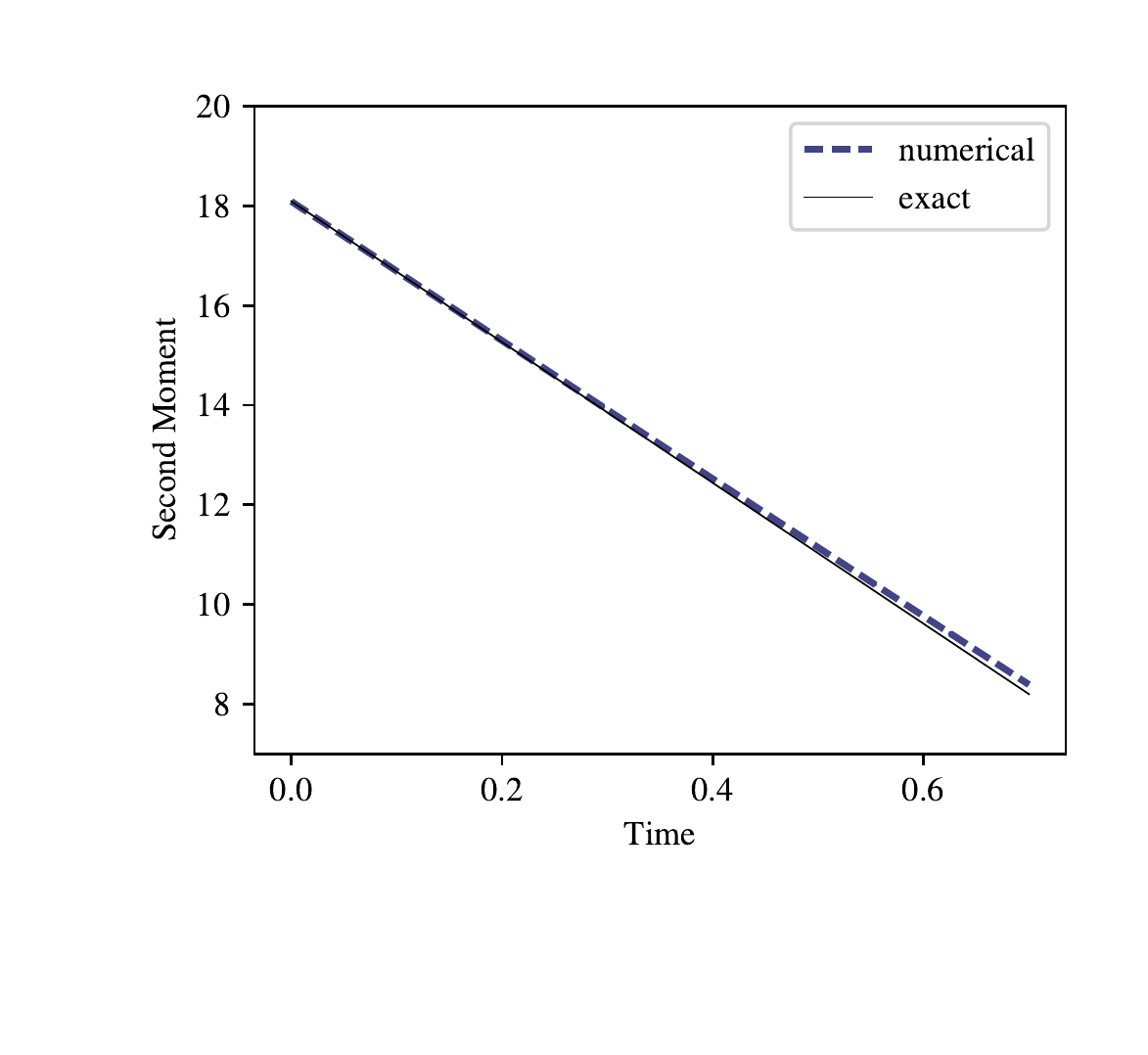}
\includegraphics[trim={2.8cm 1.6cm 3.9cm 3.1cm},clip,scale=.85]{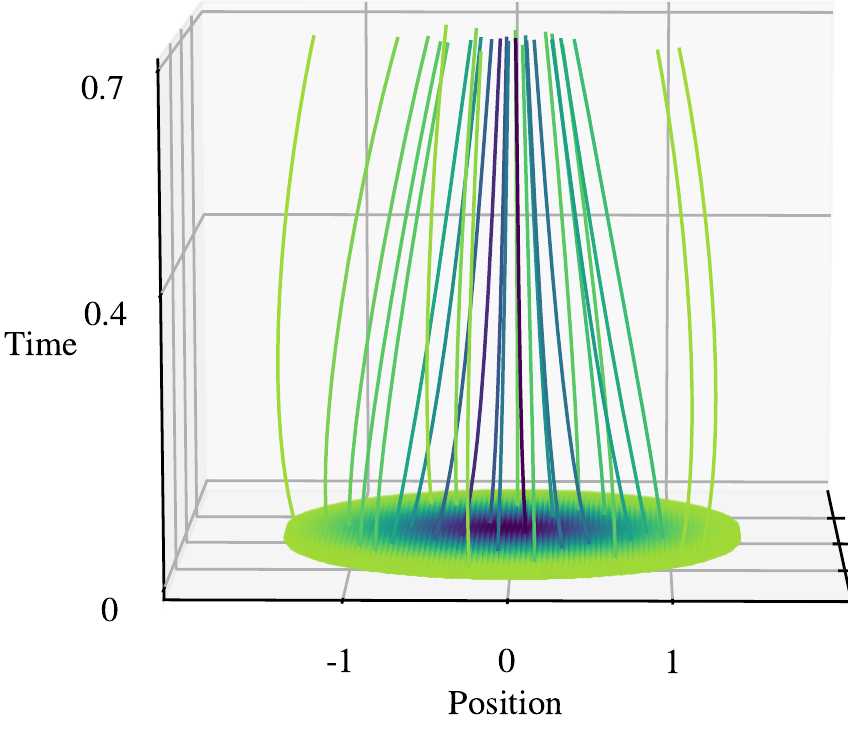}
\vspace{-.5cm}
\caption{ \textbf{Left:} Comparison of second moment of numerical solution (dashed line) to exact solution (solid line). \textbf{Right:} Evolution of particle trajectories, colored according to the relative mass of each trajectory.} \label{2DKSSup}
\end{figure}
\end{center}

\begin{center}
\begin{figure}[!ht]
\centering
\textbf{Two-Dimensional Keller--Segel Equation: Evolution of Density}
\bigskip
\begin{flushleft} \hspace{1.15in} t = 0.0 \hspace{1.15in} t = 0.075 \hspace{1.15in} t = 0.15 \hspace{.1in} \end{flushleft}

Subcritical Mass = $7\pi$

\hspace*{-1.2cm}
\includegraphics[trim={.83cm 1cm 1.6cm 1cm},clip,scale=.85]{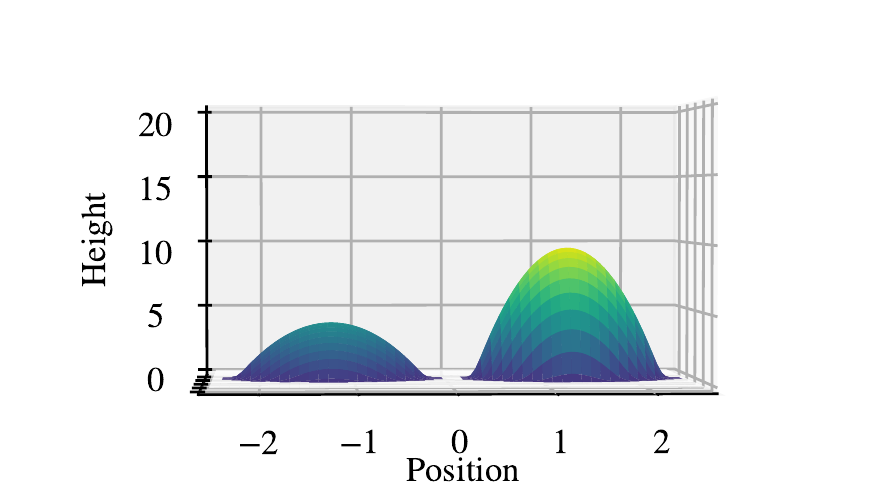}
\includegraphics[trim={2cm 1cm 1.6cm 1cm},clip,scale=.85]{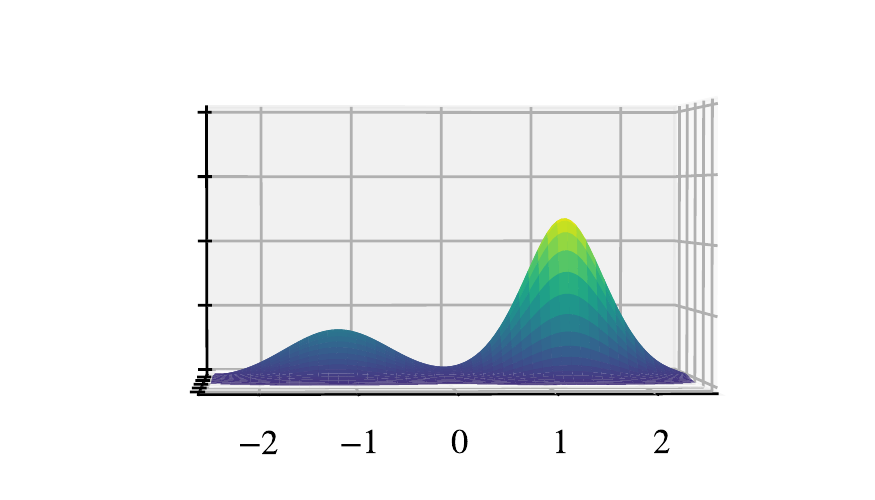}
\includegraphics[trim={2cm 1cm 1.6cm 1cm},clip,scale=.85]{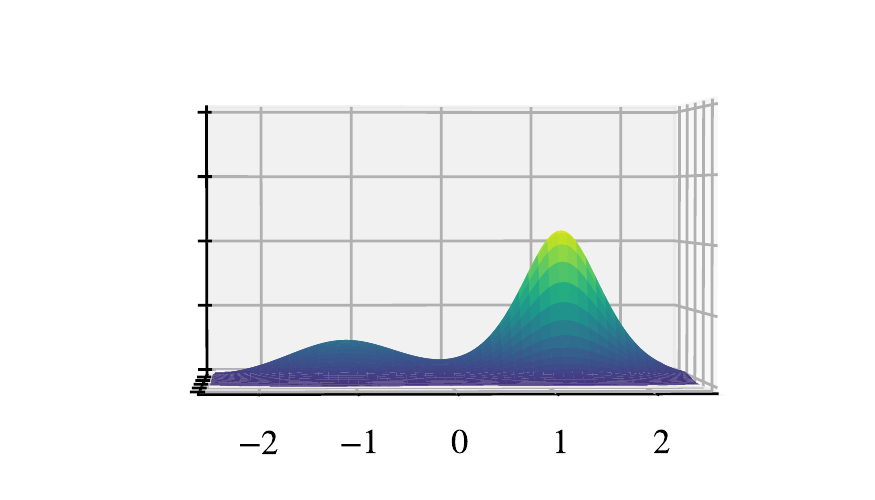}

Critical Mass = $8 \pi$

\includegraphics[trim={2.03cm 1cm 1.6cm 1cm},clip,scale=.85]{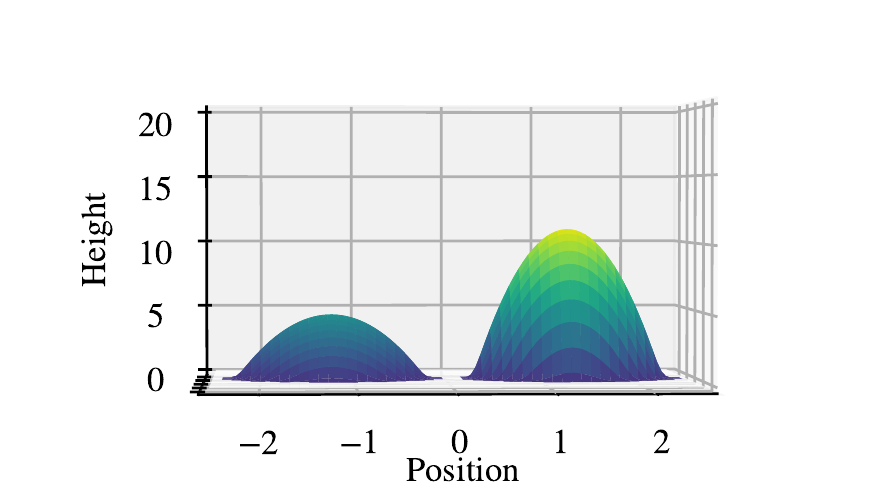}
\includegraphics[trim={2cm 1cm 1.6cm 1cm},clip,scale=.85]{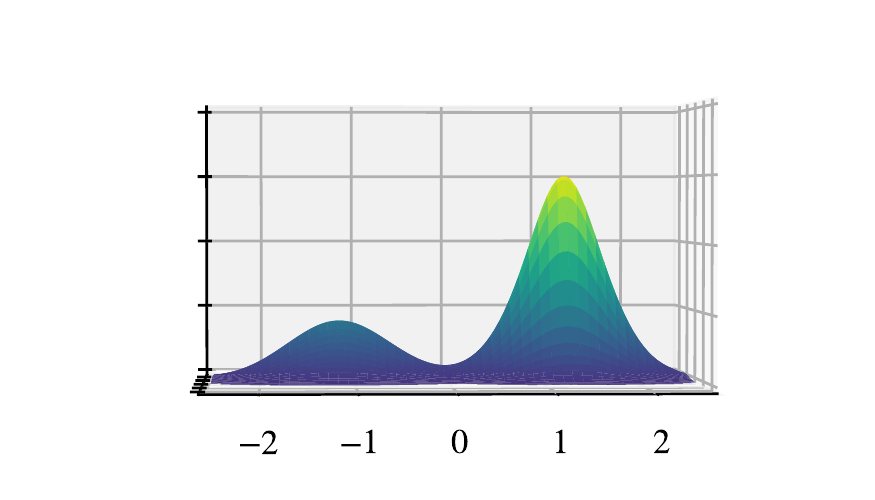}
\includegraphics[trim={2cm 1cm 1.6cm 1cm},clip,scale=.85]{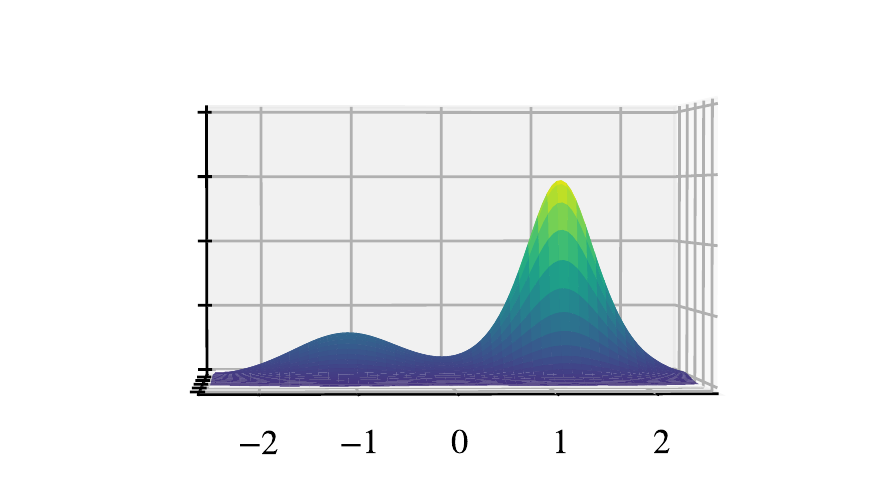}

Supercritical Mass = $9\pi$

\includegraphics[trim={2.03cm .1cm 1.6cm 1cm},clip,scale=.85]{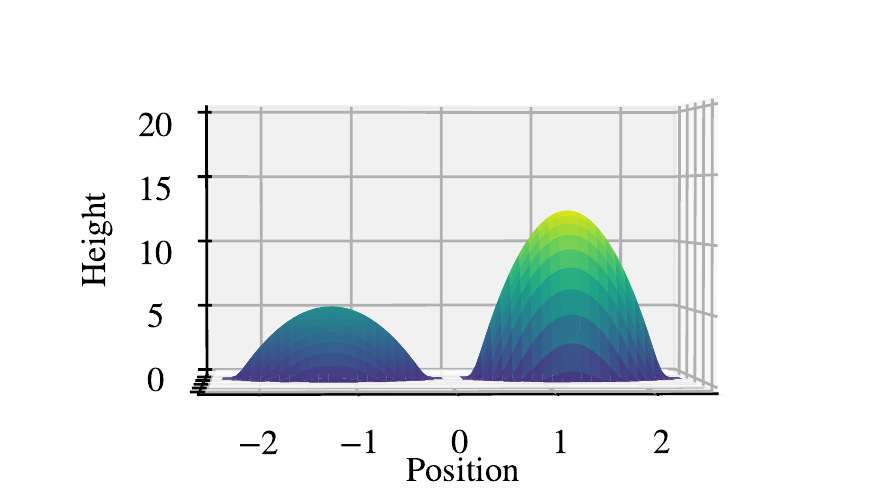}
\includegraphics[trim={2cm .1cm 1.6cm 1cm},clip,scale=.85]{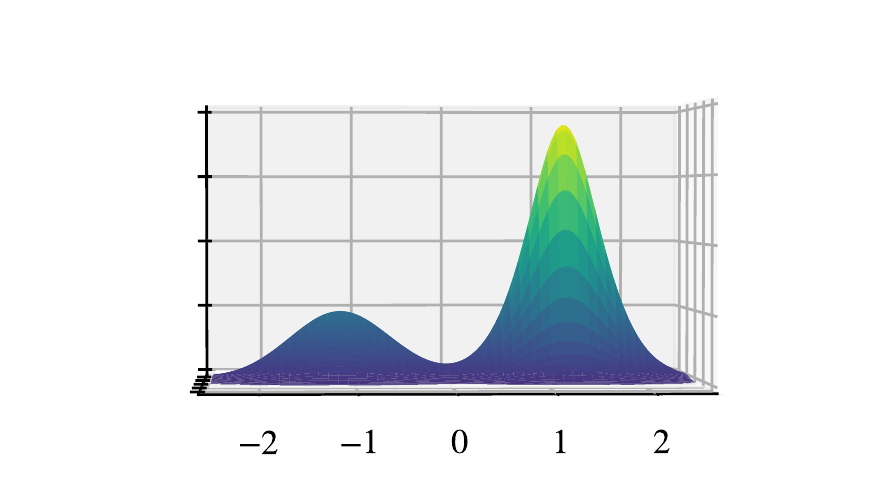}
\includegraphics[trim={2cm .1cm 1.6cm 1cm},clip,scale=.85]{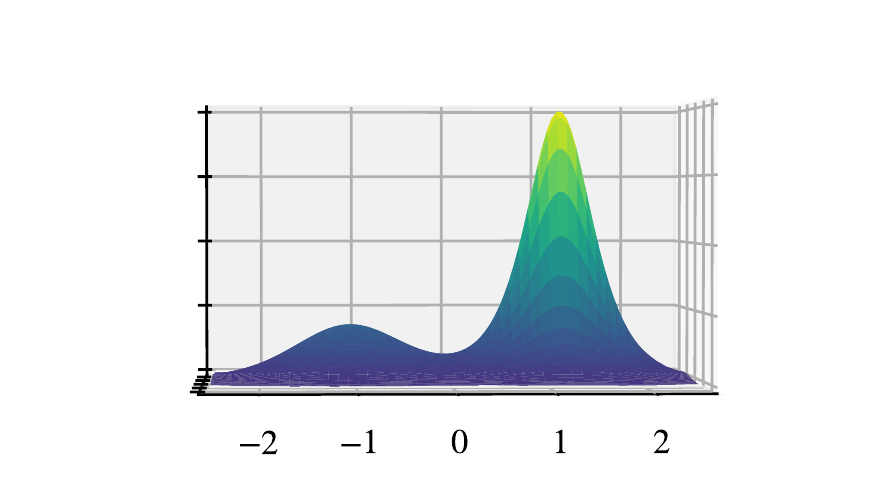}

Evolution of Second Moment

\hspace*{-0.8cm}
\includegraphics[trim={.4cm .2cm .4cm .1cm},clip,scale=.85]{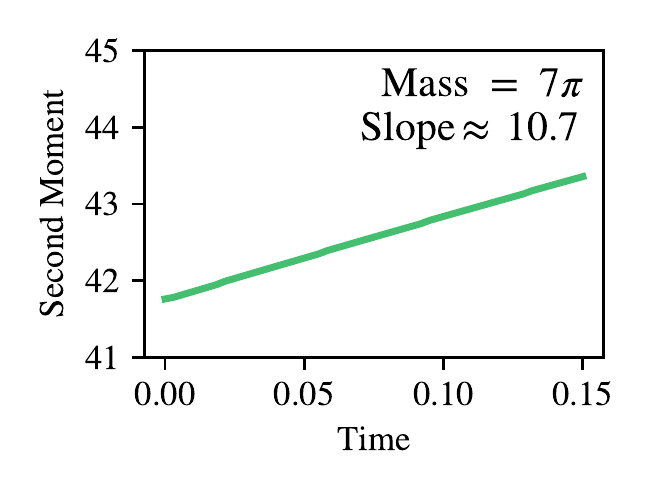}
\includegraphics[trim={.8cm .2cm .4cm .1cm},clip,scale=.85]{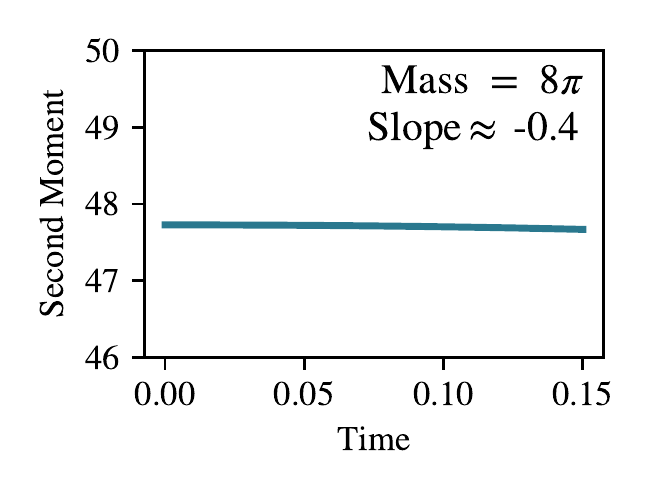}
\includegraphics[trim={.8cm .2cm .4cm .1cm},clip,scale=.85]{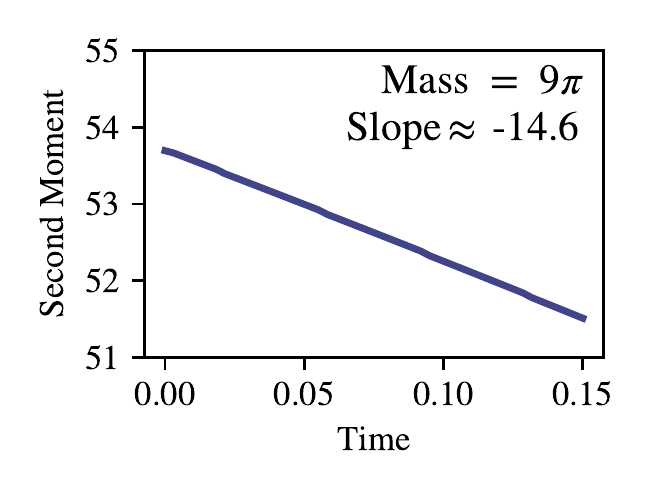}
\caption{\textbf{Top:} We plot the evolution of blob solutions to the two-dimensional Keller--Segel equation, with initial data given by constant multiples of the linear combination of Barenblatts from Figure \ref{FPfig}. In particular, we consider constant multiples $M = 7 \pi, 8 \pi,$ and $9 \pi$ and again observe that larger values of $M$ correspond to faster aggregation at the origin. \textbf{Bottom:} We consider the evolution of the second moment along particle solutions, for each choice of $M$. We estimate the slope of the line using the line of best fit. }  \label{2DKSCrazyDen}
\end{figure}
\end{center}

In Figures \ref{2DKSden}--\ref{2DKSCrazyDen} we consider the classical Keller--Segel equation ($V=0$, $W(\cdot) = 1/(2 \pi)\log\abs{\cdot}$, $m=1$) in two dimensions. In Figures \ref{2DKSden}, \ref{2DKSmom}, and \ref{2DKSSup}, the initial data is given by a Gaussian $\psi_1(\tau,\cdot)$, $\tau = 0.16$, scaled to have mass that is either supercritical ($> 8 \pi$), critical ($=8 \pi$), or subcritical ($< 8 \pi$) with respect to blowup behavior. In particular, for supercritical initial data, solutions blow up in finite time \cite{DP,BDP}. In Figure \ref{2DKSden}, we analyze the blow-up behavior. We compute the evolution of the second moment of solutions for fixed grid spacing $h = 0.0\bar{3}$ and varying mass $7 \pi, 8 \pi,$ and $9 \pi$, illustrating how initial data with larger mass aggregates more quickly at the origin.

In Figure \ref{2DKSmom}, we consider the evolution of the second moment for the solutions from Figure \ref{2DKSden}. For fixed grid spacing $h = 0.0\bar{3}$, we observe that the second moment depends linearly on time, and we compute its slope using the line of best fit. We then analyze how the slope of this line converges to the theoretically predicted slope as the grid spacing $h \to 0$.

In Figure \ref{2DKSSup}, we consider the evolution of the second moment for the supercritical mass solution from Figure \ref{2DKSden} on a longer time interval. As in the one-dimensional case (see Figure \ref{1DKellerSegel}), we are able to get approximately halfway to the time when the second moment becomes zero before the second moment of our numerical solution begins to peel away from the second moment of the exact solution. Indeed, one of the benefits of our blob method approach is that the numerical method naturally extends to two and more dimensions, and we observe similar numerical performance independent of the dimension. We also plot the evolution of particle trajectories, observing the tendency of trajectories in regions of larger mass to be driven largely by pairwise attraction, while trajectories in regions of lower mass feel more strongly the effects of diffusion.

Finally, in Figure \ref{2DKSCrazyDen}, we consider the evolution of the density and second moment for double bump initial data, with initial mass $7 \pi, 8 \pi,$ and $9 \pi$. The slopes of the second moment agree well with the theoretically predicted slopes given in Figure \ref{2DKSmom}.

\clearpage
\appendix

\section{Proofs of preliminary results} \label{appendix preliminaries}
We now turn to the proofs of some of the elementary lemmas and propositions from Sections \ref{preliminaries section} and \ref{energies section}.
We begin with the proof of the mollifier exchange lemma.
\begin{proof}[Proof of Lemma \ref{move mollifier prop}]
By the Lipschitz continuity of $f$,
\begin{align*}
	\left| \int \zeta_\e *(f\nu)\,d\sigma - \int (\zeta_\e *\nu)f\,d\sigma \right| &\leq \ird\ird \zeta_\e(x-y)  |f(x)-f(y)| \,d|\nu|(y) \,d|\sigma|(x) \\
	&\leq L_f  \ird\ird \zeta_\e(x-y) |x-y| \,d|\nu|(y) \,d|\sigma|(x)
\end{align*}
Set $p:= (q-d)/q >0$. Decomposing the domain of the integration of $|\nu|$ into $B_{\e^p}(x)$ and $\Rd \setminus B_{\e^p}(x)$, we may bound the above quantity by
\bes
	L_f  \ird \left( \int_{B_{\e^{p}}(x)} \zeta_\e(x-y) |x-y| \,d|\nu|(y) \,dx+\int_{\R^d \setminus B_{\e^{p}}(x)} \zeta_\e(x-y) |x-y| \,d|\nu|(y) \right) \,d|\sigma|(x) .
\ees
By the decay assumption on $\zeta$ (see Assumption \ref{mollifierAssumption}), for all $x,y \in \Rd$ with $|x-y|>\e^p$ we have
\be \label{zeta bound}
	\zeta_\e(x-y)|x-y| = \zeta\left( \frac{x-y}{\e} \right) \frac{|x-y|}{\e^d} \leq C_\zeta |x-y|^{1-q}\e^{q-d} \leq C_\zeta \e^{p}.
\ee
Thus, we conclude our result by estimating the above quantity by
\bes
	\e^{p} L_f  \ird (\zeta_\e*|\nu|)\, d|\sigma|(x)+ \e^{p} L_f C_\zeta |\sigma|(\Rd)|\nu|(\R^d). \qedhere
\ees
\end{proof}

We now give the proof that if $\mu_\e \wsto \mu$, then $\varphi_\e*\mu_\e \wsto \mu$.

\begin{proof}[Proof of Lemma \ref{weakst convergence mollified sequence}]
	By \cite[Remark 5.1.6]{AGS}, it suffices to show that $\varphi_\e *\mu_\e$ converges to $\mu$ in distribution, that is, in the duality with smooth, compactly supported functions. For all $f \in C_\mathrm{c}^\infty(\Rd)$,
\begin{align*}
	\left| \ird f \,d (\varphi_\e *\mu_\e) - \ird f \,d\mu \right| &\leq \left| \ird f \,d (\varphi_\e *\mu_\e) - \ird f \,d\mu_\e \right| + \left| \ird f \,d\mu_\e - \ird f \,d\mu \right|
\end{align*}
Since $\mu_\e \wsto \mu$, the second term goes to zero. We bound the first term as follows:
\begin{align*}
& \left| \ird f \,d (\varphi_\e *\mu_\e) - \ird f \,d\mu_\e \right| = \left| \ird\ird (f(y) - f(x)) \varphi_\e(x-y)\,dy \,d \mu_\e(x) \right| \\
 &\quad \leq \|\grad f\|_{L^\infty(\Rd)} \ird\ird |x-y| \varphi_\e(x-y) \,dy \,d\mu_\e(x) = \|\grad f\|_{L^\infty(\Rd)} \ird\ird \left| \frac{z}{\e^d} \right| \varphi \left(\frac{z}{\e} \right) \,dz \,d\mu_\e(x) \\
 &\quad = \e \|\grad f\|_{L^\infty(\Rd)} \ird |z | \varphi(z) \,dz ,
\end{align*}
which goes to zero as $\e \to 0$.
\end{proof}

Next, we prove the inequalities relating the regularized internal energies to the unregularized internal energies.
\begin{proof}[Proof of Proposition \ref{relative sizes lemma}]
We begin with (\ref{relative sizes equation}). To prove the left inequality, we may assume without loss of generality that $\mu \in D(\F)$.
First, we show the result for the entropy ($m=1$). Note that
\bes
	\F^1(\mu)-\F^1_\e(\mu) =  \mathcal{H}(\mu | \varphi_\e *\mu) ,
\ees
where $\mathcal{H}$ is the relative entropy; that is, for all $\nu \in \P(\Rd)$, 
\begin{align*}
	\mathcal{H}(\mu | \nu) : = \begin{cases} \int \log \left( \frac{d \mu}{d \nu} \right) d \mu &\text{if } \mu \ll \nu, \\ + \infty &\text{otherwise.}\end{cases}
\end{align*}
By Jensen's inequality for the convex function $s \mapsto s \log s$,  the relative entropy is nonnegative, which gives the result. Now, we show the left inequality in \eqref{relative sizes equation} for $1<m\leq 2$. By the above-the-tangent property of the concave function $F_m$ and H\"older's inequality, we get
\begin{align*}
\F^m(\mu) - \F^m_\e(\mu) &= \frac{1}{m-1} \int \left( \mu^{m-1} - (\varphi_\e*\mu)^{m-1} \right)\, d \mu \geq \int \left(\mu - \varphi_\e*\mu \right) \mu^{m-2}\, d \mu \\ 
&\geq - \|\mu - \varphi_\e* \mu\|_{L^m(\Rd)} \|\mu^{m-1}\|_{L^{m/(m-1)}(\Rd)} = - \|\mu - \varphi_\e* \mu\|_{L^m(\Rd)} \|\mu\|_{L^m(\Rd)}^{m-1} .
\end{align*}
Since $\mu \in D(\F^m)$ implies $\mu \in L^m(\Rd)$, the first term goes to zero as $\e \to 0$ and the second term remains bounded. This gives the result.

We now turn to the right inequality in \eqref{relative sizes equation} in the case $1\leq m\leq2$. By the fact that $\varphi_\e = \zeta_\e * \zeta_\e$ and Jensen's inequality for the concave function $F_m$, for all $x \in \Rd$ we have
\begin{align*}
	F_m(\varphi_\e * \mu(x)) &= F_m \left( \ird \zeta_\e(y) \zeta_\e*\mu(x-y) \, dy \right)\\
	&\geq  \ird \zeta_\e(y) F_m \left(\zeta_\e*\mu(x-y) \right) \, dy   =   \zeta_\e* \left(F_m \circ \left(\zeta_\e*\mu \right)\right)(x).
\end{align*}
Consequently, we deduce
\begin{align*}
	\F_\e^m(\mu) = \ird F_m(\varphi_\e * \mu(x))\, d \mu(x) &\geq \ird \zeta_\e* \left( F_m \circ (\zeta_\e*\mu)\right)(x) \, d \mu(x)\\
	&= \ird   F_m(\zeta_\e*\mu(x)) \, d (\zeta_\e*\mu)(x) = \F^m \left( \zeta_\e*\mu  \right).
\end{align*}

Now, we show (\ref{relative sizes equation 2}). Since $F_m$ is convex for $m\geq2$, this is simply a consequence of reversing the inequalities in the last two inequalities.

Finally, we consider the lower bounds (\ref{lower bounds}). When $m=1$, these follow from the right inequality in \eqref{relative sizes equation}, a Carleman-type estimate \cite[Lemma 4.1]{CPSW} ensuring that $\F_\e^m(\zeta_\e*\mu) \geq -(2\pi/\delta)^{d/2} - \delta M_2(\zeta_\e*\mu)$ for all $\delta>0$, and the fact that
\[ 
	\ird \zeta_\e(y) |x+y|^2 \,dy \leq 2|x|^2 + 2M_2(\zeta_\e) \implies M_2(\zeta_\e*\mu) \leq 2M_2(\mu) + 2M_2(\zeta_\e) = 2M_2(\mu) + 2\e^2 M_2(\zeta).
\]
When $m>1$, we simply use that $F_m\geq 0$.
\end{proof}

We now give the proof that, for all $\e>0$, the regularized energies are lower semicontinuous with respect to weak-* convergence ($m>1$) and Wasserstein convergence ($m=1$), where in the latter case, we require $\varphi$ to be a Gaussian.

\begin{proof}[Proof of Proposition \ref{lower semicontinuity}]
First, we note that for any sequence $(\mu_n)_n \subset \P(\Rd)$ and $\mu \in\P(\Rd)$ such that $\mu_n \wsto \mu$ and any sequence $x_n \to x$, we have
\begin{align}
&\left| \varphi_\e * \mu_n(x_n) - \varphi_\e*\mu(x) \right| \label{gammaliminfpoint} \\
&\quad = \left| \int \varphi_\e (x_n -y) d \mu_n(y) - \int \varphi_\e(x-y) d \mu(x) \right|  \nonumber \\
&\quad \leq  \left| \int \left( \varphi_\e (x_n -y) -  \varphi_\e (x -y) \right) d \mu_n(y)  \right| + \left| \int \varphi_\e(x-y) d \mu_n(y) - \int \varphi_\e(x-y) d \mu(x) \right| \nonumber \\
&\quad \leq  |x_n - x| \|\grad \varphi_\e\|_\infty  + \left| \int \varphi_\e(x-y) d \mu_n(y) - \int \varphi_\e(x-y) d \mu(x) \right| \xrightarrow{n \to +\infty} 0, \nonumber
\end{align}
since $\varphi_\e(x - \cdot) \in C_b(\Rd)$.

We now show \eqref{it:lsc-pme}. Suppose $\mu_n \wsto \mu$.  By Lemma \ref{lem:fatou-varying}, we have
\[ \liminf_{n\to\infty} \F^m_\e(\mu_n) = \liminf_{n \to +\infty} \frac{1}{m-1} \int_\Rd (\varphi_\e *\mu_n)^{m-1} d \mu_n \geq  \frac{1}{m-1} \int_\Rd  \liminf_{n \to +\infty, x' \to x} (\varphi_\e *\mu_n(x'))^{m-1} d \mu(x) . \]
By inequality (\ref{gammaliminfpoint}),
\[  \liminf_{n \to +\infty, x' \to x} (\varphi_\e *\mu_n(x'))^{m-1} = (\varphi_\e *\mu(x))^{m-1} . \]
Combining the two previous inequalities, we obtain $\liminf_{n\to\infty} \F^m_\e(\mu_n) \geq \F^m_\e(\mu)$, giving the result.

Next, we show \eqref{it:lsc-he}. Suppose $\mu_n \to \mu$ in the Wasserstein metric. Since $\varphi$ is a Gaussian, there exist $x_0\in\R^d$ and $C_0,C_1\in\R$ so that, for $n$ sufficiently large,
\be\label{eq:quadratic-ineq}
	\log(\varphi_{\e}*\mu_n(x)) \geq C_0|x-x_0|^2 + C_1,
\ee
Define  $f_n:= \log(\varphi_{\e}* \mu_n)$ and $q(\cdot):= C_0|\cdot-x_0|^2 + C_1$. Then, by Lemma \ref{lem:fatou-varying}, we have
\begin{align} \label{firstB3}
\liminf_{n \to +\infty} \int_\Rd (f_n(x) - q(x)) d \mu_n(x) \geq \int_\Rd  \liminf_{n \to +\infty, x' \to x}  (f_n(x') - q(x')) d \mu(x) .
\end{align}
Since $\mu_n \to \mu$ in the Wasserstein metric,
\begin{align} \label{secondB3}
\lim_{n \to +\infty}  \int_\Rd (-q(x)) d\mu_n(x) = \int_\Rd (-q(x)) d \mu(x) = \int_\Rd \liminf_{n \to +\infty, x' \to x} ( - q(x')) d \mu(x) .
\end{align}
Furthermore, by (\ref{gammaliminfpoint}) and the fact that $\log(\cdot)$ is continuous on $(0, +\infty)$,
\begin{align} \label{thirdB3}
\liminf_{n \to +\infty, x' \to x}  f_n(x') = \liminf_{n \to +\infty, x' \to x}  \log(\varphi_\e*\mu_n(x')) = \log(\varphi_\e*\mu(x)) .
\end{align}
Thus, combining (\ref{firstB3}), (\ref{secondB3}), and (\ref{thirdB3}), we obtain,
\[ \F^1_\e(\mu_n) = \liminf_{n \to +\infty} \int_\Rd f_n(x)   d \mu_n(x)  \geq \int_\Rd \log(\varphi_\e*\mu(x))  d \mu(x) = \F^1_\e(\mu) , \]
which gives the result. \qedhere

\end{proof}

Now we turn to the proof that the regularized energies are differentiable along generalized geodesics.

\begin{proof}[Proof of Proposition \ref{diff prop}]
	By definition, for all $\alpha \in[0,1]$,
\bes
	\F_\e(\mu_\alpha^{2\to3}) = \iint F\left(\varphi_\e * \mu_\alpha((1-\alpha)x + \alpha y)\right) \,d\gamma(x,y).
\ees
Therefore, we deduce
\begin{align*}
	\F_\e(\mu_\alpha^{2\to3}) - \F_\e(\mu_2) &= \!\iiint\!\left(  F\left(\varphi_\e * \mu_\alpha^{2\to3}((1-\alpha)y + \alpha z))\right) - F\left(\varphi_\e * \mu_1(y)\right)  \right) \,d \gamma(x,y,z)\\
	&= \!\int_0^1 \!\!\iiint \! F'(c_{s,\alpha}(y,z)) \left( \varphi_\e * \mu_\alpha^{2\to3}((1-\alpha)y + \alpha z)) - \varphi_\e * \mu_1(y) \right) \,d\gamma(x,y,z) \,ds,
\end{align*}
where $c_{s,\alpha}(y,z) = (1-s)\varphi_\e*\mu_1(y) + s\varphi_\e*\mu_\alpha^{2\to3}((1-\alpha)y + \alpha x)$. Using Taylor's theorem compute
\begin{align*}
	\varphi_\e * \mu_\alpha^{2\to3}((1-\alpha)y + \alpha z)) &- \varphi_\e * \mu_1(y)\\
	&= \iiint \left( \varphi_\e((1-\alpha) (y-v) + \alpha (z-w)) - \varphi_\e*(y-v) \right) \,d\gamma(u,v,w)\\
	&= \iiint \left( \alpha \grad \varphi_\e(y-v) \cdot (z-w - (y-v)) + D_\alpha(y,z,v,w)\right) \,d\gamma(u,v,w),
\end{align*}
where $D_\alpha(y,z,v,w)$ is a term depending on the Hessian of $\varphi_\e$ satisfying
\begin{align*}
	\left|\iiint  D_\alpha(y,z,v,w) \,d\gamma(u,v,w)\right| &\leq \frac{\alpha^2}{2} \norm{D^2\varphi_\e}_{L^\infty(\Rd)} \iint |z-w-(y-v)|^2 \,d\gamma(u,v,w)\\
	&\leq 2\alpha^2 \norm{D^2\varphi_\e}_{L^\infty(\Rd)} \left( |z|^2 + |y|^2 +\! \int |w|^2\,d\mu_3(w) +\!\int|v|^2 \,d\mu_2(v) \right)
\end{align*}
Hence, since $F'$ is nondecreasing,
\begin{align*}
	&\F_\e(\mu_\alpha^{2\to3}) - \F_\e(\mu_2)\\
	 &= \alpha \int_0^1 \iiint \iiint F'(c_{s,\alpha}(y,z)) \grad \varphi_\e(y-v) \cdot (z-w-(y-v)) \,d\gamma(u,v,w) \,d\gamma(x,y,z)\,ds + C_\alpha, 
\end{align*}
where $|C_\alpha| \leq 4\alpha^2 \|D^2\varphi_\e\|_{L^\infty(\Rd)} F'( \|\varphi_\e\|_{L^\infty(\Rd)}) (\int |x|^2 \,d\mu_2(x) + \int |x|^2 \,d\mu_3(x))$. Note that $c_{s,\alpha}(y,z)$ converges pointwise to $\varphi_\e*\mu_2(y)$ as $\alpha\to0$ since
\begin{align*} 
	\left|\varphi_\e * \mu_\alpha^{2\to3}((1-\alpha)y +\alpha z) \right.&\left.- \varphi_\e * \mu_2(y) \right|\\
	&= \left|\iiint \left(\varphi_\e((1-\alpha)(y-v) + \alpha (z - w)) - \varphi_\e(y-v) \right) \,d \gamma(u,v,w) \right|\\
	& \leq \alpha \|\grad \varphi_\e\|_{L^\infty(\Rd)} \left( |z|+|y| + \int |w| \,d\mu_3(w) + \int |v| \,d\mu_2(v) \right).
\end{align*}
Thus, to complete the result, it suffices to show that there exists $g \in L^1(\gamma \otimes \gamma)$ so that 
\[
	F'(c_{s,\alpha}(y,z)) \left| \grad \varphi_\e(y-v) \cdot (z-w-(y-v))  \right| \leq g(y,z,v,w),
\]
since the result then follows by the dominated convergence theorem. Since $F'$ is nondecreasing we may take
\bes
	g(y,z,v,w) = F'\left(\norm{\varphi_\e}_{L^\infty(\Rd)}\right) \norm{\grad \varphi_\e}_{L^\infty(\Rd)} |z-w-(y-v)|, 
\ees
which ends the proof.
\end{proof}

Next, we apply the result of the previous proof to characterize the subdifferential of the regularized energies.

\begin{proof}[Proof of Proposition \ref{subdiffchar}]
Suppose $v$ is given by equation (\ref{subdiffform}). This part of the proof is closely inspired by that of \cite[Proposition 2.2]{5person}. For all $x,y \in \R^d$ define $G(\alpha) = F(\varphi_\e*\mu_\alpha((1-\alpha)x + \alpha y))$ for all $\alpha \in [0,1]$, where $\mu_\alpha = ((1-\alpha)\pi^1 + \alpha \pi^2)_\#\gamma$, with some $\gamma \in \Gamma_\mathrm{o}(\mu,\mu_1)$, connects $\mu_0=\mu$ and $\mu_1$. Now define
\bes
	f(\alpha) = \frac{G(\alpha) - G(0)}{\alpha} - \frac{\lambda \alpha}{2}\left(|x-y|^2 + W_2^2(\mu_0,\mu_1)\right) \quad \mbox{for all $\alpha \in [0,1]$,}
\ees
where $\lambda = -2F'(\|\varphi_\e\|_{L^\infty(\Rd)})\|D^2\varphi_\e\|_{L^\infty(\Rd)} = \lambda_F$; see \eqref{eq:lambda-F}. We write $[a,b]_\alpha := (1-\alpha)a +\alpha b$ for any $a,b\in\R^d$. Let us compute the first two derivatives of $G$ for all $\alpha \in [0,1]$:
\be\label{eq:G'}
	G'(\alpha) = F'(\varphi_\e*\mu_\alpha([x,y]_\alpha)) \irdrd (y-x  + u-v) \cdot \nabla \varphi_\e([x-u,y-v]_\alpha)\,d\gamma(u,v),
\ee
and
\begin{align*}
	G''(\alpha) &= F''(\varphi_\e*\mu_\alpha([x,y]_\alpha)) \left( \irdrd (y-x  + u-v) \cdot \nabla \varphi_\e([x-u,y-v]_\alpha)\,d\gamma(u,v) \right)^2\\
	&\phantom{{}={}}+\! F'(\varphi_\e*\mu_\alpha([x,y]_\alpha)) \irdrd \! (y-x  + u-v) D^2\varphi_\e([x-u,y-v]_\alpha)(y-x  + u-v) \,d\gamma(u,v).
\end{align*}
Since $F'' \geq0$, $F'\geq0$ and $\norm{D^2\varphi_\e}_{L^\infty(\Rd)}$ is finite, we have
\begin{equation} \label{eq:Gsecond-ineq} \begin{split}
	G''(\alpha) &\geq - F'(\norm{\varphi_\e}_{L^\infty(\Rd)}) \norm{D^2\varphi_\e}_{L^\infty(\Rd)} \irdrd |y-x+u-v|^2 \,d\gamma(u,v) \\
	& \geq -2 F'(\norm{\varphi_\e}_{L^\infty(\Rd)}) \norm{D^2\varphi_\e}_{L^\infty(\Rd)} \irdrd \left(|y-x|^2+|u-v|^2\right) \,d\gamma(u,v)\\
	&= \lambda \left(|y-x|^2 + W_2^2(\mu_0,\mu_1)\right).
\end{split}\end{equation}
Now, by Taylor's theorem, 
\bes
	f(\alpha) = G'(0) + \int_0^\alpha \frac{\alpha-s}{\alpha} G''(s) \,d s - \frac{\lambda \alpha}{2}\left( |x-y|^2 + W_2^2(\mu_0,\mu_1) \right),
\ees
and therefore, using \eqref{eq:Gsecond-ineq} leads to
\bes
	f'(\alpha) = \frac{1}{\alpha^2} \int_0^\alpha sG''(s) \,ds - \frac{\lambda}{2}\left(|x-y|^2 + W_2^2(\mu_0,\mu_1) \right) \geq 0,
\ees
which shows that $f$ is nondecreasing, and so $f(1) \geq \lim_{\alpha \to 0} f(\alpha)$, which implies (after integrating against $d\gamma(x,y)$)
\begin{align*}
	\F_\e(\mu_1) - \F_\e(\mu_0) &\geq \irdrd \lim_{\alpha \to 0} \left(\frac{G(\alpha) - G(0)}{\alpha} \right) \,d\gamma(x,y) + \lambda W_2^2(\mu_0,\mu_1)\\
	&= \irdrd G'(0) \,d\gamma(x,y) + \lambda W_2^2(\mu_0,\mu_1).
\end{align*}
Then, by \eqref{eq:G'} and antisymmetry of $\nabla \varphi_\e$, compute
\begin{align*}
	\irdrd G'(0) \,d\gamma(x,y) &= \irdrd \irdrd  F'(\varphi_\e*\mu_0(x)) (y-x  + u-v) \cdot \nabla \varphi_\e(x-u)\,d\gamma(u,v) \,d\gamma(x,y)\\
	&= \irdrd F'(\varphi_\e*\mu_0(x)) \nabla\varphi_\e*\mu_0(x)\cdot(y-x)\,d\gamma(x,y)\\
	&\phantom{{}={}} + \irdrd \nabla\varphi_\e* (F'\circ (\varphi_\e*\mu_0)\mu_0)(u) \cdot (v-u) \,d\gamma(u,v)\\
	&= \irdrd \nabla \frac{\delta \F_\e}{\delta \mu_0}(x) \cdot (y-x) \,d\gamma(x,y).
\end{align*}
Hence 
\bes
	\F_\e(\mu_1) - \F_\e(\mu_0) \geq \irdrd \nabla \frac{\delta \F_\e}{\delta \mu_0}(x) \cdot (y-x) \,d\gamma(x,y) + \lambda W_2^2(\mu_0,\mu_1),
\ees
which shows that $\delta \F_\e/\delta \mu_0 \in \partial  \F_\e(\mu_0)$. We now prove that $v \in \Tan_\mu\P_2(\R^d)$. Consider a vector-valued function $\xi  \in C_\mathrm{c}^\infty(\R^d)^d$, and for any $x,y\in\R^d$ define $H(\alpha) = F(\ird \varphi_\e(x-y + \alpha(\xi (x) - \xi (y))\,d\mu(y))$ for all $\alpha\in[0,1]$. Then
\bes
	H'(0) = F'(\varphi_\e*\mu(x)) \ird (\xi (x) - \xi (y)) \cdot \nabla \varphi_\e(x-y) \,d\mu(y).
\ees
Now compute, using the antisymmetry of $\nabla \varphi_\e$,
\begin{align*}
	\lim_{\alpha \to0} \frac{\F_\e((\id + \alpha \xi )_\#\mu) - \F_\e(\mu)}{\alpha} &= \lim_{\alpha\to0} \ird \frac{H(\alpha)-H(0)}{\alpha} \,d\mu(x) = \ird H'(0) \,d\mu(x)\\
	&= \ird F'(\varphi_\e*\mu(x)) \nabla \varphi_\e*\mu(x) \cdot \xi (x) \,d\mu(x)\\
	&\phantom{{}={}}+ \ird \nabla\varphi_\e * (F'\circ(\varphi_\e*\mu)\mu)(x) \cdot \xi (x) \,d\mu(x)\\
	&= \ird \nabla \frac{\delta \F_\e}{\delta \mu}(x) \cdot \xi (x) \, d\mu(x),
\end{align*}
where passing the limit $\alpha \to0$ inside the integral in the first line is justified by the fact that $H'$ is bounded. Then, by the definition of the local slope of $\F_\e$,
\bes
	\liminf_{\alpha \to 0} \frac{\F_\e((\id + \alpha \xi )_\#\mu) - \F_\e(\mu)}{W_2((\id + \alpha \xi )_\#\mu,\mu)} \geq - |\partial \F_\e|(\mu).
\ees
Therefore, by the previous computation,
\bes
	\ird \nabla \frac{\delta \F_\e}{\delta \mu}(x) \cdot \xi (x) \, d\mu(x) \geq -|\partial \F_\e|(\mu) \liminf_{\alpha\to0} \frac{W_2((\id + \alpha \xi )_\#\mu,\mu)}{\alpha} \geq -|\partial \F_\e|(\mu) \|\xi \|_{L^2(\mu;\R^d)},
\ees
since, by definition of the 2-Wasserstein distance,
\bes
	\limsup_{\alpha\to0}  \frac{W_2((\id + \alpha \xi )_\#\mu,\mu)}{\alpha} \leq \|\xi \|_{L^2(\mu;\R^d)}.
\ees
Then, by replacing $\xi $ with $-\xi $, by arbitrariness of $\xi $ and by density of $C_\mt{c}^\infty$ in $L^2(\mu;\Rd)$, we get
\bes
	\left\|v\right\|_{L^2(\mu;\Rd)} = \left\| \nabla \frac{\delta \F_\e}{\delta\mu}\right\|_{L^2(\mu;\R^d)} \leq |\partial \F_\e|(\mu),
\ees
which shows the desired result. Since $|\partial \F_\e|(\mu)$ is the unique minimal norm element of $\partial \F_\e$, this also shows that we actually have equality in the right-hand side above.

Suppose now that $v \in \partial \F_\e (\mu) \cap \Tan_\mu\P_2(\R^d)$. Fix $\psi \in C_\mathrm{c}^\infty(\Rd)$ and define $\mu_\alpha = (\id + \alpha \grad \psi)_\# \mu$ and $\hat \mu_\alpha = (\id - \alpha \grad \psi)_\# \mu$ for all $\alpha \in [0,1]$. For $\alpha$ sufficiently small, $x^2/2+ \alpha \psi(x)$ is convex and $\id + \alpha \grad \psi$ is the optimal transport map from $\mu$ to $\mu_\alpha$, so $\Gamma_\mathrm{o}(\mu,\mu_\alpha) = \{\id \times (\id + \alpha \grad \psi) \}$. Similarly, $\Gamma_\mathrm{o}(\hat\mu_\alpha,\mu) = \{\id \times (\id - \alpha \grad \psi) \}$. Since $v \in \partial \F^m_\e (\mu)$, taking $\nu = \mu_\alpha$ in Definition \ref{subdiffdef} of the subdifferential, for $\alpha$ sufficiently small, gives
\[ 
	\F_\e(\mu_\alpha) - \F_\e(\mu) \geq  \int \la v,\alpha \grad \psi \ra d \mu + o(\alpha \|\grad \psi\|_{L^2(\mu)}), 
\]
and
\[ 
	\F_\e(\hat \mu_\alpha) - \F_\e(\mu) \leq  \int \la v,\alpha \grad \psi \ra d \mu + o(\alpha \|\grad \psi\|_{L^2(\mu)}), 
\]
Combining this with Proposition \ref{diff prop}, we obtain
\begin{align*}
	\int \la v, \grad \psi \ra d \mu \leq \lim_{\alpha \to 0} \frac{\F_\e(\mu_\alpha) - \F_\e(\mu)}{\alpha} &= \left. \frac{d}{d\alpha} \F_\e(\mu_\alpha) \right|_{\alpha = 0} = \left. \frac{d}{d\alpha} \F_\e(\hat \mu_\alpha) \right|_{\alpha = 0}\\
	&= \lim_{\alpha \to 0^-} \frac{\F_\e(\hat\mu_\alpha) - \F_\e(\mu)}{\alpha} \leq \int \la v, \grad \psi \ra d \mu.
\end{align*}
Rewriting the expression from equation (\ref{diff lem eqn}) gives
\[ 
	\int \la v, \grad \psi \ra d \mu = \left. \frac{d}{d\alpha} \F_\e(\mu_\alpha) \right|_{\alpha = 0} = \int \left\langle  \grad \varphi_\e* \left(F' \circ (\varphi_\e * \mu) \mu \right) + F'(\varphi_\e * \mu)\grad \varphi_\e * \mu, \grad \psi  \right \rangle d \mu .
\]
Thus, for $w = v - \grad \varphi_\e* \left(F' \circ (\varphi_\e * \mu) \mu \right) + F'(\varphi_\e * \mu)\grad \varphi_\e * \mu$, we have $\int \la w , \grad \psi \ra d \mu = 0$, i.e. $\grad \cdot (w \mu) = 0$ in the sense of distribution. By \cite[Proposition 8.4.3]{AGS}, since $v \in \Tan_\mu\P_2(\R^d)$ we get $\norm{v-w}_{L^2(\mu;\Rd)} \geq \norm{v}_{L^2(\mu;\Rd)}$. Since we have already shown that the vector in \eqref{subdiffform} is the element of minimal norm of $\partial \F_\e$, we get that $\norm{v-w}_{L^2(\mu;\Rd)} \leq \norm{v}_{L^2(\mu;\Rd)}$, and so $\norm{v-w}_{L^2(\mu;\Rd)} = \norm{v}_{L^2(\mu;\Rd)}$. Again using \cite[Proposition 8.4.3]{AGS}, we obtain $w=0$, which ends the proof.
\end{proof}

Finally, we prove the characterization of the subdifferential of the full regularized energies $\E^m_\e$.

\begin{proof}[Proof of Corollary \ref{full subdiff char}]
	Write $\lambda_V\in\R$ and $\lambda_W\in \R$ the semiconvexity constants of $V$ and $W$, respectively. The proof follows the same steps as that of Proposision \ref{subdiffchar} with the only difference being the definitions of the functions $G$, $f$ and $H$. Given $x,y \in \R^d$, we define, for all $\alpha \in [0,1]$,
$$
			G(\alpha) = F\left(\varphi_\e*\mu_\alpha((1-\alpha)x + \alpha y)\right) + V((1-\alpha)x + \alpha y) + \tfrac12 W*\mu_\alpha((1-\alpha)x + \alpha y),
$$
$$
			f(\alpha) = \frac{G(\alpha)-G(0)}{\alpha} - \frac{(\lambda+\lambda_W)\alpha}{2} \left(|x-y|^2 + W_2(\mu_0,\mu_1)\right) - \frac{\lambda_V \alpha}{2}|x-y|^2,
$$
and
$$
			H(\alpha) \!=\! F\!\left(\!\ird \varphi_\e(x-y + \alpha(\xi (x) - \xi (y))\,d\mu(y)\right)\! + V(x + \alpha \xi (y)) +\!\! \ird\!\! W(x-y + \alpha(\xi (x) - \xi (y)))\,d\mu(y),
$$
where $\mu_0$, $\mu_1$, $\lambda$ and $\xi $ are as in the proof of Proposition \ref{subdiffchar}.
\end{proof}

\section{Weak convergence of measures}
In this appendix, we recall several fundamental results on the weak convergence of measures. We begin with a result due to Ambrosio, Gigli, and Savar\'e on convergence of maps with respect to varying probability measures. This plays a key role in our proofs of both the $\Gamma$-convergence of the energies and the $\Gamma$ convergence of the gradient flows.

\begin{defi}[weak convergence with varying measures; c.f. {\cite[Definition 5.4.3]{AGS}}] \label{weakvaryingdef}
	Given a sequence $(\mu_n)_n \subset \P(\Rd)$ converging in the weak-$^*$ topology to some $\mu \in \P(\Rd)$, we say that a sequence $(v_n)_n$ with $v_n \in L^1(\mu_n;\Rd)$ for all $n \in\N$ \emph{converges weakly} to some $v \in L^1(\mu;\Rd)$ if 
\bes
	\lim_{n\to\infty} \ird f(x) v_n(x) \,d\mu_n(x) = \ird f(x) v(x) \,d\mu(x) \quad \mbox{for all $f \in C_\mt{c}^\infty(\R^d)$}.
\ees
Furthermore, we say that $(v_n)_n$ converges \emph{strongly} to $v$ in $L^p$, $p>1$, if 
\bes
	\limsup_{n\to\infty} \norm{v}_{L^p(\mu_n;\Rd)} \leq \norm{v}_{L^p(\mu;\Rd)}.
\ees
\end{defi}

\begin{prop}[properties of convergence with varying measures; {c.f. \cite[Theorem 5.4.4]{AGS}}] \label{AGSthm}
	Let $(\mu_n)_n \subset \P(\Rd)$, $\mu \in \P(\Rd)$ and $(v_n)_n$ be such that $v_n \in L^1(\mu_n;\Rd)$ for all $n\in\N$. Suppose $\mu_n \wsto \mu$ and $\sup_{n \in \mathbb{N}} \| v_n\|_{L^p(\mu_n;\Rd)} < \infty$ for some $p>1$. The following items hold.
\begin{enumerate}[(i)]
	\item There exists a subsequence of $(v_n)_n$ converging weakly to some $w \in L^1(\mu;\Rd)$. \label{weakcpt}
	\item\label{weaklsc} If $(v_n)_n$ weakly converges to some $v \in L^1(\mu;\Rd)$, then 
		\bes
			\liminf_{n \to \infty} \|v_n\|_{L^p(\mu_n;\Rd)} \geq \|v\|_{L^p(\mu;\Rd)} \quad \mbox{for all $p \geq 1$}.
		\ees
	\item \label{strongcty} If $(v_n)_n$ strongly converges in $L^p$ to some $v \in L^p(\mu;\Rd)$ and $\sup_{n \in \mathbb{N}} M_p(\mu_n) < \infty$, then
		\[ 
			\lim_{n \to \infty} \int f |v_n|^p d \mu_n = \int f |v|^p d \mu \quad \mbox{for all $f \in C_\mathrm{c}^\infty(\Rd)$}.
		\] 
\end{enumerate}
\end{prop}

We close by recalling a generalization of Fatou's lemma, for varying measures.

\begin{lem}[{Fatou's lemma for varying measures; see, e.g., \cite[Theorem 1.1]{feinberg2014fatou}, \cite[Lemma 3.3]{ambrosio2015bakry}}]\label{lem:fatou-varying}
Consisder a sequence $(\mu_n)_n \subset \P(\Rd)$ and $\mu \in \P(\Rd)$ so that $\mu_n \wsto \mu$. Then for any sequence $(f_n)_n$ of nonnegative functions on $\R^d$, we have
	\bes
		\ird \liminf_{n \to +\infty, x' \to x} f_n(x') \,d\mu(x) \leq \liminf_{n\to\infty} \ird f_n(x) \,d\mu_n(x).
	\ees
\end{lem}

\textbf{Acknowledgments:} The authors thank Andrew Bernoff, Andrea Bertozzi, Eric Carlen, Yanghong Huang, Inwon Kim, Dejan Slep\v cev, and Fangbo Zhang for many helpful discussions.

\bibliographystyle{abbrv}
\bibliography{HeightConstrainedAgg}

\end{document}